\documentclass[reqno,twoside,11pt]{amsart}

\usepackage{amsmath,amsfonts,calrsfs,fullpage,amssymb,xcolor,verbatim,eucal,yfonts,mathrsfs}

\usepackage{mathtools}
\mathtoolsset{showonlyrefs}

 \usepackage{xcolor}
\newtheorem{Theorem}{Theorem}[section]

\newtheorem{Proposition}[Theorem]{Proposition}
\newtheorem{Lemma}[Theorem]{Lemma}

\newtheorem{Remark}[Theorem]{Remark}
\newtheorem{Example}[Theorem]{Example}
\newtheorem{Hypothesis}{Hypothesis}
\makeatletter
\@addtoreset{equation}{section}

\makeatother

\def\a{\alpha}
\def\le{\left}
\def\r{\right}
\def\la{\lambda}
\def\La{\Lambda}

\def\ds{\displaystyle}

\def\e{\epsilon}

\newcommand{\hs}{\hspace{-2truecm}}
\newcommand{\hsl}{\hspace{-1truecm}}
\newcommand{\hsp}{\hspace{2truecm}}
\newcommand{\hslp}{\hspace{1truecm}}
\newcommand{\hsllp}{\hspace{0.5truecm}}

\title{\bf Smoluchowski-Kramers diffusion approximation for systems of stochastic damped wave equations with non-constant friction}\date{}

\author[S. Cerrai]{Sandra Cerrai}
\address{Department of Mathematics\\
University of Maryland\\
}
\email{cerrai@umd.edu}
\thanks{S. Cerrai was partially supported by the NSF grant  DMS-1954299 - {\em Multiscale analysis of infinite-dimensional stochastic systems}}

\author[A. Debussche]{Arnaud Debussche}
\address{Univ. de Rennes and IUF, CNRS, IRMAR - UMR 6625\\
}
\email{arnaud.debussche@ens-rennes.fr}
\thanks{A. Debussche benefits from the support of the French government “Investissements d’Avenir” program integrated to France 2030, bearing the following reference ANR-11-LABX-0020-01 and is partially funded by the ANR project ADA}

\subjclass[2010]{}


\begin{document}

 \begin{abstract}  
We consider  systems of  damped wave equations with a state-dependent damping coefficient and perturbed by a Gaussian multiplicative noise. Initially, we investigate their well-posedness, under quite general conditions on the friction. Subsequently, we study the validity of the so-called Smoluchowski-Kramers diffusion approximation. We show that, under more stringent conditions on the friction, in the small-mass limit the solution of the system of stochastic damped wave equations converge to the solution of a system of stochastic quasi-linear parabolic equations. In this convergence, an additional drift emerges as a result of the interaction between the noise and the state-dependent friction. The identification of this limit is achieved by using a suitable generalization  of the classical method of perturbed test functions, tailored to the current infinite dimensional setting.
 \end{abstract}

 \maketitle
 
 \tableofcontents

\section{Introduction}
\label{sec1}
In  this paper, we are dealing with the following system of stochastic non-linear wave equations, defined on a bounded domain $\mathcal{O}\subset \mathbb{R}^d$ having a smooth boundary and subject to Dirichlet boundary conditions
\begin{equation}\label{SPDE}
  \left\{\begin{array}{l}
\ds{\mu\,\partial_t^2 u_\mu(t) = \Delta u_\mu (t)-\gamma(u_\mu (t))\, \partial_t u_\mu (t) + f(u_\mu (t))+ \sigma(u_\mu(t))\,\partial_tw_t^Q(t),\ \ \ \ \ \ t\geq 0,}	\\[14pt]
\ds{u_\mu(t,x)=0,\ \ \ \ t\geq 0,\ \ \ \ \ x \in\,\partial\mathcal{O},}\\[14pt]
\ds{u_\mu(0,x)=u^\mu_0(x),\ \ \ \partial_t u_\mu(0,x)=v^\mu_0(x),\ \ \ \ x \in\,\mathcal{O}.}
\end{array}\right.
 \end{equation}
Here
$u_\mu(t,x) = (u_{1, \mu}(t,x),u_{2, \mu}(t,x),\cdots,u_{r, \mu}(t,x))$, for $(t,x) \in\,[0,+\infty)\times \mathcal{O}$, and for some $r\geq 2$. The stochastic perturbation $w^Q(t)$, $t\geq 0$, is given by a cylindrical Wiener process taking values in $\mathbb{R}^r$, which is white in time and colored in space, with reproducing kernel $H_Q$. The function $\gamma$ maps the space $H^1:=H^1_0(\mathcal{O};\mathbb{R}^r)$ into the space of bounded linear operators on $H:=L^2(\mathcal{O};\mathbb{R}^r)$, and satisfies suitable regularity properties, while $f:H\to H$ and $\sigma:H^1\to \mathcal{L}_2(H_Q,H^1)$ are some Lipschitz nonlinearities.

After proving that, for every fixed $\mu>0$ and every initial conditions $(u^\mu_0,v^\mu_0) \in\,\mathcal{H}_2:=H^2(\mathcal{O})\times H^1_0(\mathcal{O})$, system \eqref{SPDE} has a unique mild solution $\boldsymbol{u_\mu}=(u_\mu,\partial_t u_\mu) \in\,L^2(\Omega; C([0,T];\mathcal{H}_2))$, our purpose  will be studying the asymptotic behavior of the solution  $u_\mu$, as the mass parameter $\mu$ approaches zero. Namely, we will prove that as $\mu\downarrow 0$ the process $u_\mu$ converges, in an appropriate sense, to the solution of a  system of stochastic quasi-linear parabolic equations. Such limit is known in the literature as the Smoluchowski-Kramers diffusion-approximation.

\medskip

While this paper is the first one to address the case of systems of SPDEs having a state-dependent friction, a series of prior works have explored the validity of the Smoluchowski-Kramers approximation. In the context of finite-dimensional systems, there exists a substantial body of literature. A simplified model of the Smoluchowski-Kramers phenomenon for  stochastic differential equations was already investigated by Nelson in Chapters 9 and 10 of his  book \cite{Nelson_1967}, and, more recently, it was studied in  the following references: \cite{f}, \cite{fh}, \cite{hhv}, \cite{hmdvw}, \cite{nn-20},  \cite{spi} and \cite{XY22}. Additionally, in \cite{CF3}, \cite{CWZ}, \cite{fritz-gassiat}, and \cite{lee} an analogous problem was considered for systems subject to a magnetic field, and in \cite{hu}, \cite{Nguyen} and \cite{Shi}   related multi-scaling problems were investigated.

In recent years, there has been a considerable research activity related to the Smoluchowski-Kramers diffusion approximation for infinite-dimensional systems. The first results in this direction dealt with the case of  constant damping term, with smooth noise and regular coefficients (see \cite{CF1}, \cite{CF2}, \cite{salins}, and \cite{Lv2}). More recently, the case of constant friction has been studied in \cite{Fui} and \cite{zine} for equations perturbed by  space-time white noise in dimension $d=2$, and in \cite{Han} for equations with H\"older continuous coefficients in dimension $d=1$. In all these papers, the fact that the damping coefficient is constant leads to a perturbative result, in the sense that, in the small-mass limit, the solution $u_\mu$ of the stochastic damped wave equation converges to the solution of the stochastic parabolic problem obtained by taking $\mu=0$.
The case of SPDEs with 
state-dependent damping  was considered for the first time in \cite{CX} (see also \cite{CXIE}). Notably, this scenario differs from the previous one, as the non-constant friction leads to a noise-induced term in the small-mass limit. An analogous phenomenon was identified  
 in \cite{BC23}, where  the case of SPDEs constrained to live on the unitary sphere  of the space of square-integrable functions  $L^2(\mathcal{O})$ was considered. Actually, also in this case  the Smoluchowski-Kramers approximation leads to a stochastic parabolic problem, whose solution is constrained to live on the unitary $L^2(\mathcal{O})$-sphere and where an additional drift term appears. Somehow surprisingly, such extra drift does not account for the Stratonovich-to-It\^o correction term.   Finally, in  \cite{CS3}  systems subject to a magnetic field were studied and in this case the small-mass limit was obtained only after a suitable regularization of the problem and in the limit a  stochastic hyperbolic equation was obtained.  

The study of the Smoluchowski-Kramers approximation extends beyond just proving the limit of the solutions $u_\mu$. In several applications, it is important to understand the stability of this approximation with respect to other significant asymptotic features exhibited by the two systems.
To this end, in \cite{CGH},\cite{CF1}, and more recently in \cite{Nguyen2} and \cite{CXie}, it is demonstrated that the statistically invariant states of the stochastic damped wave equation (in cases of constant friction) converge in a suitable sense to the invariant measure of the limiting equation. In a similar vein, the papers \cite{sal}, \cite{sal2}, and \cite{CXIE} focus on analyzing the interplay between the small-mass and small-noise limits. In particular, \cite{CXIE} investigates the validity of a large deviation principle for the trajectories of the solution, while \cite{sal} and \cite{sal2} delve into the study of the convergence of the quasi-potential, which describes, as known, the asymptotics of the exit times of the solutions from a functional domain and the large deviation principle for the invariant measure.

\medskip

As we mentioned above, the present paper is the first one to deal with the study of the small-mass limit for systems of stochastic damped wave equations having a state-dependent friction coefficient. In \cite{CX}, the same problem was considered for one single equation.   
In that case,  it has been shown  that if $\gamma$ is in $C^1(\mathbb{R})$ and $0<\gamma_0\leq \gamma(s)\leq \gamma_1$, for every $s \in\,\mathbb{R}$, then for any initial condition $(u_0, v_0) \in\,H^1(\mathcal{O})\times L^2(\mathcal{O})$ and for any $\delta>0$ and $p<\infty$ 
\begin{equation}
\label{limitefinale}	
\lim_{\mu\to 0} \mathbb{P}\left(\Vert u_\mu-u\Vert_{C([0,T];H^{-\delta}(\mathcal{O}))}+ \Vert u_\mu-u\Vert_{L^p(\mathcal{O})}>\eta\right)=0,\ \ \ \ \eta>0, \end{equation}
where $u$ is the unique solution of the quasilinear stochastic parabolic equation 
\begin{equation}
\label{gs40}
\le\{\begin{array}{l}
\ds{\partial_t u= \frac{1}{\gamma(u)}\Delta u +\frac{f(u)}{\gamma(u)} -\frac{\gamma'(u)}{2\gamma^3(u)} \sum_{i=1}^\infty (\sigma(u)Qe_i)^2 +\frac{\sigma(u)}{\gamma(u)}\partial_t w^Q,\ \  t>0,\ \ \ \ x\in \mathcal{O},}\\[18pt]
\ds{u(0)=u_0, \ \ \ \ \ \ \  u_{|_{\partial\mathcal{O}}}=0.}
\end{array}\r.
\end{equation}
In particular,  it has been shown how, as a consequence of the interplay between the non constant damping and the noise, in the limit an extra drift term is created.

Being $\gamma$   a scalar function, for every function $u:[0,T]\times \mathcal{O}\to\mathbb{R}$ it was possible to write
\begin{equation}
\label{gammag}	
\gamma(u_\mu(t,x))\partial_t u_\mu(t,x)=\partial_t \rho_\mu(t,x),\ \ \ \ t \in\,[0,T],\ \ \ x \in\,\mathcal{O},\end{equation}
where $\rho_\mu=\lambda\circ u_\mu$ and $\lambda^\prime=\gamma$, and it was shown that for every $\mu>0$ the function $\rho_\mu$ solves the equation
\[ \begin{aligned}
 	\rho_\mu(t)+\mu \partial_t u_\mu(t) & =g(u_0)+\mu v_0 +\int_0^t \text{div}[b(\rho_\mu(s))\nabla \rho_\mu(s)] ds\\[10pt]
 	&\quad + \int_0^t f_\lambda(\rho_\mu(s)) ds+ \int_0^t \sigma_\lambda(\rho_\mu(s))dw^Q(s),
\end{aligned}
\]
where $b=1/\gamma \circ \lambda^{-1}$, $f_\lambda =f\circ \lambda^{-1}$, and $\sigma_\lambda(h)=\sigma(\lambda^{-1}\circ h)$.  It turned out that working with this equation, instead of \eqref{SPDE} (with $r=1$), made the proof of the small-mass limit  more direct. 
Actually, once tightness was proved, it was possible to get the weak convergence of  the sequence $\{\rho_\mu\}_{\mu \in\,(0,1)}$ to some $\rho$ that solves the quasilinear parabolic SPDE
\begin{equation}
\label{SPDERho-intro}
\le\{\begin{array}{l}
\ds{\partial_t \rho=\text{div}[b(\rho) \nabla \rho]+F(\rho)+\sigma_g(\rho)dw^Q(t), \ \ \ \  t>0, \ \ \ \ x\in \mathcal{O},}\\[10pt]
\ds{\rho(0,x)=g(u_0),\ \ \ \ \ \ \ \ \rho(t,x)=0,\ \ \ x\in \partial \mathcal{O}.}
\end{array}\r.
\end{equation}
Then, since it was possible to prove  pathwise uniqueness for equation \eqref{SPDERho-intro},   the convergence in probability followed from the weak convergence.
Finally, a generalized It\^o formula allowed  to get the convergence of $u_\mu$ to the solution of equation \eqref{gs40}, where the extra drift appears exactly as a consequence of the It\^o formula.
 
If the case of systems, in general there is no way to define a new function $\rho_\mu$ in such a way that \eqref{gammag} holds. Because of this all the arguments introduced in \cite{CX} and described above cannot be used, and the problem becomes much harder to study. Our aim in the present paper is finding an alternative approach to the proof of the validity of the Smoluchowski-Kramers approximation that applies to systems and provides a suitable generalization of \eqref{gs40} for the limiting  system of stochastic quasi-linear parabolic equations. 

Already the well-posedness of system \eqref{SPDE} for any fixed $\mu>0$ is not known in the existing literature and requires some work. We assume that $\gamma$ maps $H^1$ into $\mathcal{L}(H)$ and satisfies suitable conditions that are verified also  in the local case, when $d\leq 3$ and
\begin{equation}
\label{local-intro}	
[\gamma(h)k](x)=g(x,h(x))k(x),\ \ \ \ x \in\,\mathcal{O},\end{equation}
for some bounded measurable function $g:\mathcal{O}\times \mathbb{R}^r\to \mathbb{R}^{r\times r}$ such that $g(x,\cdot):\mathbb{R}^r\to \mathbb{R}^{r\times r}$ is Lipschitz continuous, uniformly with respect to $x \in\,\mathcal{O}$. By using a generalized splitting method that allows to handle the multiplicative noise, together with suitable localizations that allows to handle the local Lipschitz continuity of the coefficients, we prove that for every
 $(u^\mu_0, v^\mu_0) \in\,\mathcal{H}_2$ and  every $T>0$ and $p\geq 1$,  there exists a unique mild solution $u_\mu$ for equation \eqref{SPDE} such that $\boldsymbol{u_\mu} \in\,L^2(\Omega;C([0,T];\mathcal{H}_2))$.
 
 Once we have proven the well-posedness of system \eqref{SPDE}, we study the small-mass limit. The first fundamental step is proving a-priori bounds for $u_\mu$ and $\sqrt{\mu}\,\partial_t u_\mu$. Due to the nature of the problem, we need uniform bounds in  $\mathcal{H}_2$ and some other bounds in $\mathcal{H}_3:=H^3(\mathcal{O})\times H^2(\mathcal{O})$. The proofs of such bounds are quite challenging and intricate and require some more restrictive assumptions for the friction coefficient $\gamma$. In particular, we cannot consider  the local case \eqref{local-intro}, as we need to assume that
  \[[\gamma(h)k](x)=\mathfrak{g}(h)\,k(x),\ \ \ \ \ x \in\,\mathcal{O},\]
 for some differentiable mapping $\mathfrak{g}:H^1\to\mathbb{R}^{r\times r}$, which is bounded together with its derivative, and such that 
\[\inf_{h \in\,H^1}\,\langle\mathfrak{g}(h)\xi,\xi\rangle_{\mathbb{R}^{r}}\geq \gamma_0\,\Vert\xi\Vert_{\mathbb{R}^r}^2,\ \ \ \ \ \ \xi \in\,\mathbb{R}^r,\]
for some $\gamma_0>0$.

A fundamental consequence of such a-priori bounds is represented by  the tightness of the family $\{u_\mu\}_{\mu \in\,(0,1)}$ in  suitable functional spaces, which, in turn, allows to get the existence of a weak limit point for the family $\{u_\mu\}_{\mu \in\,(0,1)}$ in the same functional spaces. The final crucial step is identifying every possible limit point of $\{u_\mu\}_{\mu \in\,(0,1)}$  with the solution of the same appropriate system of stochastic parabolic equations. To this purpose, 
for every fixed $u, v\in\,H^1$ we introduce the problem
\begin{equation}
\label{sa40-intro}
dy(t)=-\mathfrak{g}(u)	y(t)\,dt+\sigma(u)dw^Q(t),\ \ \ \ \ y(0)=v,
\end{equation}
and we denote by $y^{u,v}$ its solution.
Moreover, we denote by $P^u_t$ the Markov transition semigroup associated with equation \eqref{sa40-intro}, and by $\nu^u$ its unique invariant measure.
This allows defining  
\[
S(u):=\int_{H^1}\left[D\mathfrak{g}^{-1}(u)	z\right]z\,d\nu^u(z),\ \ \ \ \ u \in\,H^1,\] and introducing the problem 
\begin{equation}
\label{lim-eq-intro}
\begin{cases}
\ds{\partial_t u(t,x)=\mathfrak{g}^{-1}(u(t))\Delta u(t,x)+\mathfrak{g}^{-1}(u(t))f(u(t,x))+S(u(t))+\mathfrak{g}^{-1}(u(t))\sigma(u(t))\partial_tw^Q(t),}\\[10pt]
\ds{u(0,x)=u_0(x),\ \ x \in\,\mathcal{O},\ \ \ \ \ \ \ \ u(t,x)=0,\ \ x \in \partial\mathcal{O}.}	
\end{cases}
	\end{equation}
	Our purpose is proving that if $(u_0^\mu,v_0^\mu) \in\,\mathcal{H}_3$, for every $\mu>0$, with \[\sup_{\mu \in\,(0,1)}\Vert(u^\mu_0,\sqrt{\mu}\,v^\mu_0)\Vert_{\mathcal{H}_1}<\infty,\ \ \ \ \ \lim_{\mu\to 0}	\mu^{\delta}\,\Vert(u^\mu_0,\sqrt{\mu}\,v^\mu_0)\Vert^2_{\mathcal{H}_2}=0,\ \ \ \ \  \lim_{\mu\to 0}	\mu^{1+\delta}\,\Vert(u^\mu_0,\sqrt{\mu}\,v^\mu_0)\Vert^2_{\mathcal{H}_3}=0,
\]
for some $\delta \in\,(0,1/2)$, then, for any  $u_0 \in\,H^1$  such that
\[
\lim_{\mu\to 0}\Vert u^\mu_0-u_0\Vert_{H^1}=0,	\]
 and for every $\varrho<1$, $\vartheta<2$,  $p<2/(\vartheta-1)$ and $\eta, T>0$ it holds	\begin{equation}
\label{sa161-intro}
\lim_{\mu\to 0}\,\mathbb{P}\left(\sup_{t \in \,[0,T]}\Vert u_\mu(t)-u(t)\Vert_{H^\varrho}+\int_0^T \Vert u_\mu(t)-u(t)\Vert_{H^\vartheta}^p\,dt>\eta\right)=0,	
\end{equation}
where $u \in\,L^2(\Omega;C([0,T];H^1)\cap L^2(0,T;H^2))$ is the  solution of \eqref{lim-eq-intro}.

	After showing that  $S:H^1\to H^1$ is well-defined and even  differentiable, we   prove the pathwise uniqueness for  system \eqref{lim-eq-intro} in $L^2(\Omega;C([0,T];H^1)\cap L^2(0,T;H^2)$.	Thus, if we can show that any weak limit point for the family $\{u_\mu\}_{\mu \in\,(0,1)}$ solves \eqref{lim-eq-intro} and  belongs to $L^2(\Omega;C([0,T];H^1)\cap L^2(0,T;H^2)$, we can conclude that its limit is uniquely identified as the solution of \eqref{lim-eq-intro} and \eqref{sa161-intro} holds.
	Our proof of the fact that  any weak limit point of $\{u_\mu\}_{\mu \in\,(0,1)}$ solves \eqref{lim-eq-intro} is based on a generalization of the classical perturbed test functions method first introduced in \cite{PSV} and extended in recent years to treat the case of several different types of infinite dimensional systems (see e.g. \cite{dBG}, \cite{DMV}, \cite{DP}, \cite{DV} and \cite{DV2}).
If we define $v_\mu(t):=\sqrt{\mu}\,\partial_t u_\mu(t)$, system \eqref{SPDE} can be written as
\begin{equation}
\label{system-corr}
\left\{\begin{array}{l}
\ds{\partial_t u_\mu(t)=\frac 1{\sqrt{\mu}}	v_\mu(t)}\\[10pt]
\ds{\partial_t v_\mu(t)=\frac 1{\sqrt{\mu}}\,\Delta u_\mu(t)-\frac 1{\mu}\gamma(u_\mu(t))\,v_\mu(t)+\frac 1{\sqrt{\mu}}\,f(u_\mu(t))+\frac 1{\sqrt{\mu}}\sigma(u_\mu(t))\partial_t w^Q(t),}
\end{array}
\right.	
\end{equation}
with the initial conditions $u_\mu(0)=u_0^\mu$ and $v_\mu(0)=\sqrt{\mu}\,v^\mu_0$. 
The Kolmogorov operator associated with the system above is given by
\begin{align*}
\mathcal{K}^{\mu} \varphi(u,v)&=\frac 1{\sqrt{\mu}}\left(\langle D_u\varphi(u,v),v\rangle_{H^1}+\langle D_v\varphi(u,v),\Delta u+f(u)\rangle_{H^1}	\right)\\[10pt]
&\hsllp-\frac 1\mu\langle D_v\varphi(u,v),\gamma(u)v\rangle_{H^1}+\frac 1{2\mu}\text{Tr}\left(D^2_v\varphi(u,v)[\sigma(u)Q][\sigma(u)Q]^\star\right).
\end{align*}
In particular, if we take
\[\varphi_{\mu}(u,v)=\langle u,h\rangle_H+\sqrt{\mu}\,\varphi_1(u,v)+\mu\,\varphi_{2}^{ \mu}(u,v),\ \ \ \ \ (u,v) \in\,H^1,\]
we have
\begin{align*}
\mathcal{K}^{\mu}& \varphi_{\mu}(u,v)=\frac 1{\sqrt{\mu}}\left(\mathcal{M}^{u}\varphi_1(u,v)+\langle h,v\rangle_{H^1}\right)+\mathcal{M}^{u}\varphi_{2}^{\mu}(u,v)+\langle D_u\varphi_1(u,v),v\rangle_{H^1}	\\[10pt]
&\hsl+\langle D_v\varphi_1(u,v),\Delta u+f(u)\rangle_{H^1}+\sqrt{\mu}\left(\langle D_v\varphi_{2}^{ \mu}(u,v),\Delta u+f(u)\rangle_{H^1}+\langle D_u\varphi_{2}^{\mu}(u,v),v\rangle_{H^1}\right),
\end{align*}
where $\mathcal{M}^{u}$ is the Kolmogorov  operator associated with equation \eqref{sa40-intro}. If we take $\varphi_1(u,v):=\langle \gamma^{-1}(u)v,h\rangle_{H}$ we have
\[\mathcal{M}^{u}\varphi_1(u,v)+\langle h,v\rangle_{H^1}=0.\]
Moreover, if we take some $\lambda(\mu)>0$ such that $\lambda(\mu)\downarrow 0$, as $\mu\downarrow 0$, and  define
\[\varphi_{2}^{\mu}(u,v):=\int_0^\infty e^{-\lambda(\mu) t}\left(P^{u}_t\psi(u,\cdot)(v)-\langle S(u),h\rangle_H\right)\,dt,\]
where 
$\psi(u,v):=\langle D_u\varphi_1(u,v),v\rangle_{H^1}$, we get
\[\mathcal{M}^{u} \varphi_{2}^{\mu}(u,v)=\lambda(\mu)\varphi_{2}^{ \mu}(u,v)-\left(\psi(u,v)-\langle S(u),h\rangle_H\right).\]
Therefore, if we apply the It\^o formula to $\varphi_\mu$ and $(u_\mu,v_\mu)$, we get
\begin{align*}   \begin{split}   \langle u_\mu(t)&,h\rangle_H=\langle u^\mu_0,h\rangle_H+\int_0^t\langle \gamma^{-1}(u_\mu(s))\Delta u_\mu(s)+\gamma^{-1}(u_\mu(s))\,f(u_\mu(s))+S(u_\mu(s)),h\rangle_{H}\,ds \\[10pt]
	&\hsp+\int_0^t\langle \gamma^{-1}(u_\mu(s))\sigma(u_\mu(s))\partial_tw^Q(s),h\rangle_{H}+\mathfrak{R}_{\mu}(t),\end{split}
	\end{align*}
where
\begin{align*}
	\mathfrak{R}_{\mu}&(t)=\sqrt{\mu}\,\langle \varphi_1(u^\mu_0,v^\mu_0)-\varphi_1(u_\mu(t),v_\mu(t)),h\rangle_H +\mu\langle \varphi_{2}^{ \mu}(u^\mu_0,v^\mu_0)-\varphi_{2}^{ \mu}(u_\mu(t),v_\mu(t)),h\rangle_H\\[10pt]
	&\hslp+\int_0^t\left(\lambda(\mu)\varphi_{2}^{ \mu}(u_\mu(s),v_\mu(s))+\mu\langle D_u\varphi_{2}^{ \mu}(u_\mu(s),v_\mu(s)),\partial_t u_\mu(s)\rangle_{H^1}\right)\,dt\\[10pt]
	&\hsp+\sqrt{\mu}\int_0^t\langle D_v\varphi_{2}^{ \mu}(u_\mu(s),v_\mu(s)),\Delta u_\mu(t)+f(u_\mu(s))\rangle_{H^1}\,ds\\[10pt]
	&\hsp \hslp+\sqrt{\mu}\int_0^t\langle D_v\varphi_{2}^{ \mu}(u_\mu(s),v_\mu(s)),\sigma(u_\mu(s))\partial_t w^Q(s)\rangle_H.
\end{align*}
Our goal will be proving that under suitable scaling conditions on $\lambda(\mu)$ the reminder $\mathfrak{R}_{\mu}$ converges to zero in $L^1(\Omega;C([0,t]))$, as $\mu\downarrow 0$. This will allow us to obtain that any weak limit of the family $\{u_\mu\}_{\mu \in\,(0,1)}$ converges to the solution of \eqref{lim-eq-intro}. Notice that this task is far from trivial. Actually, first it requires to prove that the two functions $\varphi_1$ and $\varphi_2^\mu$ are sufficiently smooth. Then it requires  to control the dependence of $\varphi_1$ and $\varphi_2^\mu$ and  their derivatives on the parameter $\mu$ and to show that the bounds we find match appropriately with the a priori bounds we have previously found for $u_\mu$ and $\partial_t u_\mu$. Also the reminder depends on derivatives of $u_\mu$ up to order $2$. Since, $D_v\varphi_2^\mu$ acts on $H^1$ we see that we need bounds on $u_\mu$ in $H^3$. As it turns out, we will not be able to prove all these results for system \eqref{system-corr} and we will need to introduce a suitable truncated system.  The arguments we have described above  will allow us to prove \eqref{sa161-intro} for the truncated problem and the truncated limiting problem, and only at that point suitable a-priori bounds for the truncated limiting problem will allow to obtain \eqref{sa161-intro} for the original system.

\section{Notations and assumptions}
\label{sec2}
Let $\mathcal{O}$ be a bounded domain in $\mathbb{R}^d$, with $d\geq 1$. In what follows, we shall denote by $H$ the Hilbert space $L^2(\mathcal{O};\mathbb{R}^r)$, endowed with the scalar product
\[\langle (h_1,\ldots,h_r),(k_1,\ldots,k_r)\rangle_H=\sum_{i=1}^r \int_{\mathcal{O}} h_i(x) k_i(x)\,dx,\]
and the corresponding norm $\|\cdot\|_H$. Moreover, for every $p\geq 1$, we will denote by $L^p$ the space $L^p(\mathcal{O};\mathbb{R}^r)$.

If we denote by $(\hat{L}, D(\hat{L}))$ the realization of the Laplace operator in $L^2(\mathcal{O})$, endowed with Dirichlet boundary conditions, and  define
\[L h:=(\hat{L}h_1,\ldots,\hat{L} h_r),\ \ \ \ \ h \in\,D(L)=D(\hat{L})\times \cdots\times D(\hat{L}),\]
it is possible to show that there exists a complete orthonormal system $\{e_i\}_{i \in\,\mathbb{N}}$ in $H$ and a non-decreasing sequence of positive real numbers $\{\a_i\}_{i \in\,\mathbb{N}}$ such that
\[L e_i=-\a_i e_i,\ \ \ \ i \in\,\mathbb{N}.\]

For every $\delta \in\,\mathbb{R}$, we define $H^\delta$ to be the completion of $C^\infty_0(\mathcal{O};\mathbb{R}^r)$ with respect to the norm
\[\|h\|_{H^\delta}=\sum_{i=1}^\infty \alpha_i^\delta \langle h,e_i\rangle_H^2,\]
and we define
\[\mathcal{H}_\delta:=H^\delta\times H^{\delta-1}.\]
In the case $\delta=0$, we simply set $\mathcal{H}:=\mathcal{H}_0$.

\medskip

The cylindrical Wiener process $w^Q$ is defined as the formal sum
\[w^Q(t)=\sum_{i=1}^\infty Q e_i \beta_i(t),\ \ \ \ t\geq 0,\]
where $Q=(Q_1,\ldots,Q_r)$ is a bounded linear operator in $\mathcal{L}(H)$, $\{\beta_i\}_{i \in\,\mathbb{N}}$ is a sequence of independent standard Brownian motions defined on some filtered probability space $(\Omega,\mathcal{F}, \{\mathcal{F}\}_{t\geq 0}, \mathbb{P})$ and $\{e_i\}_{i \in\,\mathbb{N}}$ is the orthonormal basis of $H$ introduced above.

In what follows we shall denote by $H_Q$ the set $Q(H)$. $H_Q$ is the so-called {\em Reproducing Kernel} of the noise $w^Q$ and is a Hilbert space, endowed with the scalar product
\[\langle h,k\rangle_{H_Q}=\langle Q^{-1} h,Q^{-1}k\rangle_H.\]
In particular, the sequence $\{Q e_i\}_{i \in\,\mathbb{N}}$ is a complete orthonormal system for $H_Q$.
\smallskip

As for the coefficient $\gamma$, we assume the following conditions.

\begin{Hypothesis}\label{as1}
\begin{enumerate}
\item[1.] The function $\gamma$ maps $H^1$ into $\mathcal{L}(H)$ with 
\begin{equation}
\label{gsb15}
   \sup_{h \in\,H^1}\, \Vert\gamma(h)\Vert_{\mathcal{L}(H)} <+\infty.
\end{equation}
and \begin{equation}
\label{gsb102}
\Vert \left[\gamma(h_1)-\gamma(h_2)\right]k\Vert_{H}\leq c\,\Vert h_1-h_2\Vert_{H^1}\Vert k\Vert_{H^1},	
\end{equation}
for every $h_1, h_2, k, \in\,H^1$.

\item[2.] The function $\gamma$ maps $H^2$ into $\mathcal{L}(H^1)$ and for every $h \in\,H^2$ and $k \in\,H^1$
\begin{equation}
\label{gsb76-bis}
\Vert \gamma(h)k
\Vert_{H^1}\leq c\,\Vert k\Vert_{H^1}+c\,\Vert h\Vert_{H^2}^{3/4}\Vert h\Vert_{H^1}^{1/4}\Vert k\Vert_{H^1}^{3/4}\Vert k\Vert_{H}^{1/4}.
\end{equation}
\end{enumerate}
\end{Hypothesis}

Notice that as a consequence of \eqref{gsb76-bis} we have 
\begin{equation}
\label{gsb76}
\Vert \gamma(h)\Vert_{\mathcal{L}(H^1)}\leq c\left(1+\Vert h\Vert_{H^2}\right),\ \ \ \ \ h \in\,H^2.	
\end{equation}

\begin{Example}
\label{Example-1}
{\em For every $h, k:\mathcal{O}\to \mathbb{R}^r$, we define
\[[\gamma(h)k](x)=g(x,h(x))k(x),\ \ \ \ x \in\,\mathcal{O},\]
for some bounded measurable function $g:\mathcal{O}\times \mathbb{R}^r\to \mathbb{R}^{r\times r}$. It is immediate to check that $\gamma$ maps $H^1$ into $\mathcal{L}(H)$ and \eqref{gsb15} holds. Moreover, if we assume that
$g(x,\cdot):\mathbb{R}^r\to\mathbb{R}^{r\times r}$ is Lipschitz continuous, with
\[ [g]_{\text{\tiny{Lip}}}:=\sup_{x \in\,\mathcal{O}}\,\sup_{\xi\neq \eta}\frac{\Vert g(x,\xi)-g(x,\eta)\Vert_{\mathbb{R}^{r\times r}}}{\Vert \xi-\eta\Vert_{\mathbb{R}^r}}<+\infty,\]
 then, for every $h_1, h_2, k \in\,H^1$ we have
\[\Vert \left[\gamma(h_1)-\gamma(h_2)\right]k\Vert_H\leq \Vert g(\cdot,h_1)-g(\cdot,h_2)\Vert_{L^4}\Vert k\Vert_{L^4}\leq [g]_{\text{\tiny{Lip}}}\Vert h_1-h_2\Vert_{L^4}\Vert k\Vert_{L^4}.\]
If we assume that  $d\leq 4$ we have that $H^1\hookrightarrow L^4$ and then  we get \eqref{gsb102}.

Finally, if $g:\mathcal{O}\times \mathbb{R}^r\to \mathbb{R}^{r\times r}$ is differentiable, with bounded derivative,
for every $h \in\,H^2$ and $k \in\,H^1$, we have
\begin{equation}
\label{n40}	
\Vert \gamma(h)k\Vert_{H^1}\leq \Vert Dg(\cdot,h)\nabla h \,k\Vert_H+\Vert g(\cdot,h)\,\nabla k\Vert_H\leq \Vert Dg\Vert_{\infty}\,\Vert \nabla h\Vert_{L^4}\,\Vert k\Vert_{L^4}+c\,\Vert k\Vert_{H^1}.
\end{equation}
If we assume  that $d\leq 3$, we have that $H^1 \hookrightarrow L^6$ and by interpolation for every $l \in\,H^1$ we have
\[\Vert l\Vert_{L^4}\leq \Vert l\Vert_{L^6}^{3/4}\Vert l\Vert_{H}^{1/4}\leq \Vert l\Vert_{H^1}^{3/4}\Vert l\Vert_{H}^{1/4}.\]
Thus, from \eqref{n40} we get
\[\Vert \gamma(h)k\Vert_{H^1}\leq c\,\Vert h\Vert_{H^2}^{3/4}\Vert h\Vert_{H^1}^{1/4}\Vert k\Vert_{H^1}^{3/4}\Vert k\Vert_{H}^{1/4}+c\,\Vert k\Vert_{H^1},\]
and \eqref{gsb76-bis} follows.  

}
\end{Example}

\begin{Example}
\label{Example-2}
{\em For every $h \in\,H^1$ and $k \in\,H$, we define
\[[\gamma(h)k](x)=\mathfrak{g}(h)\,k(x),\ \ \ \ x \in\,\mathcal{O},\]
for some  bounded and Lipschitz function $\mathfrak{g}:H^1\to \mathbb{R}^{r\times r}$. It is easy to check that   $\gamma$ satisfies \eqref{gsb15}, \eqref{gsb102} and \eqref{gsb76-bis}. In fact, in this case \eqref{gsb102} and \eqref{gsb76-bis} are considerably improved, as we have for every $s \in\,\mathbb{R}$
\begin{equation}
\label{xfine109}
\Vert \left[\gamma(h_1)-\gamma(h_2)\right]k\Vert_{H^s}\leq c\,\Vert h_1-h_2\Vert_{H^1}\Vert k\Vert_{H^s}	
\end{equation}
and 
\begin{equation}
\label{xfine110}
\Vert \gamma(h)k
\Vert_{H^s}\leq c\,\Vert k\Vert_{H^s}.
\end{equation}

}	
\end{Example}

As far as  the non-linearity $f$ is concerned, we make the following assumptions.  
\begin{Hypothesis}\label{as3}
The function $f:H\to H$ is Lipschitz continuous. Moreover, $f$ maps $H^1$ into itself with
\begin{equation}
\label{n30}	
\Vert f(h)\Vert_{H^1}\leq c\,(\Vert h\Vert_{H^1}+1),\ \ \ \ \ h \in\,H^1,
\end{equation}
and
\begin{equation}
\label{xfine	100}
\Vert f(h_1)-f(h_2)\Vert_{H^1}\leq c\left(1+\Vert h_1\Vert_{H^2}\right)\,\Vert h_1-h_2\Vert_{H^1}.
\end{equation}

\end{Hypothesis}

\begin{Remark}\label{remx2.3}
{\em Assume that
\[f(h)(x)=\mathfrak{f}(h(x)),\ \ \ \ x \in\,\mathcal{O},\]
for some Lipschitz continuous function $\mathfrak{f}:\mathbb{R}^r\to \mathbb{R}^r$. Then $f:H\to H$ is Lipschitz continuous and \eqref{n30} holds. Moreover, if $f$ is differentiable and its derivative is Lipschitz continuous and bounded, we have 
\begin{align*}
\Vert f(h_1)-f(h_2)&\Vert_{H^1}\leq \Vert \left(Df(h_1)-D f(h_2)\right)\nabla h_1\Vert_{H}+\Vert D f(h_2)\nabla (h_1-h_1)\Vert_{H}\\[10pt]
&	\leq [Df]_{\text{\tiny{Lip}}}\Vert h_1-h_2\Vert_{L^4}\Vert \nabla h_1\Vert_{L^4}+\Vert Df\Vert_\infty \Vert h_1-h_2\Vert_{H^1},
\end{align*}
and if $d\leq 4$, this implies \eqref{xfine	100}.}	\flushright{$\Box$}
\end{Remark}

Next, we assume the following conditions for the diffusion coefficient.
\begin{Hypothesis}\label{as2}
\begin{enumerate}
\item[1.] The function $\sigma$ maps $H^1$ into $\mathcal{L}_2(H_Q,H^1)$, and
\begin{equation}
\label{gsb43}
\Vert \sigma(h)\Vert_{\mathcal{L}_2(H_Q,H^1)}\leq c\,\left(1+\Vert h\Vert_{H^1}\right),\ \ \ \ h \in\,H^1.
\end{equation}

\item[2.] 
It holds
\begin{equation}\label{gsb40}
\sup_{h \in\,H^1}\|\sigma(h)\|_{\mathcal{L}_2(H_Q,H)}<\infty,
\end{equation}
and for every $h_1, h_2 \in\,H^1$
\begin{equation}
\label{gsb7}
	\Vert \sigma(h_1)-\sigma(h_2)\Vert_{\mathcal{L}_2(H_Q,H)}^2\leq c\,\Vert h_1-h_2\Vert_{H^1}^2.
\end{equation}

\item[3.] For every $h \in\,H^2$ and $k \in\,H^1$ we have
\begin{equation}
\label{gsb43-tris}
\Vert \sigma(h+k)-\sigma(h)\Vert_{\mathcal{L}_2(H_Q,H^1)}\leq c\,\Vert k\Vert_{H^1}\,\left(1+\Vert h\Vert_{H^2}\right).	
\end{equation}
 \end{enumerate}
\end{Hypothesis}

\begin{Remark}
\label{rem3.4}
{\em \begin{enumerate}

\item[1.] 
Assume that there exist some measurable mappings $\sigma_i:\mathcal{O}\times \mathbb{R}^r\to \mathbb{R}^r$ such that for every $h \in\,H$ and $i \in\,\mathbb{N}$
\[\left[\sigma(h)Q e_i\right](x)=\sigma_i(x,h(x)),\ \ \ \ x \in\,\mathcal{O}.\]	
Then, if 
\[\sup_{h \in\,H^1}\,\sum_{i=1}^\infty \int_\mathcal{O}\vert\sigma_i(x,h(x))\vert^2\,dx<\infty, \]
\eqref{gsb40} follows. Moreover, if there exists $c>0$ such that
\[	
\sup_{x \in\,\mathcal{O}}\sum_{i=1}^\infty |\sigma_i(x,\xi)-\sigma_i(x,\eta)|^2_{\mathbb{R}^r}\leq c\,|\xi-\eta|^2_{\mathbb{R}^r},\ \ \ \ \xi,\ \eta \in\,\mathbb{R}^r,
\]
then \eqref{gsb7} holds.

\item[2.] If $\sigma$ is constant,  then \eqref{gsb40} means that  $\sigma Q$ is a Hilbert-Schmidt operator in $H$. Equivalently, when $\sigma$ is the identity operator, \eqref{gsb40} means that the noise $w^Q$ lives in $H$, so that 
\begin{equation*}
    w^Q\in C([0,T];H).
\end{equation*}

\item[3.] If $\sigma$ is not constant, then \eqref{gsb40} and \eqref{gsb7} are satisfied when, for example,
\begin{equation}
\label{gsb115}
\sigma_i(x,\xi)=\la(\xi)Qe_i(x),\ \ \ \ (x,\xi) \in\,\mathcal{O}\times \mathbb{R}^r, \ \ \ \ i \in\,\mathbb{N},	
\end{equation}
 for some   $\la:\mathbb{R}^r\to \mathbb{R}^{r\times r}$ bounded and Lipschitz continuous and for some $Q \in\,\mathcal{L}(H)$ such that
\begin{equation}
\label{gsb113}
\sum_{i=1}^\infty \Vert Qe_i\Vert^2_{L^\infty(\mathcal{O})}<\infty.	
\end{equation}
In case $Q$ is diagonalizable with respect the basis $\{e_i\}_{i \in\,\mathbb{N}}$, with $Q e_i=\vartheta_i e_i$,
the condition above reads
\begin{equation}\label{gx005}
    \sum_{i=1}^\infty \vartheta_i^2\Vert e_i\Vert^2_{L^\infty(\mathcal{O})}<\infty.
\end{equation}
Under reasonable assumptions on the domain $\mathcal{O}$, it holds that 
\[\Vert e_i\Vert_{L^\infty(\mathcal{O})} \leq c\, i^\alpha\]
for some $\alpha>0$ depending on $d$, and \eqref{gx005} becomes
\[\sum_{i=1}^\infty \vartheta_i^2 \, i^{2\alpha}<\infty.\]
In particular, when $d=1$ or the domain is a hyper-rectangle and $d\geq 1$, the eigenfunctions $\{e_i\}_{i\in\mathbb{N}}$ are equi-bounded and \eqref{gx005} becomes $\sum_{i=1}^\infty \vartheta_i^2 \, <\infty$.
\item[4.] Concerning condition \eqref{gsb43},  assume as above that all mappings $\sigma_i$ are given by \eqref{gsb115}. Then,  if we assume that  $\lambda$ is differentiable, with bounded derivative, for every $h \in\,H^2$ we have
 \[\begin{array}{l}
\ds{\Vert \sigma_i(\cdot,h)\Vert_{H^1}\leq \Vert D\lambda(h)\nabla h\, Q e_i\Vert_H+\Vert \lambda(h)\nabla (Qe_i)\Vert_H
}\\[14pt]
\ds{\leq \Vert D\lambda\Vert_\infty \Vert \nabla h\Vert_{H}\Vert Q e_i\Vert_{L^\infty}+	\Vert \lambda \Vert_\infty \Vert \nabla (Qe_i)\Vert_H.}
\end{array}\]
Hence, if we assume \eqref{gsb113} and
\begin{equation}
\label{gsb114}
\sum_{i=1}^\infty\Vert Q e_i\Vert_{H^1}^2<\infty,	
\end{equation}
we conclude that
\[\Vert \sigma(h)\Vert_{\mathcal{L}_2(H_Q,H^1)}^2=\sum_{i=1}^\infty \Vert \sigma_i(\cdot,h)\Vert_{H^1}^2\leq c\, \left(1+\Vert h\Vert^2_{H^1}\right).\]
Notice that, if we assume as above that 
$Q$ is diagonalizable with respect the basis $\{e_i\}_{i \in\,\mathbb{N}}$, with 
$Q e_i=\vartheta_i e_i$, then \eqref{gsb114} follows once we assume
\[\sum_{i=1}^\infty \vartheta_i^2 \alpha_i<\infty.\]
\item[5.] If  we assume that $\lambda$ is twice differentiable, with bounded derivatives, and take $d\leq 4$, for every  $h \in\,H^2$ and $k \in\,H^1$ we have
\begin{align*}
&\hslp\Vert \sigma_i(\cdot,h+k)-\sigma_i(\cdot,h)\Vert_{H^1}\leq \Vert D\lambda(h+k)\,\nabla k\,Q e_i\Vert_{H}\\[10pt]
 &\hsp+\Vert \left(D\la(h+k)-D\la(h)\right)\nabla h\, Q e_i\Vert_{H}+ \Vert \left(\la(h+k)-\la(h)\right)\nabla (Q e_i)\Vert_H	\\[10pt]
&\hslp\leq \Vert D\lambda\Vert_\infty \Vert Q e_i\Vert_{L^\infty} \Vert k\Vert_{H^1}+[ D\lambda]_{\text{\tiny{Lip}}}\Vert k\Vert_{H^1}\,\Vert h\Vert_{H^2}\Vert Q e_i\Vert_{L^\infty}+[ \lambda]_{\text{\tiny{Lip}}} \Vert k\Vert_{H^1}\Vert \nabla (Q e_i)\Vert_{L^4}.
\end{align*}
This means that if we assume  \eqref{gsb113}, \eqref{gsb114}, and
\begin{equation}
\label{fx1}	
\sum_{i=1}^\infty \Vert \nabla (Q e_i)\Vert_{L^4}^2<\infty,
\end{equation}
we can conclude that  \eqref{gsb43-tris} holds.
If we assume, as above, $Q e_i=\vartheta_i e_i$, since $H^{d/6}\hookrightarrow L^3$, then \eqref{fx1} holds if
\[\sum_{i=1}^\infty \vartheta_i^2 \alpha_i^{1+\frac{d}{6}}<\infty.\]

\end{enumerate}

	}
	\flushright{$\Box$}

\medskip

\end{Remark}

In the proof of the small-mass limit, we will need to assume some further conditions on the coefficients $\gamma$, $f$ and $\sigma$.
\begin{Hypothesis}
\label{as4}
\begin{enumerate}
\item[1.]	There exists a differentiable mapping $\mathfrak{g}:H^1\to\mathbb{R}^{r\times r}$, which is bounded together with its derivative, such that for every $h \in\,H^1$ and $k \in\,H$
\[[\gamma(h)k](x)=\mathfrak{g}(h)\,k(x),\ \ \ \ \ x \in\,\mathcal{O}.\]
Moreover there exists $\gamma_0>0$ such that
\begin{equation}
\label{gsb15-bis}
\inf_{h \in\,H^1}\,\langle\mathfrak{g}(h)\xi,\xi\rangle_{\mathbb{R}^{r}}\geq \gamma_0\,\Vert\xi\Vert_{\mathbb{R}^r}^2,\ \ \ \ \ \ \xi \in\,\mathbb{R}^r.
\end{equation}
\item[2.]  The function $f$ maps $H^2$ into itself and
 	\begin{equation}
 	\label{fx20}
 	\Vert f(u)\Vert_{H^2}\leq c\left(1+\Vert u\Vert^2_{H^2}\right),\ \ \ \ \ \ u \in\,H^2.	
 	\end{equation}
 	\item[3.] The function $\sigma$ maps $H^2$ into $\mathcal{L}_2(H_Q,H^2)$ and there exists $\bar{\kappa}<2$ such that
 	\begin{equation}
\label{xfine113}
\Vert \sigma(h)\Vert_{\mathcal{L}_2(H_Q,H^2)}\leq c\left(1+\Vert h\Vert_{H^2}+\Vert h\Vert_{H^{\bar{\kappa}}}^2\right).	
\end{equation}
\item[4.] If $\sigma:H^1\to\mathcal{L}_2(H_Q,H^1)$ is not bounded, then 
\smallskip
\begin{enumerate}
\item[a)] $\sigma$ maps $H^\rho$ into $\mathcal{L}_2(H_Q,H^\rho)$ and $f$ maps $H^\rho$ into $H^\rho$, for every $\rho<1$, with
\begin{equation}
\label{xfine204}
\Vert \sigma(u)\Vert_{	\mathcal{L}_2(H_Q,H^\rho)}\leq c_\rho\left(1+\Vert u\Vert_{H^\rho}\right),\ \ \ \ \Vert f(u)\Vert_{H^\rho}\leq c_\rho\left(1+\Vert u\Vert_{H^\rho}\right),\end{equation}
\item[b)] there exists $\bar{s}<1$ such that $\mathfrak{g}:H^{\bar{s}}\to \mathbb{R}^{r\times r}$ is differentiable, with
\begin{equation}
\label{gsb15-tris}
\sup_{h \in\,H^1}\Vert D\mathfrak{g}(h)\Vert_{\mathcal{L}(H^{\bar{s}},	\mathbb{R}^{r\times r})}<+\infty,\ \ \ \ \ \inf_{h \in\,H^{\bar{s}}}\,\langle\mathfrak{g}(h)\xi,\xi\rangle_{\mathbb{R}^{r}}\geq \gamma_0\,\Vert\xi\Vert_{\mathbb{R}^r}^2,\ \ \ \xi \in\,\mathbb{R}^r, 
\end{equation}
\item[c)] there exists $\bar{\kappa}<2$ such that
\begin{equation}
\label{xfine113-f}
\Vert f(h)\Vert_{H^2}\leq c\left(1+\Vert h\Vert_{H^2}+\Vert h\Vert_{H^{\bar{\kappa}}}^2\right).	
\end{equation}
\end{enumerate}

\end{enumerate}
 
\end{Hypothesis}

\begin{Remark}
\label{rem-2}
{\em \begin{enumerate}
\item[1. ] There is no loss of generality in assuming that the constant $\bar{\kappa}<2$ appearing in condition 3 and  condition 4.c  in Hypothesis \ref{as4} is  the same.
 \item[2.]	As we have seen in Example \ref{Example-2},  if $\gamma$ satisfies  condition 1. in Hypothesis \ref{as4}, then it satisfies conditions \eqref{gsb15} and \eqref{gsb76-bis}  in Hypothesis \ref{as1}, even in the stronger versions \eqref{xfine109} and \eqref{xfine110}. Moreover, the fact that $\mathfrak{g}:H^1\to \mathbb{R}^{r\times r}$ is differentiable with bounded derivative, implies that condition \eqref{gsb102} is satisfied as well, even in its stronger version \eqref{xfine109}. 

\item[3.] Condition 2. on $f$ is satisfied in the case of the example described in Remark \ref{remx2.3}, if we assume that $\mathfrak{f}$ is twice differentiable, with bounded derivatives. 

\item[4.] If for some $s\leq 1$, the matrix $\mathfrak{g}(h)$ is invertible for every $h \in\,H^s$ and  $\mathfrak{g}:H^s\to \mathbb{R}^{r\times r}$ is differentiable, then we have
 	\[D\mathfrak{g}^{-1}(h)k=-\mathfrak{g}^{-1}(h)[D\mathfrak{g}(h)k]\mathfrak{g}^{-1}(h),\ \ \ \ h, k \in\,H^s.\]
 	In particular,  we have
 	\begin{equation}
 	\label{sa51}
 	\sup_{h \in\,H^s}\Vert D\gamma^{-1}(h)k\Vert_{\mathbb{R}^{r\times r}}\leq c\,\Vert k\Vert_{H^s},\ \ \ \ \ k \in\,H^s.	
 	\end{equation}
 	
\item[5.] Assume that, as in Remark \ref{rem3.4},
\begin{equation} \label{xfine203}[\sigma(u)h](x)=\lambda(u(x))h(x),\ \ \ \ x \in\,\mathcal{O},\end{equation}	
for some bounded $\lambda:\mathbb{R}^r\to\mathbb{R}^{r\times r}$, which is twice differentiable, with bounded derivatives. If we assume \eqref{gsb113}, \eqref{gsb114} and \eqref{fx1}, then, when $d\leq 3$ we have 
\begin{equation} \label{xsa8}\Vert \sigma(u)\Vert_{\mathcal{L}_2(H_Q,H^2)}\leq c\left(1+\Vert u\Vert_{H^2}+\Vert \nabla u\Vert_{L^4}^2\right)\leq c\left(1+\Vert u\Vert_{H^2}+\Vert u\Vert_{H^{7/4}}^2\right),\end{equation} 
so that \eqref{xfine113} holds with $\bar{\kappa}=7/4$.

\item[6.] By interpolation, for every $r<\bar{\kappa}$ we have 
\[\Vert h\Vert_{H^{\bar{\kappa}}}\leq \Vert h\Vert^{\frac{\bar{\kappa}-r}{3-r}}_{H^3}\Vert h\Vert_{H^r}^{\frac{3-\bar{\kappa}}{3-r}}.\]
Then, thanks to \eqref{xfine113},  for every $R\geq 1$ we get \begin{equation}
\label{xfine120}
\Vert h\Vert_{H^r}\leq R\Longrightarrow \Vert \sigma(h)\Vert_{\mathcal{L}_2(H_Q,H^2)}\leq c_{R,r}\left(1+\Vert h\Vert_{H^2}+\Vert h\Vert_{H^3}^{\frac{2(\bar{\kappa}-r)}{3-r}}\right).	
\end{equation}
\item[7.] With similar arguments, we have that if $f$ is given as in Remark \ref{remx2.3}, then it satisfies \eqref{xfine113-f}. In particular, we have 
\begin{equation}
\label{xfine120-f}
\Vert h\Vert_{H^r}\leq R\Longrightarrow \Vert f(h)\Vert_{H^2}\leq c_{R,r}\left(1+\Vert h\Vert_{H^2}+\Vert h\Vert_{H^3}^{\frac{2(\bar{\kappa }-r)}{3-r}}\right).	
\end{equation} 
\item[8.] Thanks to \eqref{gsb15-bis} we have that for every $h \in\,H^1$ the matrix $\mathfrak{g}(h)$ is invertible and 
\begin{equation}
	\label{xfine160}
	\inf_{h \in\,H^1}\langle \mathfrak{g}^{-1}(h)\xi,\xi\rangle_{\mathbb{R}^r}\geq \tilde{\gamma_0}\,\Vert \xi\Vert^2_{\mathbb{R}^r},\ \ \ \ \ \xi \in\,\mathbb{R}^r.
\end{equation}
\item[9.] As shown in \cite[Theorem 1, Section 5.3.6]{RS}, if $\lambda:\mathbb{R}^r\to\mathbb{R}^r$ is differentiable, with bounded derivative, then for any $0<\rho<1$
\[\Vert\lambda(u)\Vert_{H^\rho}\leq c\,\Vert D\lambda\Vert_{L^\infty}\,\Vert u\Vert_{H^\rho},\]
so that $f$ satisfies \eqref{xfine204} if  it is given as in Remark \ref{rem3.4}, for some $\mathfrak{f}$ which is differentiable with bounded derivative. Moreover, as proven in \cite[Theorem 2, section 4.6.4.]{RS}, we have
\[\Vert \lambda(u)Qe_i\Vert_{H^\rho}\leq \Vert \lambda(u)\Vert_{L^\infty}\Vert Q e_i\Vert_{H^\rho}+\Vert \lambda(u)\Vert_{H^\rho}\Vert Q e_i\Vert_{L^\infty}.\]
hence, if in addition to \eqref{gsb113} we assume
\[\sum_{i=1}^\infty \Vert Q e_i\Vert^2_{H^\rho}<\infty,\]
then $\sigma$ satisfies \eqref{xfine204}.
 \end{enumerate}
\flushright{$\Box$}	

}

\medskip

\end{Remark}
 
This last set of conditions will be required when we introduce the noise-induced drift that appears in the limiting equation, and we need to  prove its well-posedness and identify the limit.  
 \begin{Hypothesis}
 \label{as5} 
 \begin{enumerate}
 \item[1.] 	The mapping $\sigma:H^2\to\mathcal{L}_2(H_Q,H^1)$ is differentiable and
 \begin{equation}  \label{gsb43-bis}\Vert D\sigma(u)h\Vert_{	\mathcal{L}_2(H_Q,H^1)}\leq c\,\left(1+\Vert u\Vert_{H^2}\right)\,\Vert  h\Vert_{H^1},\ \ \ \ \ \ u, h \in\,H^2.\end{equation}
 \item[2.] The mapping $\mathfrak{g}:H^1\to\mathbb{R}^{r\times r}$ is  twice continuously differentiable, with
 \begin{equation}  \label{na-bis}\sup_{u \in\,H^1}\Vert D^2\mathfrak{g}(u)\Vert_{\mathcal{L}(H^1\times H^1;{R}^{r\times r})}<\infty.\end{equation}
 \end{enumerate}

 \end{Hypothesis}

\begin{Remark}
	{\em Assume that, as in Remark \ref{rem3.4},
\[[\sigma(u)h](x)=\lambda(u(x))h(x),\ \ \ \ x \in\,\mathcal{O},\]	
for some bounded $\lambda:\mathbb{R}^r\to\mathbb{R}^{r\times r}$. 
In this case, if $\lambda$ is differentiable, we have
\[	[(D\sigma(u)k)Qe_i](x)=[D\lambda(u(x))k(x)]Q e_i(x),\ \ \ \ x \in\,\mathcal{O}.\]	
Thus, if we assume that $\lambda$ is twice differentiable, with bounded derivatives, we have
\begin{align*}
	\Vert (D\sigma(u)k)&Qe_i\Vert_{H^1}\leq \Vert D^2\lambda(u)(\nabla u,k)Qe_i\Vert_{H}+\Vert [D\lambda(u)\nabla k]\,Qe_i\Vert_{H}+\Vert [D\lambda(u)k] \nabla(Qe_i)\Vert_{H}\\[10pt]
	&\leq c\left(\Vert \nabla u\Vert_{L^4}\Vert  k\Vert_{L^4}\Vert Q e_i\Vert_{L^\infty}+\Vert \nabla k\Vert_{H}\Vert Q e_i\Vert_{L^\infty}+\Vert k\Vert_{L^4}\Vert \nabla(Q e_i)\Vert_{L^4}\right).
\end{align*}
Therefore, if  $d\leq 4$ and we assume	\eqref{gsb113} and \eqref{fx1}, we have
\begin{align*}
	\sum_{i=1}^\infty\Vert (D\sigma(u)k)&Qe_i\Vert^2_{H^1}\leq  c\left(\Vert  u\Vert^2_{H^2}\Vert  k\Vert^2_{H^1}+\Vert  k\Vert^2_{H^1}\right),\end{align*}
	and \eqref{gsb43-bis} follows.
\flushright{$\Box$}
	
	}
\end{Remark}

\medskip

To conclude this section, for every $\mu>0$ and $\delta \in\,\mathbb{R}$, we define on $\mathcal{H}_\delta$ the unbounded linear operator  \[\mathcal{L}_\mu(u,v)=\frac 1\mu\left(\mu\,v, L u\right),\ \ \ \ \ (u,v) \in\,\mathcal{H}_{\delta+1}.\]
It can be proven that $\mathcal{L}_\mu$ is the generator of a  strongly continuous group of 
bounded linear operators $\{S_\mu(t)\}_{t\geq 0}$ on each $\mathcal{H}_\delta$ (for a proof see \cite[Section 7.4]{Pazy}.
Moreover, for every  $z \in\,\mathcal{H}_\delta$
\begin{equation}
	\label{gsb5}
	\Vert S_\mu(t) z\Vert_{\mathcal{H}_\delta}\leq \sqrt{\mu\vee \mu^{-1}}\,\Vert z\Vert_{\mathcal{H}_\delta},\ \ \ \ \ \ t\geq 0.
\end{equation}

Finally, for every $\mu>0$ and $(u,v) \in\,\mathcal{H}_1$, we define
\[B_\mu(u,v)=\frac 1\mu (0,-\gamma(u)v+f(u)),\ \ \ \ \Sigma_\mu(u,v)=\frac 1\mu (0,\sigma(u)).\]

\section{Statement of main results}
In what follows, for every function $u:[0,T]\times \mathcal{O}\to \mathbb{R}^r$ which is differentiable with respect to $t$, we will denote \begin{equation}
\label{gsb95}
{\boldsymbol u}=(u,\partial_t u).
\end{equation}
Moreover, for every $T>0$ we will denote
\begin{equation}
\label{gsb94}
\begin{aligned}
\ds{\mathcal{K}_T:=\left\{ \ \right.}  &  \ds{u \in\,C([0,T];H^2)\ \text{such that}\ u(\cdot,x):[0,T]\to\mathbb{R}^r\ \text{ is differentiable, for a.e.}\ x \in\,\mathcal{O},}\\[10pt]
&\hsp\ds{\left. \text{and}\ \ \boldsymbol{u} \in\,	C([0,T];\mathcal{H}_2)\right\}.}	
\end{aligned}
\end{equation}

By using the notations  introduced in the previous section,  equation \eqref{SPDE} can be rewritten as
\begin{equation} 
\label{SPDE-bis}
\left\{\begin{array}{l}
\ds{d\boldsymbol{u}_\mu(t)=\left[\,\mathcal{L}_\mu\boldsymbol{u}_\mu(t)+B_\mu(\boldsymbol{u}_\mu(t))\right]\,dt+\Sigma_\mu(\boldsymbol{u}_\mu(t))\,dw^Q(t)}\\[10pt]
\ds{\boldsymbol{u}_\mu(0)=(u_0,v_0).}
\end{array}\right.
\end{equation}

\medskip

In section \ref{sec3} we will prove that this equation admits a  mild solution and such solution is unique in a proper space of processes. Namely, we will prove the following result

\begin{Theorem} \label{teo3.3}
Assume Hypotheses \ref{as1}, \ref{as3} and \ref{as2}, and fix $\mu>0$ and  $(u^\mu_0, v^\mu_0) \in\,\mathcal{H}_2$. Then for every $T>0$ and $p\geq 1$,  there exists a unique adapted process $u_\mu$ taking values in $\mathcal{K}_T$, $\mathbb{P}$-a.s., such that $\boldsymbol{u_\mu} \in\,L^2(\Omega;C([0,T];\mathcal{H}_1))$
 and
\begin{equation}
\label{gsb66}	\boldsymbol{u}_\mu(t)=(u^\mu_0,v^\mu_0)+\int_0^t \mathcal{L}_\mu \boldsymbol{u_\mu}(s)\,ds+\int_0^t B_\mu(\boldsymbol{u_\mu}(s))\,ds+\int_0^t\Sigma_\mu(\boldsymbol{u_\mu}(s))\,dw^Q(s).
\end{equation}
	
\end{Theorem}

\begin{Remark}
It is not difficult to prove that in fact we have $\boldsymbol{u_\mu} \in\,L^2(\Omega;C([0,T];\mathcal{H}_2))$.
\end{Remark}
\medskip

Once proved Theorem \ref{teo3.3}, we will study the limiting behavior of $u_\mu$, as $\mu\to 0$ when Hypothesis \ref{as4} and Hypothesis \ref{as5} hold. Namely, we will prove that $u_\mu$ converges in probability to some stochastic quasi-linear parabolic equation.
In order to describe such limiting equation, we need to introduce some notations and preliminary results.

For every fixed $u, v\in\,H^1$ we consider the problem
\begin{equation}
\label{sa40}
dy(t)=-\mathfrak{g}(u)	y(t)\,dt+\sigma(u)dw^Q(t),\ \ \ \ \ y(0)=v,
\end{equation}
where $\mathfrak{g}$ is the mapping defined on $H^1$ and taking values in $\mathbb{R}^{r\times r}$ introduced in Hypothesis \ref{as4}, such that $[\gamma(u)k](x)=\mathfrak{g}(u)k(x)$, for every $u, k \in\,H^1$ and $x \in\,\mathcal{O}$.
Due to Hypotheses  \ref{as2} and \ref{as4}, the process
\[y^{u,v}(t)=e^{-\mathfrak{g}(u)t}v+\int_0^t e^{-\mathfrak{g}(u)(t-s)}\sigma(u)dw^Q(s),\ \ \ \ \ t\geq 0,\] 
is the unique solution of equation \eqref{sa40}, belongs to $L^p(\Omega;C([0,+\infty);H^1))$, for every  $p\geq 1$, and satisfies 
\begin{equation}
\label{sa41-H1}
\sup_{t \geq 0}\mathbb{E}\Vert y^{u,v}(t)\Vert_{H^1}^p\leq c_{p}\left(1+e^{-\gamma_0 t}\Vert v\Vert_{H^1}^p+\Vert u\Vert_{H^1}^p\right),\ \ \ \ \ t\geq 0.
\end{equation}

Now, for every $u \in\,H^1$, we denote by $P^u_t$ the Markov transition semigroup associated with equation \eqref{sa40} and defined by
\[P^u_t\varphi(v)=\mathbb{E}\,\varphi(y^{u,v}(t)),\ \ \ \ \ v \in\,H^1,\ \ \ \ t\geq 0,\]
for every function $\varphi \in\,B_b(H^1)$, where $B_b(H^1)$ is the space of  Borel and bounded functions defined on $H^1$ and taking values in $\mathbb{R}$. 
Moreover, we define
\[\Lambda_u:=\int_0^\infty e^{-\mathfrak{g}(u)s}[\sigma(u)Q][\sigma(u)Q]^\star e^{-\mathfrak{g}^t(u)s}\,ds.\]
Due to Hypotheses  \ref{as2} and \ref{as4}, we have that $\Lambda_u \in\, \mathcal{L}^+_1(H)\cap\mathcal{L}^+_1(H^1)$, with
\begin{equation}
\label{xfine150}
\mbox{Tr}_H\Lambda_u\leq c,\ \ \ \ \ \ \mbox{Tr}_{H^1}\Lambda_u\leq c\,\left(1+\Vert u\Vert_{H^1}\right).	
\end{equation}
 and the Gaussian measure
$\nu^u:=\mathcal{N}(0,\Lambda_u)$ is the unique invariant measure for the semigroup $P^u_t$.

Next, we define
\begin{equation}
\label{sa160}
S(u):=\int_{H^1}\left[D\mathfrak{g}^{-1}(u)	z\right]z\,d\nu^u(z),\ \ \ \ \ u \in\,H^1.
\end{equation}
In Section \ref{sec7}, we will show that  $S:H^1\to H^1$ is well-defined and  we will study  its differentiability properties. Once introduced the mapping $S$, we can state our limiting result.

\begin{Theorem}
\label{teo3.4}
Assume Hypotheses \ref{as1} to \ref{as5} and fix  an arbitrary $\varrho<1$ and $\vartheta \in\,[1,2)$. Moreover, assume that $(u_0^\mu,v_0^\mu) \in\,\mathcal{H}_3$, for every $\mu>0$, with 
\begin{equation}
\label{finex}
\sup_{\mu \in\,(0,1)}\Vert(u^\mu_0,\sqrt{\mu}\,v^\mu_0)\Vert_{\mathcal{H}_1}<\infty,
\end{equation}
and
\begin{equation}
\label{xsa11}
\lim_{\mu\to 0}	\mu^{\delta}\,\Vert(u^\mu_0,\sqrt{\mu}\,v^\mu_0)\Vert^2_{\mathcal{H}_2}=0,\ \ \ \ \ \ \  \lim_{\mu\to 0}	\mu^{1+\delta}\,\Vert(u^\mu_0,\sqrt{\mu}\,v^\mu_0)\Vert^2_{\mathcal{H}_3}=0,
\end{equation}
for some $\delta \in\,(0,1/2)$.
Then, if  $u_0 \in\,H^1$ is such that
\begin{equation}
\label{fx2}
\lim_{\mu\to 0}\Vert u^\mu_0-u_0\Vert_{H^1}=0,	\end{equation}
 for every   $p<2/(\vartheta-1)$ and $\eta, T>0$ we have	
\begin{equation}
\label{sa161}
\lim_{\mu\to 0}\,\mathbb{P}\left(\sup_{t \in \,[0,T]}\Vert u_\mu(t)-u(t)\Vert_{H^\varrho}+\int_0^T \Vert u_\mu(t)-u(t)\Vert_{H^\vartheta}^p\,dt>\eta\right)=0,	
\end{equation}
where $u \in\,L^2(\Omega;C([0,T];H^1)\cap L^2(0,T;H^2))$ is the unique solution of the problem 
\begin{equation}
\label{lim-eq}
\begin{cases}
\ds{\partial_t u(t,x)=\mathfrak{g}^{-1}(u(t))\Delta u(t,x)+\mathfrak{g}^{-1}(u(t))f(u(t,x))+S(u(t))+\mathfrak{g}^{-1}(u(t))\sigma(u(t))\partial_tw^Q(t),}\\[10pt]
\ds{u(0,x)=u_0(x),\ \ x \in\,\mathcal{O},\ \ \ \ \ \ \ \ u(t,x)=0,\ \ x \in \partial\mathcal{O}.}	
\end{cases}
	\end{equation}

\end{Theorem}

In what follows, for the sake of simplicity of notations, we shall denote
\begin{equation}
\label{iet1}
\Lambda_6\mu_i:=\Vert (u^\mu_0,\sqrt{\mu} v^\mu_0)\Vert^2_{\mathcal{H}_i},\ \ \ \ \ \mu>0,\ \ \ \ i=1,2,3.	
\end{equation}

\section{Proof of Theorem \ref{teo3.3}}
\label{sec3}

In this section we study the well-posedness of system \eqref{SPDE-bis}, with initial conditions $(u^\mu_0, v^\mu_0) \in\,\mathcal{H}_2$, for every  $\mu>0$. 
Since here $\mu$ is a fixed positive parameter, without any loss of generality we can consider the case $\mu=1$ and  the coefficients $\mathcal{L}_\mu$, $B_\mu$ and $\Sigma_\mu$ and the initial conditions $u^\mu_0$ and $v^\mu_0$ will be just denoted by $\mathcal{L}$, $B$, $\Sigma$, $u_0$ and $v_0$.

\medskip

 In order to study the well-posedness of equation \eqref{SPDE-bis}, we will use a generalized splitting method, similar to the one introduced in \cite{bds} (see also \cite{salins}). For every  $\eta \in\,\mathcal{K}_T$ we consider the  deterministic auxiliary problem
 \begin{equation}
 \label{gsb50}
  \frac{d}{dt}\boldsymbol{\zeta}(t)=\mathcal{L}\,\boldsymbol{\zeta}(t)+B(\boldsymbol{\zeta}(t)+\boldsymbol{\eta}(t)),\ \ \ \ 
\boldsymbol{\zeta}(0)=(u_0,v_0).
\end{equation}
A function  $\zeta \in\,\mathcal{K}_T$ is a solution to problem \eqref{gsb50} if the equality 
\[\boldsymbol{\zeta}(t)=(u_0,v_0)+\int_0^t\mathcal{L} \boldsymbol{\zeta}(s)\,ds +\int_0^t B(\boldsymbol{\zeta}(s)+\boldsymbol{\eta}(s))\,ds \]
holds in $\mathcal{H}_1$. 
  
\begin{Proposition}
\label{teo3.1}
Assume that Hypotheses \ref{as1} and \ref{as3} are satisfied. Then, for every $(u_0,v_0) \in\,\mathcal{H}_2$ and  $\eta \in\,\mathcal{K}_T,$
 problem \eqref{gsb50} admits a unique solution $\zeta \in\,\mathcal{K}_T$.
 Moreover, for every $t \in\,[0,T]$
 \begin{equation}
	\label{gsb100}
\sup_{s \in\,[0,t]}\Vert \boldsymbol{\zeta}(s)\Vert_{\mathcal{H}_1}\leq c_T \left(1+\sup_{s \in\,[0,t]} \Vert \boldsymbol{\eta}(s)\Vert_{\mathcal{H}_1}	\right),
	\end{equation}
	and
 \begin{equation}
 \label{gsb53}
 \sup_{s \in\,[0,t]}\Vert \boldsymbol{\zeta}(s)\Vert_{\mathcal{H}_2}\leq  \exp\left(c\,\left[\Vert (u_0,v_0)\Vert_{\mathcal{H}_1}+1+\sup_{s \in\,[0,t]} \Vert \boldsymbol{\eta}(s)\Vert_{\mathcal{H}_2}	\right]t\right)\left(\Vert (u_0,v_0)\Vert_{\mathcal{H}_2}+1\right).\end{equation}
\end{Proposition}

Thanks to Proposition \ref{teo3.1}, we can introduce the mapping 
$\Lambda:\mathcal{K}_T\to \mathcal{K}_T$ that associates to every $\eta \in\,\mathcal{K}_T$ the solution $\Lambda(\eta)\in\,\mathcal{K}_T$ of equation \eqref{gsb50}. With this definition, we consider the  stochastic problem
\begin{equation}
\left\{ \begin{array}{l}
 \ds{d\boldsymbol{\eta}(t)=\mathcal{L} \boldsymbol{\eta}(t)\,dt	+\Sigma(\boldsymbol{\eta}(t)+\boldsymbol{\Lambda(\eta)}(t))\,d w^Q(t),}\\[10pt]
\ds{\boldsymbol{\eta}(0)=0.}
	\end{array}\right.
\end{equation}

\begin{Proposition}
\label{teo3.2}
	Under Hypotheses \ref{as1}, \ref{as3} and \ref{as2}, for every $T>0$ and $p\geq 1$ there exists a unique adapted process 
	$\eta$ such that $\boldsymbol{\eta} \in\,L^2(\Omega;C([0,T];\mathcal{H}_2))$ and
		\begin{equation}
		\label{gsb123}	
		\boldsymbol{\eta}(t)=\int_0^t \mathcal{L} \boldsymbol{\eta}(s)\,ds+\int_0^t \Sigma(\boldsymbol{\eta}(s)+\boldsymbol{\Lambda(\eta)}(s))\,dw^Q(s).
		\end{equation}

\end{Proposition}

Once proved Proposition \ref{teo3.2}, we can conclude the proof of Theorem \ref{teo3.3}.
Actually, thanks to Proposition \ref{teo3.2}, there exists $\eta$ such that $\boldsymbol{\eta}$ belongs to  $L^2(\Omega;C([0,T];\mathcal{H}_2))$ and solves equation \eqref{gsb123}.
Thus, if we define
\begin{equation}
\label{gsb122}
u=\eta+\Lambda(\eta),	
\end{equation}
due to Theorem \ref{teo3.1} it is immediate to check  that $u$ takes values in $\mathcal{K}_T$, $\mathbb{P}$-a.s.,  $\boldsymbol{u}$ satisfies equation \eqref{gsb66} and belongs to $L^2(\Omega;C([0,T];\mathcal{H}_1)$. 

Finally, concerning uniqueness, if $\boldsymbol{u}$ is any solution of equation \eqref{gsb66}, we consider the problem
\[d\boldsymbol{\eta}(t)=\mathcal{L}\boldsymbol{\eta}(t)\,dt+\Sigma(\boldsymbol{u}(t))\,dw^Q(t),\ \ \ \ \boldsymbol{\eta}(0)=\boldsymbol{0}.\]
This problem admits a unique solution $\boldsymbol{\eta} \in\,L^p(\Omega;C([0,T];\mathcal{H}_2))$. Then, if we define $\boldsymbol{\zeta}:=\boldsymbol{u}-\boldsymbol{\eta}$, we have that $\boldsymbol{\zeta}$ solves equation \eqref{gsb50}. Since uniqueness holds for equation \eqref{gsb50}, we have that $\zeta=\Lambda(\eta)$ and hence $u=\eta+\Lambda(\eta)$, with $\boldsymbol{\eta}$ being the unique solution of \eqref{gsb123}.

\subsection{Proof of Proposition \ref{teo3.1}}

In this subsection, we assume that Hypotheses \ref{as1} and \ref{as3} are always satisfied. Moreover, $(u_0,v_0) \in\,\mathcal{H}_2$ and $\eta \in\,\mathcal{K}_T$ are fixed once for all.

As a consequence of \eqref{gsb15} and the Lipschitz continuity of $f$ in $H$, for every $\zeta=(\zeta_1,\zeta_2) \in\,\mathcal{H}_1$ we have
\[
\begin{aligned}
\ds{	\Vert B(t,\zeta)\Vert_{\mathcal{H}_1}}  &  \ds{=\Vert -\gamma(\zeta_1+\eta(t))(\partial_t\eta(t)+\zeta_2)+f(\eta(t)+\zeta_1)\Vert_{H}}\\[10pt]
&  \ds{\leq c\left(\Vert \partial_t\eta(t)\Vert_{H}+\Vert \zeta_2\Vert_{H}\right)+c\left(1+\Vert \eta(t)\Vert_{H}+\Vert \zeta_1\Vert_{H}\right), }
\end{aligned}
	\]
	and this implies that $B(t,\cdot):\mathcal{H}_1\to \mathcal{H}_1$ is well defined, for all $t \in\,[0,T]$, and
\begin{equation}
\label{gsb80-bis}
\Vert B(t,\zeta)\Vert_{\mathcal{H}_1} \leq c\,\Vert \zeta\Vert_{\mathcal{H}_1}+c\left(1+\Vert \boldsymbol{\eta}(t)\Vert_{\mathcal{H}_1}\right).
\end{equation}

Due to \eqref{gsb76} and \eqref{n30}, we have
\[
\begin{aligned}
\ds{	\Vert B(t,\zeta)\Vert_{\mathcal{H}_2}}  &  \ds{=\Vert -\gamma(\zeta_1+\eta(t))(\partial_t\eta(t)+\zeta_2)+f(\eta(t)+\zeta_1)\Vert_{H^1}}\\[10pt]
&  \ds{\leq c\left(\Vert \zeta_1+\eta(t)\Vert_{H^2}+1\right)\left(\Vert \partial_t\eta(t)\Vert_{H^1}+\Vert \zeta_2\Vert_{H^1}\right)+c\left(1+\Vert \eta(t)\Vert_{H^1}+\Vert \zeta_1\Vert_{H^1}\right), }
\end{aligned}
	\]
so that $B(t,\cdot)$ maps $\mathcal{H}_2$ into $\mathcal{H}_2$, for every $t \in\,[0,T]$, and
\begin{equation}
\label{gsb80}
\Vert B(t,\zeta)\Vert_{\mathcal{H}_2} \leq c\left(1+\Vert \boldsymbol{\eta}(t)\Vert^2_{\mathcal{H}_2}+\Vert \zeta\Vert^2_{\mathcal{H}_2}\right).
\end{equation}
Moreover, due to \eqref{gsb76} and \eqref{xfine	100} we have
\begin{align*}
	\Vert B(t,\zeta)-B(t,&\bar{\zeta})\Vert_{\mathcal{H}_2}=\Vert -\gamma(\zeta_1+\eta(t))(\zeta_2-\bar{\zeta}_2)+f(\eta(t)+\zeta_1)-f(\eta(t)+\bar{\zeta}_1)\Vert_{H^1}\\[10pt]
	&\hsl\leq c\left(1+\Vert \zeta_1\Vert_{H^2}+\Vert \eta(t)\Vert_{H^2}\right)\Vert \zeta_2-\bar{\zeta}_2\Vert_{H^1}+c\left(1+\Vert \eta(t)\Vert_{H^2}+\Vert \zeta_1\Vert_{H^2}\right)\Vert \zeta_1-\bar{\zeta}_1\Vert_{H^1},\end{align*}
	so that
	\begin{equation}
	\label{xfine101}	
	\Vert B(t,\zeta)-B(t,\bar{\zeta})\Vert_{\mathcal{H}_2}\leq 
	c\left(1+\Vert \eta(t)\Vert_{H^2}+\Vert \zeta_1\Vert_{H^2}\right)\Vert \zeta-\bar{\zeta}\Vert_{\mathcal{H}_2}.
	\end{equation}

Now, for every $\zeta \in\,\mathcal{K}_T$ we define
\[\mathcal{A}(\zeta)(t):=S(t)(u_0,v_0)+\int_0^t S(t-s)B(s,\boldsymbol{\zeta}(s))\,ds,\ \ \ \ t \in\,[0,T].\]
Due to to \eqref{gsb5} and \eqref{gsb80}, we have
\begin{align}\begin{split} \label{xfine105}
	\Vert \mathcal{A}(\zeta)(t)&\Vert_{\mathcal{H}_2}\leq \Vert(u_0,v_0)\Vert_{\mathcal{H}_2}+\int_0^t\Vert B(s,\boldsymbol{\zeta}(s))\Vert_{\mathcal{H}_2}\,ds\\[10pt]
	&\hsp\leq \Vert(u_0,v_0)\Vert_{\mathcal{H}_2}+c\,\left(\chi_t (\eta)+\sup_{s \in\,[0,t]}\Vert \boldsymbol{\zeta}(s)\Vert_{\mathcal{H}_2}^2\right)t,\end{split}
\end{align}
where
\[\chi_t(\eta):=\sup_{s \in\,[0,t]}\Vert \boldsymbol{\eta}(s)\Vert_{\mathcal{H}_2}^2.\]
Hence, if for  any $t \in\,(0,T]$ and $R>0$ we define 
\begin{equation} \label{KRT}	\mathcal{K}_t(R):=\left\{ \zeta \in\,\mathcal{K}_T\,:\,\sup_{s \in\,[0,t]}\Vert \boldsymbol{\zeta}(s)\Vert_{\mathcal{H}_2}\leq R\right\},\end{equation}
we have
\[\zeta \in\,\mathcal{K}_t(R)\Longrightarrow \sup_{s \in\,[0,t]}\Vert \mathcal{A}(\zeta)(s)\Vert_{\mathcal{H}_2}\leq  \Vert(u_0,v_0)\Vert_{\mathcal{H}_2}+c\,\left(\chi_T (\eta)+R^2\right)t.\]
This implies that, if we take 
\[R>\Vert(u_0,v_0)\Vert_{\mathcal{H}_2},\ \ \ \ \ \ \ \ t_1\leq \frac{R-\Vert(u_0,v_0)\Vert_{\mathcal{H}_2}}{c\,\left(\chi_T (\eta)+R^2\right)},\]
we have that $\mathcal{A}$ maps $\mathcal{K}_{t_1}(R)$ into itself. Moreover, due to \eqref{xfine101}
\[\zeta, \bar{\zeta} \in\,\mathcal{K}_t(R)\Longrightarrow \sup_{s \in\,[0,t]}\Vert \mathcal{A}(\zeta)(s)-\mathcal{A}(\bar{\zeta})(s)\Vert_{\mathcal{H}_2}\leq c\,\left(\chi_T (\eta)+R\right)t\sup_{s \in\,[0,t]}\Vert\boldsymbol{\zeta}-\boldsymbol{\bar{\zeta}}\Vert_{\mathcal{H}_2}.\]
Hence, taking 
\[\bar{t}:=t_1\wedge \left(2c\,\left(\chi_T (\eta)+R\right)\right)^{-1},\]
we conclude that $\mathcal{A}$ maps $\mathcal{K}_{\bar{t}}(R)$ into itself as a contraction. This means that there exists a unique $\zeta \in\,\mathcal{K}_{\bar{t}}(R)$ such that 
\[\boldsymbol{\zeta}(t)=S(t)(u_0,v_0)+\int_0^t S(t-s)B(s,\boldsymbol{\zeta}(s))\,ds,\ \ \ \ t \in\,[0,\bar{t}].\]

This means that there exists a local solution to problem \eqref{gsb50}. Thus, in order to conclude the proof of Proposition \ref{teo3.1}, we have to prove that such solution  can be extended to a unique global solution defined in $[0,T]$,  satisfying \eqref{gsb100} and \eqref{gsb53}.

Let $\zeta$ be a local solution defined on a maximal interval $(0,T^\prime)$. We show that it does not blow up at $T^\prime$. For every $t \in\,[0,T^\prime)$, due to \eqref{gsb15}
we have 
\[\begin{aligned}
\ds{\frac 12\,\frac d{dt}\Vert \boldsymbol{\zeta}(t)\Vert_{\mathcal{H}_1}^2}  &  \ds{=-\langle \gamma(\zeta(t)+\eta(t))\left(\partial_t\zeta(t)+\partial_t\eta(t)\right),\partial_t\zeta(t)\rangle_{H}+\langle f(\zeta(t)+\eta(t)),\partial_t\zeta(t)\rangle_H}\\[10pt]
& \hsp \ds{\leq c\,\Vert \boldsymbol{\zeta}(t)\Vert_{\mathcal{H}_1}^2+c\left(1+\Vert \boldsymbol{\eta}(t)\Vert_{\mathcal{H}_1}^2\right),}
	\end{aligned}\]
	and this implies \eqref{gsb100}, that is
	\[
\Vert \boldsymbol{\zeta}(t)\Vert_{\mathcal{H}_1}\leq c_T \left(\Vert (u_0,v_0)\Vert_{\mathcal{H}_1}+1+\sup_{t \in\,[0,T]} \Vert \boldsymbol{\eta}(t)\Vert_{\mathcal{H}_1}	\right).\]
	
Moreover, due to \eqref{gsb76-bis}  and \eqref{n30} 
we have
\[\begin{array}{l}
\ds{\frac 12\,\frac d{dt}\Vert \boldsymbol{\zeta}\Vert_{\mathcal{H}_2}^2=-\langle\gamma(\zeta(t)+\eta(t))(\partial_t\zeta(t)+\partial_t \eta(t)),\partial_t\zeta(t)\rangle_{H^1}+\langle f(\zeta(t)+\eta(t)),\partial_t\zeta(t)\rangle_{H^1}}	\\[14pt]
\ds{\hslp\leq c\,\Vert \zeta(t)+\eta(t)\Vert_{H^1}^{1/4}\Vert \zeta(t)+\eta(t)\Vert_{H^2}^{3/4}\Vert \partial_t\zeta(t)+\partial_t\eta(t)\Vert_{H}^{1/4}\Vert \partial_t\zeta(t)+\partial_t\eta(t)\Vert_{H^1}^{3/4}}\\[14pt]
\ds{\hsp+c\left(\Vert \partial_t \zeta(t)\Vert_{H^1}+\Vert \partial_t\eta(t)\Vert_{H^1}+\Vert \zeta(t)\Vert_{H^1}+\Vert\eta(t)\Vert_{H^1}+1\right)\Vert \partial_t \zeta(t)\Vert_{H^1},}
 \end{array}\]
so that
\[\begin{array}{ll}
\ds{\frac d{dt}\Vert \boldsymbol{\zeta}(t)\Vert_{\mathcal{H}_2}^2\leq} & \ds{ c\,\left(1+\Vert \boldsymbol{\eta}(t)\Vert_{\mathcal{H}_2}^2\right)\Vert \boldsymbol{\zeta}(t)\Vert_{\mathcal{H}_2}^2+c\,\Vert \boldsymbol{\eta}(t)\Vert_{\mathcal{H}_2}^2}\\[14pt]
& \ds{+c\Vert \zeta(t)\Vert_{H^1}^{1/2}\Vert \zeta(t)\Vert_{H^2}^{1/2}\Vert \partial_t\zeta(t)\Vert_{H}^{1/2}\Vert \partial_t\zeta(t)\Vert_{H^1}^{3/2}.	}
\end{array}
\]
Now, thanks to \eqref{gsb100} we have
\[\begin{array}{l}
\ds{ \Vert\zeta(t)\Vert_{H^1}^{1/2}\Vert \zeta(t)\Vert_{H^2}^{1/2}\Vert \partial_t\zeta(t)\Vert_{H}^{1/2}\Vert \partial_t \zeta(t)\Vert_{H^1}^{3/2}}\\[14pt]
\ds{\hslp\leq c\left(\Vert (u_0,v_0)\Vert_{\mathcal{H}_1}+1+\sup_{s \in\,[0,t]} \Vert \boldsymbol{\eta}(s)\Vert_{\mathcal{H}_1}	\right)\Vert \zeta(t)\Vert_{H^2}^{1/2}\Vert \partial_t \zeta(t)\Vert_{H^1}^{3/2}}\\[14pt]
\ds{\hsp\leq c\left(\Vert (u_0,v_0)\Vert_{\mathcal{H}_1}+1+\sup_{s \in\,[0,t]} \Vert \boldsymbol{\eta}(s)\Vert_{\mathcal{H}_1}	\right)\,\Vert \boldsymbol{\zeta}(t)\Vert^2_{\mathcal{H}_2}.}
\end{array}
\]
This implies 
\[\frac d{dt}\Vert \boldsymbol{\zeta}(t)\Vert_{\mathcal{H}_2}^2\leq c\,\left(\Vert (u_0,v_0)\Vert_{\mathcal{H}_1}+1+\sup_{s \in\,[0,t]} \Vert \boldsymbol{\eta}(s)\Vert_{\mathcal{H}_2}^2	\right)\le(\Vert \boldsymbol{\zeta}(t)\Vert_{\mathcal{H}_2}^2+1\r),\]
and hence
\[\Vert \boldsymbol{\zeta}(t)\Vert_{\mathcal{H}_2}^2\leq \exp\left(c\,\left(\Vert (u_0,v_0)\Vert_{\mathcal{H}_1}+1+\sup_{s \in\,[0,t]} \Vert \boldsymbol{\eta}(s)\Vert_{\mathcal{H}_2}^2	\right)t\right)\left(\Vert (u_0,v_0)\Vert_{\mathcal{H}_2}^2+1\right).\]

Finally, as far as uniqueness is concerned, if $\zeta_1$ and $\zeta_2$ are two solutions and if we denote $\rho:=\zeta_1-\zeta_2$, due to \eqref{gsb102} we have
\[\begin{aligned}
	&\ds{\frac 12 \frac d{dt}\Vert \boldsymbol{\rho}\Vert^2_{\mathcal{H}_1}}\\[10pt]
	&\ds{=-\langle \gamma(\zeta_1(t)+\eta(t))\partial_t \rho(t),\partial_t \rho(t)\rangle_H-\langle \left[\gamma(\zeta_1(t)+\eta(t))-\gamma(\zeta_2(t)+\eta(t))\right]\partial_t \zeta_2(t),\partial_t \rho(t)\rangle_H}\\[10pt]
	&\hsllp\ds{-\langle \left[\gamma(\zeta_1(t)+\eta(t))-\gamma(\zeta_2(t)+\eta(t))\right]\partial_t \eta(t)+\left[f(\zeta_1(t)+\eta(t))-f(\zeta_2(t)+\eta(t))\right],\partial_t \rho(t)\rangle_H}\\[10pt]
	&\hsp\ds{\leq c\left(1+\Vert \partial_t\eta(t)\Vert_{H^1}+\Vert \partial_t\zeta_2(t)\Vert_{H^1}\right)\Vert \partial_t\rho(t)\Vert_{H}\Vert \rho(t)\Vert_{H^1}.}
\end{aligned}\]
Therefore,
\[ \frac d{dt}\Vert \boldsymbol{\rho}\Vert^2_{\mathcal{H}_1}\leq 
c\left(1+\sup_{s \in\,[0,t]}\Vert \partial_t\eta(s)\Vert_{H^1}+\sup_{s \in\,[0,t]}\Vert \partial_t\zeta_2(s)\Vert_{H^1}\right)\Vert \boldsymbol{\rho}(t)\Vert_{\mathcal{H}_1}^2,\]
and this allows to conclude that $\rho=0$, that is $\zeta_1=\zeta_2$.

\subsection{Proof of Proposition \ref{teo3.2}}

As a consequence of Proposition \ref{teo3.1}, we can introduce the mapping 
$\Lambda:\mathcal{K}_T\to\mathcal{K}_T$ that associates to every $\eta$ the solution $\Lambda(\eta)$ of equation \eqref{gsb50}. 
In what follows, we want to prove that for every $T>0$ the stochastic equation
\begin{equation}
\label{gsb104}	
d\boldsymbol{\eta}(t)=\mathcal{L} \boldsymbol{\eta}(s)\,dt+\Sigma(\boldsymbol{\eta}(s)+\boldsymbol{\Lambda(\eta)}(s))\,dw^Q(s),\ \ \ \ \boldsymbol{\eta}(0)=0,
\end{equation}
admits a unique mild solution $\eta$ such that $\boldsymbol{\eta} \in\,L^p(\Omega;C([0,T];\mathcal{H}_2))$, for every $p\geq 1$.

To this purpose, we start by showing that the mapping $\Lambda$ 
is Lipschitz continuous with respect to the $\mathcal{H}_1$-norm, locally in  $\mathcal{H}_2$.
\begin{Lemma}
	\label{lem3.11}
	For every $T, R>0$ there exists some $L_{T, R}>0$ such that if $\eta_1,  \eta_2 \in\,\mathcal{K}_{T}(R)$ (as defined in \eqref{KRT}) then for every $t \in\,[0,T]$
	\begin{equation}
	\label{gsb101}
	\sup_{s \in\,[0,t]}\Vert \boldsymbol{\Lambda(\eta_1)}(s)-\boldsymbol{\Lambda(\eta_2)}(s)\Vert_{\mathcal{H}_1}\leq L_{T, R} \int_0^t\Vert \boldsymbol{\eta_1}(s)-\boldsymbol{\eta_2}(s)\Vert_{\mathcal{H}_1}\,ds.	
	\end{equation}
	\end{Lemma}
	
\begin{proof}
	We fix $R>0$ and $\eta_1, \eta_2$ in $\mathcal{K}_{T}(R)$ and we define $\rho:=\Lambda(\eta_1)-\Lambda(\eta_2)$.
	We have $\rho(0)=\partial_t \rho(0)=0$ and
\[\begin{aligned}
&\ds{\frac 12 \frac d{dt}\Vert \boldsymbol{\rho}(t)\Vert^2_{\mathcal{H}_1}=-\langle \left[\gamma(\Lambda(\eta_1)(t)+\eta_1(t))-	\gamma(\Lambda(\eta_2)(t)+\eta_2(t))\right]	\left(\partial_t \Lambda(\eta_1)(t)+\partial_t\eta_1(t)\right),\partial_t\rho(t)\rangle_H}\\[10pt]
&\hslp\ds{-\langle\gamma(\Lambda(\eta_2)(t)+\eta_2(t))\partial_t\rho(t),\partial_t\rho(t)\rangle_H-\langle\gamma(\Lambda(\eta_2)(t)+\eta_2(t))\partial_t(\eta_1-\eta_2)(t),\partial_t\rho(t)\rangle_H}\\[10pt]
&\hsp\ds{+\langle f(\Lambda(\eta_1)(t)+\eta_1(t))-f(\Lambda(\eta_2)(t)+\eta_2(t)),\partial_t\rho(t)\rangle_H.}
\end{aligned}\]
Then, in view of  \eqref{gsb15} and \eqref{gsb102}, we have
\[\begin{aligned}
\ds{ \frac d{dt}\Vert \boldsymbol{\rho}(t)\Vert^2_{\mathcal{H}_1}\leq}  &  \ds{\, c\,\Vert \partial_t\rho(t)\Vert_H\left(\Vert \rho(t)\Vert_{H^1}+\Vert (\eta_1-\eta_2)(t)\Vert_{H^1}\right)\left(\Vert \partial_t\eta_1(t)+ \partial_t \Lambda(\eta_1)(t)\Vert_{H^1}\right)}\\[10pt]
&  \ds{+c\,\Vert \partial_t\rho(t)\Vert_{H}\left(\Vert \rho(t)\Vert_{H}+\Vert \partial_t\rho(t)\Vert_{H}+\Vert (\eta_1-\eta_2)(t)\Vert_{H}+\Vert \partial_t\left(\eta_1-\eta_2\right)(t)\Vert_{H}\right).}
\end{aligned}\]
This implies
\[\frac d{dt}\Vert \boldsymbol{\rho}(t)\Vert^2_{\mathcal{H}_1}\leq M(\eta_1)\Vert \boldsymbol{\rho}(t)\Vert^2_{\mathcal{H}_1}+M(\eta_1)\Vert \boldsymbol{\eta_1}(t)-\boldsymbol{\eta_2}(t)\Vert^2_{\mathcal{H}_1},\]
where
\[M(\eta)=c\,\left(1+\sup_{t \in\,[0,T]}\Vert \partial_t(\eta+\Lambda(\eta))(t)\Vert_{H^1}\right).\]
According to \eqref{gsb53}, we have that there exists some  constant $m_{T,R}>0$ such that
\[\eta \in\, \mathcal{K}_{T}(R)\Longrightarrow M(\eta)\leq m_{T, R}.\]
Hence
\[\frac d{dt}\Vert \boldsymbol{\rho}(t)\Vert^2_{\mathcal{H}_1}\leq m_{T, R}\, \Vert \boldsymbol{\rho}(t)\Vert^2_{\mathcal{H}_1}+m_{T, R}\,\Vert \boldsymbol{\eta_1}(t)-\boldsymbol{\eta_2}(t)\Vert^2_{\mathcal{H}_1},\]
so that
\[\Vert \boldsymbol{\rho}(t)\Vert^2_{\mathcal{H}_1}\leq m_{T, R}\, e^{m_{T, R} T}\int_0^t\Vert \boldsymbol{\eta_1}(s)-\boldsymbol{\eta_2}(s)\Vert^2_{\mathcal{H}_1}\,ds,\]
and \eqref{gsb101} follows.

\end{proof}
	
Now,  for every $R>0$, we fix a smooth function $\Phi_R$, with $0\leq \Phi_R\leq 1$,  such that $\Phi_R(r)=1$, if $r\leq R$ and $\Phi_R(r)=0$, if $r\geq R+1$.
Moreover, for every $\eta \in\,\mathcal{K}_T$ and $t \in\,[0,T]$ we define
\[\Psi_R(t,\eta)=\Phi_R\left(\,\sup_{s\leq t}\Vert \boldsymbol{\eta}(s)\Vert_{\mathcal{H}_2}\right).\]
It is immediate to check that there exists some $c$, independent of $R$, such that for every $\eta_1, \eta_2 \in\,\mathcal{K}_T$ and $t \in\,[0,T]$
\begin{equation}
\label{gsb117}
\left|\Psi_R(t,\eta_1)-\Psi_R(t,\eta_2)\right|\leq c\,\sup_{s\leq t}\Vert \boldsymbol{\eta}_1(t)-\boldsymbol{\eta}_2(t)\Vert_{\mathcal{H}_2}.
	\end{equation}

Next, we introduce the regularized problem
\begin{equation}
\label{gsb104-R}	
d\boldsymbol{\eta}(t)=\mathcal{L} \boldsymbol{\eta}(s)\,dt+\Psi_R^2(t,\eta)\Sigma(\boldsymbol{\eta}(t)+\boldsymbol{\Lambda(\eta)}(t))\,dw^Q(t)\ \ \ \ \boldsymbol{\eta}(0)=0.
\end{equation}

\begin{Lemma}
\label{lem3.12}

For every $T, R>0$ and $p\geq 2$, problem 	\eqref{gsb104-R} admits a unique  $\mathcal{K}_T$-valued weak solution $\eta_R$, such that
\begin{equation}
\label{gsb120}
\mathbb{E} \,\sup_{t \in\,[0,T]} \Vert \boldsymbol{\eta_R}(t)\Vert^p_{\mathcal{H}_2}\leq c_{T,p},
\end{equation}
for some constant $c_{T, p}>0$ independent of $R$.

\end{Lemma}

\begin{proof}
Equation \eqref{gsb104-R}, can be rewritten as the following integral equation 
\begin{equation}
\label{gsb105}
\boldsymbol{\eta}(t)=\int_0^t S(t-s) 	B_R(s,\eta)\,dw^Q(s),
\end{equation}
where $S(t)$ is the group generated by the operator $\mathcal{L}$ in $\mathcal{H}_2$  and 
\[B_R(t,\eta)=\Psi_R^2(t,\eta)\,\Sigma(\boldsymbol{\eta}(t)+\boldsymbol{\Lambda(\eta)}(t)).\]

As a first step, we  prove that there exists $c_T>0$, independent of $R$, such that for every $\eta \in\,\mathcal{K}_T$ and $t \in\,[0,T]$
\begin{equation}
\label{gsb116}
\Vert B_R(t,\eta)\Vert_{\mathcal{L}_2(H_Q,\mathcal{H}_2)}\leq c_T\,\left(1+\sup_{s\leq t}\Vert \boldsymbol{\eta}(s)\Vert_{\mathcal{H}_1}\right),
	\end{equation}
and there exists $c_{T, R}>0$ such that for every $\eta_1, \eta_2 \in\,\mathcal{K}_T$ and $t \in\,[0,T]$
\begin{equation}
\label{gsb110}
\begin{array}{l}\ds{	\Vert B_R(t,\eta_1)-B_R(t,\eta_2)\Vert_{\mathcal{L}_2(H_Q,\mathcal{H}_2)}\leq c_{T, R}\,\sup_{s\leq t}\Vert \boldsymbol{\eta_1}(s)-\boldsymbol{\eta_2}(s)\Vert_{\mathcal{H}_2}.}
\end{array}
	\end{equation}

Due to \eqref{gsb43} and \eqref{gsb100}, we have
\[\begin{array}{l}
\ds{	\Vert B_R(t,\eta)\Vert_{\mathcal{L}_2(H_Q,\mathcal{H}_2)}= \Psi_R^2(t,\eta)\,\Vert \sigma(\eta(t)+\Lambda(\eta)(t)) \Vert_{\mathcal{L}_2(H_Q,H^1)}}\\[10pt]
\hsp\ds{\leq c\,\left(1+\Vert \eta(t)+\Lambda(\eta)(t)\Vert_{H^1}\right)\leq c_T\,\left(1+\sup_{s\leq t}\Vert \boldsymbol{\eta}(s)\Vert_{\mathcal{H}_1}\right),}
\end{array}
\]
and this gives \eqref{gsb116}.
As for \eqref{gsb110}, we have
\[\begin{array}{l}
\ds{\Vert B_R(t,\eta_1)-B_R(t,\eta_2)\Vert_{\mathcal{L}_2(H_Q,\mathcal{H}_2)}}\\[10pt]
\hsllp\ds{\leq  \left|\Psi_R(t,\eta_1)-\Psi_R(t,\eta_2)\right|\,\Psi_R(t,\eta_1)\,\Vert \sigma(\eta_1(t)+\Lambda(\eta_1)(t))\Vert_{\mathcal{L}_2(H_Q,H^1)}}	\\[10pt]
\hslp\ds{+  \left|\Psi_R(t,\eta_1)-\Psi_R(t,\eta_2)\right|\,\Psi_R(t,\eta_2)\,\Vert \sigma(\eta_2(t)+\Lambda(\eta_2)(t))\Vert_{\mathcal{L}_2(H_Q,H^1)}}	\\[10pt]
\hslp\ds{+  \Psi_R(t,\eta_1)\Psi_R(t,\eta_2)\,\Vert \sigma(\eta_1(t)+\Lambda(\eta_1)(t))-\sigma(\eta_2(t)+\Lambda(\eta_2)(t))\Vert_{\mathcal{L}_2(H_Q,H^1)}}\\[10pt]
\hsp\ds{=:I_{R,1}(t)+I_{R,2}(t)+I_{R,3}(t).}	
\end{array}\]

Thanks to  \eqref{gsb43} and \eqref{gsb100}, there exists $c_R>0$ such that 
\[\sup_{t \in\,[0,T]}\Psi_R(t,\eta_i)\,\Vert \sigma(\eta_i(t)+\Lambda(\eta_i)(s))\Vert_{\mathcal{L}_2(H_Q,H^1)}\leq c_R,\ \ \ \ i=1, 2.\]
Therefore, thanks to \eqref{gsb117}, we get
\begin{equation}
\label{gsb118}	
I_{R,1}(t)+I_{R,2}(t)\leq c_R\, \sup_{s\leq t}\,\Vert \boldsymbol{\eta_1}(s)-\boldsymbol{\eta_2}(s)\Vert_{\mathcal{H}_2}.
\end{equation}
As for $I_{R, 3}(t)$, according to \eqref{gsb43-tris} we have
\[
\begin{aligned}
\ds{I_{R, 3}(t)\leq}  &  \ds{ \Psi_R(t,\eta_1)\Psi_R(t,\eta_2)\,\left(\Vert \eta_1(t)-\eta_2(t)\Vert_{H^1}+\Vert \La(\eta_1)(t)-\La(\eta_2)(t)\Vert_{H^1}\right)}\\[10pt]
&  \ds{\times \left(1+\Vert \eta_2(t)\Vert_{H^2}+\Vert \La(\eta_2)(t)\Vert_{H^2}\right),}
	\end{aligned}	\]
and then, from  \eqref{gsb53} and \eqref{gsb101}, we get
\[I_{R, 3}(t)\leq c_{T,R} \sup_{s\leq t}\Vert \boldsymbol{\eta_1}(s)-\boldsymbol{\eta_2}(s)\Vert_{\mathcal{H}_1}.\]
Combining this together with \eqref{gsb118}, we obtain \eqref{gsb110}.

Next, for every $\mathcal{K}_T$-valued $\eta$ such that $\boldsymbol{\eta} \in\,L^p(\Omega;\mathcal{H}_2)$,  we define 
\[\mathcal{F}_R(\eta)(t)=\int_0^tS(t-s)B_R(s,\eta)\,dw^Q(s),\ \ \ \ t \in\,[0,T].	\]
If we can prove that once fixed $p\geq 2$
there exists $\alpha \in\,(0,1)$ and $t_0>0$ such that
for every $\eta_1, \eta_2$ 
 \begin{equation}
 \label{gsb119}
 \mathbb{E}\,\sup_{t\leq t_0}\Vert \mathcal{F}_R(\eta_1)(t)-\mathcal{F}_R(\eta_2)(t)\Vert_{\mathcal{H}_2}^p\leq \alpha\, 
\mathbb{E}\,\sup_{t\leq t_0}\Vert \boldsymbol{\eta_1}(t)-\boldsymbol{\eta_2}(t)\Vert_{\mathcal{H}_2}^p,	
 \end{equation}
then  equation \eqref{gsb104-R} has a unique solution $\eta_R \in\,L^p(\Omega;C([0,t_0];\mathcal{H}_2))$. Since we can repeat the same argument in all intervals $[t_0,2t_0]$, $[2t_0, 3t_0]$ and so on, we get a unique solution $\eta_R$ defined in the whole interval $[0,T]$.

As shown for example in \cite[Proposition 7.3]{DPZ}, by using a stochastic factorization argument we can prove that $\mathcal{F}_R(\eta)$ is continuous in time with values in $\mathcal{H}_2$ and for every $t>0$ and $p>2$
\[\mathbb{E}\,\sup_{s\leq t}\Vert \mathcal{F}_R(\eta_1)(s)-\mathcal{F}_R(\eta_2)(s)\Vert_{\mathcal{H}_2}^p\leq c_{T, p}\,\mathbb{E}\int_0^t 	\Vert B_R(s,\eta_1)-B_R(s,\eta_2)\Vert^p_{\mathcal{L}_2(H_Q,\mathcal{H}_2)}\,ds.\]
Hence, thanks to \eqref{gsb110}, we obtain
\[\begin{array}{l}
\ds{	\mathbb{E}\,\sup_{s\leq t}\Vert \mathcal{F}_R(\eta_1)(s)-\mathcal{F}_R(\eta_2)(s)\Vert_{\mathcal{H}_2}^p\leq c_{T,R, p}\,t\,\mathbb{E}\,\sup_{s\leq t}\Vert \boldsymbol{\eta_1}(s)-\boldsymbol{\eta_2}(s)\Vert^p_{\mathcal{H}_2}.
}
\end{array}\]
This implies that, if we fix some $\a<1$ and take $t_0>0$ such that
\[c_{T,R, p}\,t_0\leq \alpha,\]
 \eqref{gsb119} follows. Finally, we obtain estimate \eqref{gsb120} from \eqref{gsb116} and the Gronwall lemma, since
 \[\begin{aligned}
\ds{	\mathbb{E}\,\sup_{s\leq t}\Vert \boldsymbol{\eta_R}(s)\Vert_{\mathcal{H}_2}^p\leq c_{T,p}\,\mathbb{E}\int_0^t 	\Vert B_R(s,\eta_R)\Vert^p_{\mathcal{L}_2(H_Q,\mathcal{H}_2)}\,ds\leq c_{T,p}\int_0^t \left(1+\mathbb{E}\,\sup_{r\leq s}\Vert \boldsymbol{\eta_R}(r)\Vert_{\mathcal{H}_1}^p\right)\,ds.}
\end{aligned}\]

\end{proof}

In order to conclude the proof of Proposition \ref{teo3.2}, for every $R>0$ we define
\[\tau_R:=\inf \left\{ t \in\,[0,T]\,:\, \Vert \boldsymbol{\eta_R}(t)\Vert_{\mathcal{H}_2}> R \right\},\]
with the  convention that $\inf \emptyset=+\infty$.
Due to \eqref{gsb120}, we have
\[\mathbb{P}\left(\tau_R < +\infty\right)\leq \mathbb{P}\left(\sup_{t \in\,[0,T]}\Vert \boldsymbol{\eta_R}(t)\Vert_{\mathcal{H}_2}\geq R\right)\leq \frac 1{R^2}\,\mathbb{E}\,\sup_{t \in\,[0,T]}\Vert \boldsymbol{\eta_R}(t)\Vert_{\mathcal{H}_2}^2\leq \frac{c_T}{R^2},\]
so that, if we set 
\[\tau:=\lim_{R\to\infty} \tau_R,\]
we get 
$\mathbb{P}\left(\tau=+\infty\right)=1$.
Now, if we fix $\omega \in\,\{\tau=+\infty\}$ and $t \in\,[0,T]$, there exists some $R>0$ such that $t\leq \tau_R(\omega)$ and then we define
\[\eta(t)(\omega)=\eta_R(t)(\omega).\]
This is a good definition, because we can prove that if $t\leq \sigma(\omega):=\tau_{R_1}(\omega)\wedge \tau_{R_2}(\omega)$, then
\begin{equation}
\label{gsb121}
\eta_{R_1}(t)(\omega)=\eta_{R_2}(t)(\omega).	
\end{equation}
Actually, if we assume that $R_1\leq R_2$, for every $t \in\,[0,T]$ we have
\[\begin{array}{l}
\ds{\boldsymbol{\eta_{R_1}}(t\wedge \sigma)-\boldsymbol{\eta_{R_2}}(t\wedge \sigma)=\int_0^t \mathbb{I}_{\{s\leq \sigma\}}\,S(t-s)\left[B_{R_1}(s\wedge \sigma,\eta_{R_1})-B_{R_{2}}(s\wedge\sigma,\eta_{R_2})\right]\,dw^Q(s).}
\end{array}
\]
Then, by using \eqref{gsb110} as in the proof of Lemma \ref{lem3.12}, we conclude that 
\[\boldsymbol{\eta_{R_1}}(t\wedge \sigma)=\boldsymbol{\eta_{R_2}}(t\wedge \sigma),\ \ \ \ t \in\,[0,T],\] which implies \eqref{gsb121}.
Once we have that $\eta$ is a well  defined process and  $\boldsymbol{\eta}\in\,L^2(\Omega,C[0,T];\mathcal{H}_2)$, by noticing that
\[B_R(t,\eta)=\Sigma(\boldsymbol{\eta}(t)+\Lambda(\boldsymbol{\eta})(t)),\ \ \ \ t\leq \tau_R,\]
we have that $\eta$ solves equation \eqref{gsb104}.

Finally, we show that such solution is unique. If $\eta_1$ and $\eta_2$ are two solutions,  we have
\[\boldsymbol{\eta_1}(t)-\boldsymbol{\eta_2}(t)=\int_0^t S(t-s)\left[\Sigma(\boldsymbol{\eta_1}(s)+\Lambda(\boldsymbol{\eta_1})(s))-\Sigma(\boldsymbol{\eta_2}(s)+\Lambda(\boldsymbol{\eta_2})(s))\right]\,dw^Q(s).\]
Thanks to \eqref{gsb7}, this yields
\[\begin{array}{l}
\ds{\mathbb{E}\sup_{s \in\,[0,t]} \Vert \boldsymbol{\eta_1}(s)-\boldsymbol{\eta_2}(s)\Vert_{\mathcal{H}_1}^p\leq c_{T,p} \int_0^t\mathbb{E}\,\Vert \sigma(\eta_1(s)+\Lambda(\eta_1)(s))-\sigma(\eta_1(s)+\Lambda(\eta_1)(s))\Vert_{\mathcal{L}_2(H_Q,H)}^p\,ds}	\\[14pt]
\hsp\ds{\leq c_{T, p}\int_0^t \mathbb{E}\left(\Vert \eta_1(s)-\eta_2(s)\Vert_{H^1}^p+\Vert \Lambda(\eta_1)(s)-\Lambda(\eta_2)(s)\Vert_{H^1}^p\right)\,ds.}
\end{array}\]
Thus, thanks to \eqref{gsb101} we get
\[\mathbb{E}\sup_{s \in\,[0,t]} \Vert \boldsymbol{\eta_1}(s)-\boldsymbol{\eta_2}(s)\Vert_{\mathcal{H}_1}^p\leq c_{T, p}\int_0^t \mathbb{E}\sup_{r \in\,[0,s]} \Vert \boldsymbol{\eta_1}(r)-\boldsymbol{\eta_2}(r)\Vert_{\mathcal{H}_1}^p\,ds,\]
and we conclude that $\eta_1=\eta_2$.

\color{black}

\section{Uniform estimates in $\mathcal{H}_1$}
\label{sec4}

 In Theorem \ref{teo3.3} we have seen that, under Hypotheses \ref{as1}, \ref{as3} and \ref{as2}, for every $\mu>0$ and every initial condition $(u^\mu_0, v^\mu_0) \in\,\mathcal{H}_2$, equation \eqref{SPDE} admits a unique solution $\boldsymbol{u}_\mu=(u_\mu,\partial_t u_\mu)$ in   $L^2(\Omega;C([0,T];\mathcal{H}_1))$, taking values  in the functional space $\mathcal{K}_T$ defined in \eqref{gsb94}, $\mathbb{P}$-a.s. In particular,  $\boldsymbol{u}_\mu \in\,C([0,T];\mathcal{H}_2)$, $\mathbb{P}$-a.s. In what follows, by using the notations introduced in \eqref{iet1}, we denote
 \begin{equation}
 \label{iet2}
 \Lambda_1:=\sup_{\mu \in\,(0,1)}\Lambda_1.	
 \end{equation}
 Due to \eqref{finex}, we have $\Lambda_1<\infty$.

In this section assuming Hypothesis \ref{as4}, we are going to prove that for every $T>0$ there exists a  constant  $c_T>0$ depending on $\Lambda_1$ such that 
\begin{equation}
\label{sg31-tris}
\mathbb{E}\sup_{s \in\,[0,T]}\Vert  u_\mu(s) \Vert _{H}^2 +\int_0^T \mathbb{E}\Vert  u_\mu(s)\Vert _{H^1}^2 ds\leq c_{T}.
\end{equation}

Moreover, we will prove that 
\begin{align}
 \begin{split}
 \label{sa1-4th-mom-H}
 \mu^3\,&\mathbb{E}\sup_{t \in\,[0,T]}\Vert \partial_t u_\mu(t)\Vert_{H}^4+\mu\,\mathbb{E} \sup_{t \in\,[0,T]}\Vert u_\mu(t)\Vert_{H^1}^4+\mu\int_0^T\mathbb{E}\,\Vert u_\mu(t)\Vert_{H^1}^2\Vert \partial_t u_\mu(t)\Vert_{H}^2\,dt\\[10pt]
 &\hsp +\mu^2\int_0^T\mathbb{E}\,\Vert \partial_t u_\mu(t)\Vert_{H}^4\,dt+ \mu\int_0^T \mathbb{E}\Vert  \partial_t u_\mu(s)\Vert _{H}^2 ds\leq c_{T}.\end{split}
 \end{align}

 \begin{Lemma}
 \label{lemma1-bis}
 Assume Hypotheses \ref{as1} to \ref{as4} and fix $(u^\mu_0,v^\mu_0) \in\,\mathcal{H}_2$ satisfying condition \eqref{finex}. Then for every $T>0$ there exists $c_T>0$, depending on $\Lambda_1$, such that for every $\mu \in\,(0,1)$ and $t \in\,[0,T]$
 \begin{align}
 \begin{split}
 \label{sa1-bis}
 \mu^2\,\mathbb{E}&\sup_{s \in\,[0,t]}\Vert \partial_t u_\mu(s)\Vert_H^4+\mathbb{E} \sup_{s \in\,[0,t]}\Vert u_\mu(s)\Vert_{H^1}^4+\int_0^t\mathbb{E}\Vert \partial_t u_\mu(s)\Vert_H^2\,\Vert u_\mu(s)\Vert_H^2\,ds\\[10pt]
 &+\mu\,\int_0^t\mathbb{E}\,\Vert \partial_t u_\mu(s)\Vert_H^4\,ds\leq \frac{c_T}{\mu}\,\left(\int_0^t \mathbb{E}\,\Vert u_\mu(s)\Vert_{H^1}^2\,ds+1\right)+c_T\,	\mu\,\mathbb{E}\sup_{s \in\,[0,t]}\Vert \partial_t u_\mu(s)\Vert_H^2.	
 \end{split}
 \end{align}
 	
 \end{Lemma}

\begin{proof}
If we define $K(v)=\Vert v\Vert^4_H$, we have
\[DK(v)=4\,\Vert v\Vert_H^2 v,\ \ \ \ \ \ D^2 K(v)=8\, v\otimes v+4\,\Vert v\Vert_H^2\,I.\]
In particular, according to \eqref{gsb40}, we get
\begin{align*}
\sum_{k=1}^\infty \langle D^2K(v)\sigma(u)Q e_k,\sigma(u)Q e_k\rangle_H=	8\,\Vert [\sigma(u)Q]^\star v\Vert_H^2+4\,\Vert v\Vert_H^2 \Vert\sigma(u)\Vert^2_{\mathcal{L}_2(H_Q,H)}\leq c\,\Vert v\Vert_H^2.
\end{align*}
Hence, if we apply
It\^o's formula to $K$ and $\partial_t u_\mu(t)$, we get	
\begin{align*}
	d\Vert \partial_t u_\mu(t)\Vert_H^4&\leq \frac 4\mu \Vert \partial_t u_\mu(t)\Vert_H^2\langle \partial_t u_\mu(t),\Delta u_\mu(t)-\gamma(u_\mu(t))\partial_t u_\mu(t)+f(u_\mu(t))\rangle_H\,dt\\[10pt]
	&\hsp+\frac c{\mu^2}\,\Vert \partial_t u_\mu(t)\Vert_H^2\,dt+\frac 4\mu\langle\sigma(u_\mu(t))dw^Q(t),\partial_t u_\mu(t)\rangle_H\Vert \partial_t u_\mu(t)\Vert_H^2\\[10pt]
	&\leq -\frac 2\mu\,\Vert \partial_t u_\mu(t)\Vert_H^2\,d\Vert u_\mu(t)\Vert_{H^1}^2-\frac{4\gamma_0}\mu\,\Vert \partial_t u_\mu(t)\Vert_H^4\,dt+\frac{c}\mu\left(1+\Vert u_\mu(t)\Vert_H\right)\Vert \partial_t u_\mu(t)\Vert_H^3\,dt\\[10pt]
	&\hsp +\frac c{\mu^2}\,\Vert \partial_t u_\mu(t)\Vert_H^2\,dt+\frac {4}\mu\langle\sigma(u_\mu(t))dw^Q(t),\partial_t u_\mu(t)\rangle_H\Vert \partial_t u_\mu(t)\Vert_H^2,
\end{align*}
so that
\begin{align}
\begin{split}   \label{sa2-bis}
	&d\Vert \partial_t u_\mu(t)\Vert_H^4+\frac{2\gamma_0}{\mu}\,\Vert \partial_t u_\mu(t)\Vert_H^4\,dt+\frac 2\mu\,d\left(\Vert \partial_t u_\mu(t)\Vert_H^2\,\Vert u_\mu(t)\Vert_{H^1}^2\right)\leq  \frac 2\mu\,d\,\Vert \partial_t u_\mu(t)\Vert_H^2\,\Vert u_\mu(t)\Vert_{H^1}^2\\[10pt]
	&+\frac{c}\mu\,\left(1+\Vert u_\mu(t)\Vert_H^4\right)\,dt+\frac c{\mu^3}\,dt+\frac 4\mu\langle\sigma(u_\mu(t))dw^Q(t),\partial_t u_\mu(t)\rangle_H\Vert \partial_t u_\mu(t)\Vert_H^2.
	\end{split}
\end{align}
Now, thanks again to the It\^o formula, we have
\begin{align*}
d\,\Vert \partial_t &u_\mu(t)\Vert_H^2=\frac 2\mu\,\langle \partial_t u_\mu(t),\Delta u_\mu(t)-\gamma(u_\mu(t))\partial_t u_\mu(t)+f(u_\mu(t))\rangle_H\,dt\\[10pt]
&\hsp+\frac 1{\mu^2}\,\Vert\sigma(u_\mu(t))\Vert_{\mathcal{L}_2(H_Q,H)}^2\,dt+\frac2\mu\langle \sigma(u_\mu(t))dw^Q(t),\partial_t u_\mu(t)\rangle_H\\[10pt]
&\leq -\frac 1\mu d\,\Vert u_\mu(t)\Vert_{H^1}^2-\frac{\gamma_0}\mu\,\Vert \partial_t u_\mu(t)\Vert_H^2\,dt+\frac{c}\mu\left(1+\Vert u_\mu(t)\Vert_H^2\right)\,dt\\[10pt]
&\hsp +\frac c{\mu^2}\,dt+\frac2\mu\langle \sigma(u_\mu(t))dw^Q(t),\partial_t u_\mu(t)\rangle_H,
\end{align*}
and this gives
\begin{align*}
	\frac 2\mu\,d\,&\Vert \partial_t u_\mu(t)\Vert_H^2\,\Vert u_\mu(t)\Vert_{H^1}^2\leq -\frac 1{\mu^2}\,d\,\Vert u_\mu(t)\Vert_{H^1}^4-\frac{2\gamma_0}{\mu^2}\,\Vert \partial_t u_\mu(t)\Vert_H^2 \Vert u_\mu(t)\Vert_{H^1}^2\,dt\\[10pt]
	&\hslp+\frac{c}{\mu^2}\,dt+\frac{c}{\mu^2}\,\Vert u_\mu(t)\Vert_{H^1}^4\,dt+\frac c{\mu^3}\,\Vert u_\mu(t)\Vert_{H^1}^2\,dt+\frac 4{\mu^2}\langle \sigma(u_\mu(t))dw^Q(t),\partial_t u_\mu(t)\rangle_H\Vert u_\mu(t)\Vert_{H^1}^2.
\end{align*}
Therefore, if we plug the inequality above into \eqref{sa2-bis}, we get
\begin{align}
\begin{split}   \label{sa3-bis}
	&d\Vert \partial_t u_\mu(t)\Vert_H^4+\frac {2}{\mu}\,d\left(\Vert \partial_t u_\mu(t)\Vert_H^2\,\Vert u_\mu(t)\Vert_H^2\right)+\frac 1{\mu^2}\,d\,\Vert u_\mu(t)\Vert_{H^1}^4+\frac{2\,\gamma_0}{\mu}\,\Vert \partial_t u_\mu(t)\Vert_H^4\,dt\\[10pt]
	&\hsllp+\frac{2\gamma_0}{\mu^2}\,\Vert \partial_t u_\mu(t)\Vert_H^2 \Vert u_\mu(t)\Vert_{H^1}^2\,dt\leq\frac{c}{\mu^2}\,\Vert u_\mu(t)\Vert_{H^1}^4\,dt+\frac c{\mu^3}\,\Vert u_\mu(t)\Vert_{H^1}^2\,dt+\frac{c}{\mu^3}\,dt\\[10pt]
	&\hsllp +\frac 4{\mu^2}\langle \sigma(u_\mu(t))dw^Q(t),\partial_t u_\mu(t)\rangle_H\Vert u_\mu(t)\Vert_{H^1}^2+\frac 4\mu\langle\sigma(u_\mu(t))dw^Q(t),\partial_t u_\mu(t)\rangle_H\Vert \partial_t u_\mu(t)\Vert_H^2.  
		\end{split}
\end{align}
Thanks to \eqref{gsb40}, we have
\begin{align*}\frac 4{\mu^2}\,\mathbb{E}&\sup_{s \in\,[0,t]}\left\vert\int_0^s \langle \sigma(u_\mu(r))dw^Q(r),\partial_t u_\mu(r)\rangle_H\Vert u_\mu(s)\Vert_{H^1}^2\right\vert\leq \frac c{\mu^2}\,\mathbb{E}\left(\int_0^t\Vert \partial_t u_\mu(s)\Vert_H^2\Vert u_\mu(s)\Vert_{H^1}^4\,ds\right)^{\frac 12}\\[10pt]
&\leq \frac c{\mu^2}\,	 \mathbb{E}\left(\sup_{s \in\,[0,t]}\Vert u_\mu(s)\Vert_{H^1}^2\int_0^t\Vert \partial_t u_\mu(s)\Vert_H^2\Vert u_\mu(s)\Vert_{H^1}^2\,ds\right)^{1/2}\\[10pt]
&\leq \frac{\gamma_0}{\mu^2}\,\mathbb{E}\int_0^t\Vert \partial_t u_\mu(s)\Vert_H^2\Vert u_\mu(s)\Vert_{H^1}^2\,ds+\frac 1{2\mu^2}\,\mathbb{E}\sup_{s \in\,[0,t]}\Vert u_\mu(s)\Vert_{H^1}^4+\frac c{\mu^2},
\end{align*}
and
\begin{align*}
\frac 4\mu\,\mathbb{E}&\sup_{s \in\,[0,t]}\left\vert\int_0^s\Vert \partial_t u_\mu(r)\Vert_H^2\langle\sigma(u_\mu(r))dw^Q(r),\partial_t u_\mu(r)\rangle_H\right\vert	\\[10pt]
&\leq \frac c\mu\, \mathbb{E}\left(\int_0^t\Vert \partial_t u_\mu(s)\Vert_H^6\,ds\right)^{\frac 12}\leq \frac c\mu\, \mathbb{E}\left(\sup_{s \in\,[0,t]}\Vert \partial_t u_\mu(s)\Vert_H^2\int_0^t\Vert \partial_t u_\mu(s)\Vert_H^4\,ds\right)^{\frac 12}\\[10pt]
&\leq \frac{\gamma_0}{\mu} \mathbb{E}\int_0^t\Vert \partial_t u_\mu(s)\Vert_H^4\,ds+\frac c\mu\,\mathbb{E}\sup_{s \in\,[0,t]}\Vert \partial_t u_\mu(s)\Vert_H^2.
\end{align*}
Hence, if we  first integrate with respect to time both sides in \eqref{sa3-bis}, then  take the supremum with respect to time and finally take the expectation, we get
\begin{align*}
	\mathbb{E}\sup_{s \in\,[0,t]}&\Vert \partial_t u_\mu(s)\Vert_H^4+\frac 1{2\mu^2}\mathbb{E}\sup_{s \in\,[0,t]}\Vert u_\mu(s)\Vert_{H^1}^4+\frac {\gamma_0}{\mu^2}\,\int_0^t\mathbb{E}\Vert \partial_t u_\mu(s)\Vert_H^2\,\Vert u_\mu(s)\Vert_H^2\,ds\\[10pt]
	&\hsl+\frac{\gamma_0}{\mu}\,\int_0^t\mathbb{E}\Vert \partial_t u_\mu(s)\Vert_H^4\,ds\leq \Vert v^\mu_0\Vert_{H}^4+\frac 2\mu\,\Vert u^\mu_0\Vert_{H^1}^2\Vert v^\mu_0\Vert_H^2+\frac 1{\mu^2}\,\Vert u^\mu_0\Vert_{H^1}^4    \\[10pt]
	&+\frac{c}{\mu^2}\int_0^t\mathbb{E}\,\Vert u_\mu(s)\Vert_{H^1}^4\,ds+\frac c{\mu^3}\,\int_0^t\mathbb{E}\,\Vert u_\mu(s)\Vert_{H^1}^2\,dr+\frac{c}{\mu^3}\,t+\frac c\mu\,\mathbb{E}\sup_{s \in\,[0,t]}\Vert \partial_t u_\mu(s)\Vert_H^2.
\end{align*}
Thanks to \eqref{finex} we have
\[\Vert v^\mu_0\Vert_{H}^4+\frac 2\mu\,\Vert u^\mu_0\Vert_{H^1}^2\Vert v^\mu_0\Vert_H^2+\frac 1{\mu^2}\,\Vert u^\mu_0\Vert_{H^1}^4\leq \frac{c}{\mu^2}\,\Lambda_1.\]
Hence, the Gronwall lemma allows to conclude that 
\begin{align*}
	\mathbb{E}\sup_{s \in\,[0,t]}&\Vert \partial_t u_\mu(s)\Vert_H^4+\frac 1{\mu^2}\mathbb{E}\sup_{s \in\,[0,t]}\Vert u_\mu(s)\Vert_{H^1}^4+\frac {1}{\mu^2}\,\int_0^t\mathbb{E}\Vert \partial_t u_\mu(s)\Vert_H^2\,\Vert u_\mu(s)\Vert_H^2\,ds\\[10pt]
	&+\frac{1}{\mu}\,\int_0^t\mathbb{E}\Vert \partial_t u_\mu(s)\Vert_H^4\,ds\leq\frac{c_T}{\mu^3}+\frac{c}{\mu^2}\Lambda_1+\frac {c_T}{\mu^3}\int_0^t\mathbb{E}\,\Vert u_\mu(s)\Vert_{H^1}^2\,ds+\frac {c_T}\mu\,\mathbb{E}\sup_{s \in\,[0,t]}\Vert \partial_t u_\mu(s)\Vert_H^2, 
\end{align*}
and \eqref{sa1-bis} follows, once we multiply both sides by $\mu^2$.

\end{proof}

\begin{Lemma}
\label{prop4.1-tris}
Assume Hypotheses \ref{as1} to  \ref{as4}, and fix $T>0$ and $(u^\mu_0,v^\mu_0) \in\,\mathcal{H}_2$ satisfying \eqref{finex}. Then, there exists $c_{T}>0$, depending on $\Lambda_1$,  such that for every $\mu \in\,(0,1)$ and $t \in\,[0,T]$
\begin{equation}
\label{sg27-tris-bis}
\begin{aligned}
&\mathbb{E}\sup_{s \in\,[0,t]}\Vert  u_\mu(s) \Vert _{H}^2 +\int_0^t \mathbb{E}\Vert  u_\mu(s)\Vert _{H^1}^2 ds\\[10pt]
& \hsp\leq c_{T} \le(1+ \mu \int_0^t \mathbb{E}\Vert  \partial_t u_\mu(s)\Vert _{H^1}^2 ds +\mu^2\,  \mathbb{E}\sup_{s \in\,[0,t]}\Vert  \partial_t u_\mu(s)\Vert _{H}^2\right). 
\end{aligned}
\end{equation}
\end{Lemma}
\begin{proof}
We have
\begin{equation}
\label{E3-tris-bis}
\begin{aligned}
\ds{\langle \eta(u_\mu(t)),\mu\,   \partial_t u_\mu(t)\rangle_H =}  &  \ds{\langle \eta(u_0),\mu\,  v_0\rangle_H +\int_0^t\langle \partial_t \eta(u_\mu(s)),\mu\,\partial_t u_\mu(s)\rangle_H\,ds}\\[10pt]
 &\ds{\hs+\int_0^t\langle \eta(u_\mu(s)),\Delta u_\mu(s)\rangle_H\,ds-\int_0^t\langle \eta(u_\mu(s)),\gamma(u_\mu(s))\partial_t u_\mu(s)\rangle_H\,ds}\\[10pt]
  &\ds{\hs+\int_0^t\langle \eta(u_\mu(s)),f(u_\mu(s))\rangle_H\,ds+\int_0^t\langle \eta(u_\mu(s)),\sigma(u_\mu(s))\,dw^Q(s)\rangle_H}\\[10pt]
 \ds{=:}  &  \ds{\langle \eta(u_0),\mu\,  v_0\rangle_H+\sum_{k=1}^4 I_k(t),}
	\end{aligned} 
 \end{equation}
 where
\[\eta(h):=\gamma^{-1}(h)h,\ \ \ \ h \in\,H,\]

We have
\[\partial_t\,\eta(u_\mu(t))=[D\gamma^{-1}(u_\mu(t))\partial_t u_\mu(t)	]u_{\mu}(t)+\gamma^{-1}(u_\mu(t))\partial_t u_\mu(t).
\]
Then, since
\[D\gamma^{-1}(u_\mu(t))\partial_t u_\mu(t) \in\,\mathbb{R}^{r\times r},\ \ \ \ \ \gamma^{-1}(u_\mu(t))\in\,\mathbb{R}^{r\times r},\]
and
\[D\mathfrak{g}^{-1}(u)=-\mathfrak{g}^{-1}(u)\,D\mathfrak{g}(u)\mathfrak{g}^{-1}(u),\]
 we have
\begin{equation*}
\begin{array}{l}
\ds{|I_1(t)|\leq\mu\,\int_0^t \left|\langle [D\gamma^{-1}(u_\mu(s))\partial_t u_\mu(s)]u_\mu(s),\partial_t u_\mu(s)\rangle_{H}+\langle\gamma^{-1}(u_\mu(s))\partial_t u_\mu(s),\partial_t u_\mu(s)\rangle_{H^1}\right| \,ds}\\[18pt]
\ds{\leq c\,\mu\int_0^t \Vert D\mathfrak{g}(u_\mu(s))\Vert_{\mathcal{L}(H^1,\mathbb{R}^{r\times r})}\Vert \partial_t u_\mu(s)\Vert_{H^1}\Vert \partial_t u_\mu(s)\Vert_{H}\Vert u_\mu(s)\Vert_{H}\,ds+c\,\mu\int_0^t\Vert \partial_t u_\mu(s)\Vert_{H}^2\,ds.}
\end{array}
\end{equation*}
Hence,  for every $\e \in\,(0,1)$ there exists $c_\e>0$ such that 
\begin{align}
	\begin{split}
	\label{sg10-tris}
|I_1(t)|\leq 	\e \,\mu\int_0^t\Vert u_\mu(s)\Vert_{H}^2\Vert \partial_t u_\mu(s)\Vert_{H}^2\,ds+c_\e \,\mu\int_0^t\Vert \partial_t u_\mu(s)\Vert_{H^1}^2\,ds.
	\end{split}
\end{align}
Concerning $I_2(t)$, due to \eqref{xfine160}, we have
\begin{equation}
\label{sg11-tris}
I_2(t)=-\int_0^t\langle \gamma^{-1}(u_\mu(s))\nabla u_\mu(s),\nabla u_\mu(s)\rangle_H\,ds\leq -\tilde{\gamma_0}\int_0^t \Vert u_\mu(s)\Vert^2_{H^1}\,ds.	\end{equation}
As for $I_3(t)$,  we have
\begin{equation}
\label{sg 12-tris}
I_3(t)=	-\int_0^t\langle u_\mu(t),\partial_t u_\mu(s))\rangle_H\,ds=\frac12(-\Vert  u_\mu(t)\Vert_{H}^2+\Vert u^\mu_0\Vert_{H}^2).
\end{equation}
Next, concerning $I_4(t)$,  we have
\begin{equation}
\label{sg21-tris}
|I_4(t)|\leq 	c\int_0^t \Vert u_\mu(s)\Vert_{H}^2\,ds+c\,t.
\end{equation}
Finally,  we have
\begin{equation}
	\label{n11-tris}
\begin{array}{ll}
\ds{\left|\langle\eta(u_\mu(t)),\mu\,   \partial_t u_\mu(t)\rangle_H\right|} & \ds{\leq c\,\mu\,\Vert u_\mu(t)\Vert_{H}\,\Vert \partial_t u_\mu(t)\Vert_{H}\leq \frac 14\,\Vert u_\mu(t)\Vert_{H}^2+c\,\mu^2\,\Vert \partial_tu_\mu(t)\Vert_{H}^2.}
	\end{array}
	\end{equation}

Therefore, collecting together \eqref{sg10-tris}, \eqref{sg11-tris}, \eqref{sg 12-tris}, \eqref{sg21-tris} and \eqref{n11-tris}, from \eqref{E3-tris-bis} we get
\[\begin{array}{l}
\ds{\frac 14\Vert u_\mu(t)\Vert_{H}^2+\tilde{\gamma_0} \int_0^t \Vert u_\mu(s)\Vert_{H^1}^2\,ds}\\[14pt]
\ds{\hslp\leq  \langle \eta(u^\mu_0),\mu\,  v^\mu_0\rangle_H+\frac14\Vert u^\mu_0\Vert_{H}^2+\e\,\mu\int_0^t \,\Vert u_\mu(s)\Vert^2_{H}\Vert \partial_t u_\mu(s)\Vert^2_{H}\,ds+c_\e\, \mu \int_0^t\Vert \partial_t u_\mu(s)\Vert_{H^1}^2\,ds}\\[14pt]
\ds{\hsp+c\int_0^t \Vert u_\mu(s)\Vert_{H}^2\,ds+c\,t+c\,\mu^2\,\Vert \partial_tu_\mu(t)\Vert_{H}^2+\left|\int_0^t\langle \eta(u_\mu(s)),\sigma(u_\mu(s))\,dw^Q(s)\rangle_H\right|,}
\end{array}\]
so that, thanks to Gronwall's lemma and \eqref{finex} there exists $c_T>0$, depending on $\Lambda_1$, such that
\begin{equation}
\label{sg25-tris}
\begin{aligned}
\ds{\frac 12\sup_{s \in\,[0,t]}}&\ds{\Vert u_\mu(s)\Vert^2_{H}+	\frac{\tilde{\gamma_0}}{2}\int_0^t \Vert u_\mu(s)\Vert_{H^1}^2\,ds}\\[10pt]
&\hsl \leq c_{T}+\e\,c_T\,\mu\int_0^t \,\Vert u_\mu(s)\Vert^2_{H}\Vert \partial_t u_\mu(s)\Vert^2_{H}\,ds+c_{T,\e}\, \mu \int_0^t\Vert \partial_t u_\mu(s)\Vert_{H^1}^2\,ds\\[10pt]
&\ds{\hslp +c_T\,\mu^2\sup_{s \in\,[0,t]}\Vert \partial_t u_\mu(s)\Vert_{H}^2+c_T\sup_{s \in\,[0,t]}	\left|\int_0^s\langle \eta(u_\mu(r)),\sigma(u_\mu(r))\,dw^Q(r)\rangle_H\right|.}
\end{aligned}
 \end{equation}
 According to \eqref{gsb40}, 
 we have
\[\begin{aligned}
&\ds{\mathbb{E}\sup_{r \in\,[0,t]}
\left|\int_0\langle \eta(u_\mu(s)),\sigma(u_\mu(s))\,dw^Q(s)\rangle_H\right|}\\[10pt]
& \ds{\hsp\leq  c\,\mathbb{E}\left(\int_0^t\left(1+\Vert u_\mu(s)\Vert_{H}^2\right)\,ds\right)^\frac 12\leq  \frac{\tilde{\gamma_0}}{2}\,\mathbb{E}\int_0^t\Vert u_\mu(s)\Vert_{H^1}^2\,ds+c_T,}	
\end{aligned}\]
and then, if we take  the expectation in \eqref{sg25-tris}, we get
\begin{equation}
\begin{aligned}
\label{sa25-tris}
\ds{\mathbb{E}}&\ds{\sup_{s \in\,[0,t]}\Vert u_\mu(s)\Vert^2_{H}+	\int_0^t \mathbb{E}\Vert u_\mu(s)\Vert_{H^1}^2\,ds\leq c_T+\e\,c_T\,\mu\int_0^t \mathbb{E}\,\Vert u_\mu(s)\Vert^2_{H}\Vert \partial_t u_\mu(s)\Vert^2_{H}\,ds}\\[10pt]
& \hsp\ds{+c_{T,\e}\, \mu \int_0^t\mathbb{E}\,\Vert \partial_t u_\mu(s)\Vert_{H^1}^2\,ds+c_T\,\mu^2\,\mathbb{E}\sup_{s \in\,[0,t]}\Vert \partial_t u_\mu(s)\Vert_{H}^2.}\end{aligned}
 \end{equation}
According to \eqref{sa1-bis},  we have
\begin{align*}
\mu\int_0^t \mathbb{E}\,&\Vert u_\mu(s)\Vert^2_{H}\Vert \partial_t u_\mu(s)\Vert^2_{H}\,ds\leq c_T\,\left(\int_0^t \mathbb{E}\,\Vert u_\mu(s)\Vert_{H^1}^2\,ds+1\right) +c_T\,\mu^2\,\mathbb{E}\sup_{s \in\,[0,t]}\Vert \partial_t u_\mu(s)\Vert_H^2,	
\end{align*}
and if we plug this inequality into \eqref{sa25-tris} we get
\begin{align*}
\mathbb{E}&\sup_{s \in\,[0,t]}\Vert u_\mu(s)\Vert^2_{H}+	\int_0^t \mathbb{E}\Vert u_\mu(s)\Vert_{H^1}^2\,ds\\[10pt]
&\leq c_T+\e\,c_T\int_0^t \mathbb{E}\,\Vert u_\mu(s)\Vert_{H^1}^2\,ds+c_{T,\e}\, \mu \int_0^t\mathbb{E}\,\Vert \partial_t u_\mu(s)\Vert_{H^1}^2\,ds+c_{T,\e}\,\mu^2\,\mathbb{E}\sup_{s \in\,[0,t]}\Vert \partial_t u_\mu(s)\Vert_{H}^2.
 \end{align*} 
 Therefore, if we pick $\bar{\e}\, \in\,(0,1)$ such that $\bar{\e}c_T\leq 1/2$, we get
 \eqref{sg27-tris-bis}.

\end{proof}

\begin{Lemma}
\label{prop4.2-tris}
Under Hypotheses \ref{as1} to \ref{as4}, for every $T>0$ and  $(u^\mu_0, v^\mu_0) \in\,\mathcal{H}_2$ satisfying \eqref{finex},  there exists some constant $c_{T}>0$, depending on $\Lambda_1$, such that for every $ \mu \in\,(0,1)$ and $t \in\,[0,T]$
\begin{equation}
\label{sg30-tris-bis}
\begin{aligned}
\begin{split}
 \mathbb{E}\sup_{s \in\,[0,T]}&\Vert  u_\mu(s) \Vert _{H^1}^2+ \int_0^T \mathbb{E}\Vert  \partial_t u_\mu(s)\Vert _{H}^2 ds +\mu\, \mathbb{E}\sup_{s \in\,[0,T]}\Vert   \partial_t u_\mu(s) \Vert _{H}^2\leq \frac{c_T}\mu.\end{split}\end{aligned}
\end{equation}

\end{Lemma}
\begin{proof}
We have
\[\begin{array}{l}
\ds{\frac 12 d \le(\Vert  u_\mu(t)\Vert _{H^1}^2+\mu\, \Vert  \partial_t u_\mu(t)\Vert _{H}^2 \r)=-\langle \gamma(u_\mu(t))\partial_t u_\mu(t),\partial_tu_\mu(t)\rangle_{H}\,dt}\\[14pt]
\ds{\hsllp+\langle f(u_\mu(t)),\partial_tu_\mu(t)\rangle_{H}\,dt+\frac1{2\mu}\Vert \sigma(u_\mu(t))\Vert^2_{\mathcal{L}_2(H_Q,H)}\,dt+\langle \sigma(u_\mu(t))dw^Q(t),\partial_tu_\mu(t)\rangle_{H}.}\end{array}\]
Due to  \eqref{gsb15} and \eqref{gsb40}, for every $\mu \in\,(0,1)$  this gives
\[    \begin{aligned}
    \ds{\frac 12 d } & \ds{\le(\Vert  u_\mu(t)\Vert _{H^1}^2+\mu\, \Vert  \partial_t u_\mu(t)\Vert _{H}^2 \r) }\\[10pt]
    &\ds{\hsllp \leq -\frac {\gamma_0}2 \Vert  \partial_t u_\mu(t)\Vert _{H}^2\,dt+c\,\Vert  u_\mu(t)\Vert _{H}^2\,dt+\frac c\mu\,dt+\langle \partial_t u_\mu(t),\sigma(u_\mu(t))dw^Q(t)\rangle_{H},}
\end{aligned}
\]
so that, thanks to \eqref{finex},
\begin{equation}  \label{sa26-tris}
    \begin{array}{l}
 \ds{  \Vert  u_\mu(t)\Vert _{H^1}^2+\mu\, \Vert  \partial_t u_\mu(t)\Vert _{H}^2+\gamma_0\int_0^t\Vert  \partial_t u_\mu(s)\Vert _{H}^2\,ds}\\[14pt]
 \ds{\hsp\leq \frac{\Lambda_1}2+\frac{c_T}\mu +c\int_0^t \Vert u_\mu(s)\Vert_{H}^2\,ds+\int_0^t \langle \partial_t u_\mu(s),\sigma(u_\mu(s))dw^Q(s)\rangle_{H}.}
    \end{array}
\end{equation}
In particular, 
since
\begin{align*}  
\mathbb{E}\sup_{s \in\,[0,t]}&\left|\int_0^t \langle \partial_t u_\mu(s),\sigma(u_\mu(s))dw^Q(s)\rangle_{H}\right|\\[10pt]
&\leq c\,\mathbb{E}\left(\int_0^t\Vert \partial_t u_\mu(s)\Vert_{H}^2\,ds\right)^{\frac 12}\leq \frac{\gamma_0}{2}\int_0^t\mathbb{E}\,\Vert  \partial_t u_\mu(s)\Vert _{H}^2\,ds+c_T,	
\end{align*}
by taking the supremum with respect to time and the expectation in both sides of \eqref{sa26-tris}, we obtain
\[\begin{array}{l}
\ds{ \mathbb{E}\sup_{s \in\,[0,t]}\Vert  u_\mu(s)\Vert _{H^1}^2+\mu\, \mathbb{E}\sup_{s \in\,[0,t]}\Vert  \partial_t u_\mu(s)\Vert _{H}^2+\frac{\gamma_0}{2}\int_0^t\mathbb{E}\,\Vert  \partial_t u_\mu(s)\Vert _{H}^2\,ds\leq \frac{c_T}\mu +c\int_0^t \mathbb{E}\Vert u_\mu(s)\Vert_{H}^2\,ds},	
\end{array}	
\]
and the Gronwall lemma allows to obtain \eqref{sg30-tris-bis}.
\end{proof}

 From the combination of Lemma  \ref{prop4.1-tris} and Lemma \ref{prop4.2-tris},  we obtain \eqref{sg31-tris}. Moreover, if we replace \eqref{sg31-tris} and \eqref{sg30-tris-bis} into \eqref{sa1-bis} we obtain \eqref{sa1-4th-mom-H}.

\section{Uniform estimates in $\mathcal{H}_2$}  \label{ssec5.2}

We have different bounds in $\mathcal{H}_2$ depending on the fact that the mapping $\sigma:H^1\to\mathcal{L}_2(H_Q,H^1)$ is bounded or not and for this reason we will have to consider separately the two  cases. 
First, we  construct a family of suitable truncations for the  coefficients $\gamma$ and  $\sigma$ and introduce a family of approximating problems.

For every $R\geq 1$, we fix a continuously differentiable function $\Phi_R:\mathbb{R}\to[0,1]$ having a bounded derivative and   such that 
$\Phi_R(t)=1$, if $t\leq R, $and $, \Phi_R(t)=0$, 
if $t\geq R+1$.
Once introduced $\Phi_R$, we fix
\begin{equation}
\label{xfine130}	
\bar{r} \in\,\left(\frac{2\bar{s}}{1+\bar{s}}\vee (2\bar{\kappa}-3)\vee \varrho,1\right)
\end{equation}
where $\bar{s}<1$ and $\bar{\kappa}<2$ are the constants introduced in  Hypothesis \ref{as4} and $\varrho<1$ is the constant introduced in the statement of Theorem \ref{teo3.4}.
Next,
for every $R\geq 1$ and $h \in\,H^1$ we define
\begin{equation}\label{xsa60}
\gamma_R(h):=\Phi_R(\Vert h\Vert_{H^{\bar{r}}})\,\gamma(h),\ \ \ \ \ \,f_R(h):=\Phi_R(\Vert h\Vert_{H^{\bar{r}}})\,f(h),\ \ \ \ \ \ \  \sigma_R(h):=\Phi_R(\Vert h\Vert_{H^{\bar{r}}})\,\sigma(h).\end{equation}

 Clearly, 
\begin{equation}
	\label{xfine125}
\Vert h\Vert_{H^{\bar{r}}}\leq R\Longrightarrow \gamma_R(h)=\gamma(h),\ f_R(h)=f(h),\ \sigma_R(h)=\sigma(h).	
\end{equation}
Moreover, it is immediate to check that all the mappings $\gamma_R:H^1\to\mathcal{L}(H)$ satisfy Hypothesis \ref{as1} and Condition 1 in Hypothesis \ref{as4}, with
\[\mathfrak{g}_R(h):=\gamma_0+\Phi_R(\Vert h\Vert_{H^{\bar{r}}})\left(\mathfrak{g}(h)
-\gamma_0\right),\\ \ \ \ h \in\,H^1,\] 
and all constants involved  are independent of $R\geq 1$.  In the same way, it is easy to check that all  mappings $\sigma_R:H^1\to\mathcal{L}_2(H_Q,H^1)$ satisfy Hypotheses \ref{as2}, with all constants that are independent of $R\geq 1$. Finally, if $\sigma_R:H^1\to\mathcal{L}_2(H_Q,H^1)$ is not bounded, then $\sigma:H^1\to\mathcal{L}_2(H_Q,H^1)$ is not bounded, so that, due to Condition 3 in Hypothesis \ref{as4}, $\mathfrak{g}$ and $\sigma$ satisfy \eqref{gsb15-tris} and \eqref{xfine113}. We can check that this implies that also $\mathfrak{g}_R$ and $\sigma_R$ satisfy  \eqref{gsb15-tris} and \eqref{xfine113}, with constants independent of $R\geq 1$. Finally, all $f_R$ satisfy Hypothesis \ref{as3}, and Condition 2 in Hypothesis \ref{as4}, with all constants independent of $R\geq 1$.

 	\medskip

 Now, for every $R\geq 1$ and and $\mu>0$, we consider the problem
\begin{equation}\label{SPDE-R}
 \left\{\begin{array}{l}
\ds{	\mu\, \partial_t^2 u^R_\mu(t) = \Delta u^R_\mu (t)-\gamma_R(u^R_\mu (t)) \partial_t u^R_\mu (t) + f_R(u^R_\mu (t))+ \sigma_R(u^R_\mu (t,x))dw_t^Q(t),}\\[14pt]
\ds{u^R_\mu(0,x)=u^\mu_0(x),\ \ \ \ \partial_t u^R_\mu(0,x)=v^\mu_0(x),\ \ \ \ \ \ \ \ \ u^R_\mu(t,x)=0,\ \ \ \ x \in\,\partial \mathcal{O}.}\end{array}\right.
   \end{equation}
   In view of what we have discussed above for $\gamma_R$, $f_R$ and $\sigma_R$,   we have that  for every $\mu>0$ and $(u^\mu_0, v^\mu_0) \in\,\mathcal{H}_2$, and for every $R\geq 1$, $T>0$ and $p\geq 1$ there exists a unique adapted process $u^R_\mu$ taking values in $\mathcal{K}_T$, $\mathbb{P}$-a.s., and solving equation \eqref{SPDE-R}, such that $\boldsymbol{u}^R_\mu \in\,L^2(\Omega;C([0,T];\mathcal{H}_1))$. 
 Moreover, all $\boldsymbol{u}^R_\mu$ satisfy \eqref{sg31-tris}, for some constant $c_T>0$ independent of $R\geq 1$.
 In what follows, we shall denote
 \[\Lambda_2:=\sup_{\mu \in\,(0,1)}\mu^\delta\,\Lambda_2^\mu,\]
 where $\Lambda_2^\mu$ has been defined in \eqref{iet1}. 
 
 \medskip
 In this section, we are going to prove that  if $(u^\mu_0,v^\mu_0)$ satisfy \eqref{finex} and \eqref{xsa11}, then there exists some $c_{T, R}>0$, depending on $\Lambda_1$ and $\Lambda_2$,  such that for every $\mu\in\,(0,1)$ 
 \begin{equation}
\label{sg31-bis}
\mathbb{E}\sup_{s \in\,[0,T]}\Vert  u^R_\mu(s) \Vert _{H^1}^2 +\int_0^T \mathbb{E}\Vert  u^R_\mu(s)\Vert _{H^2}^2 ds\leq c_{T,R},
\end{equation}
and
\begin{equation}
\label{sa29}
\mu\,\mathbb{E}\sup_{s \in\,[0,T]}\Vert  u^R_\mu(s) \Vert _{H^2}^2 +\mu^2\, \mathbb{E}\sup_{s \in\,[0,T]}\Vert   \partial_t u^R_\mu(s) \Vert _{H^1}^2+ \mu\int_0^T \mathbb{E}\Vert  \partial_t u^R_\mu(s)\Vert _{H^1}^2 ds\leq c_{T,R}.	
\end{equation}
We will also prove that for every $\mu \in\,(0,1)$
\begin{equation}
 \label{sg32}\sqrt{\mu}\left(\mathbb{E}\sup_{t \in\,[0,T]}\Vert  u^R_\mu(t)\Vert^2_{H^{1+\bar{r}}}	+\mu\,\mathbb{E}\sup_{t \in\,[0,T]}\Vert \partial_t u^R_\mu(t)\Vert^2_{H^{\bar{r}}}\right)\leq c_{T,R},
\end{equation}
 where $\bar{r}$ is the constant introduced in \eqref{xfine130}.

\subsection{Proof of the $\mathcal{H}_2$-bounds in the case of bounded $\sigma$}
\label{ssec6.1}
Throughout this subsection, we assume that
\begin{equation}\label{xfine127}\sup_{h \in\,H^1}\Vert \sigma(h)\Vert_{\mathcal{L}_2(H_Q,H^1)}<+\infty.\end{equation}
We will prove that in this case $\boldsymbol{u}_\mu$ itself satisfies estimates \eqref{sg31-bis}, and even a stronger version of \eqref{sg32}. Namely 
\begin{equation}
 \label{sg32-bis}\sqrt{\mu}\left(\mathbb{E}\sup_{t \in\,[0,T]}\Vert  u_\mu(t)\Vert^2_{H^{2}}	+\mu\,\mathbb{E}\sup_{t \in\,[0,T]}\Vert \partial_t u_\mu(t)\Vert^2_{H^{1}}\right)\leq c_{T}.
\end{equation}
In particular,  all $\boldsymbol{u}^R_\mu$ satisfy the same bounds as well, for some constant  $c_T>0$ independent of $R\geq 1$.

Moreover, we will also prove  that for all $\mu \in\,(0,1)$
\begin{align}
 \begin{split}
 \label{sa1-final}
 \mu\,&\mathbb{E} \sup_{t \in\,[0,T]}\Vert u_\mu(t)\Vert_{H^2}^4+\mu^2\, \mathbb{E}\sup_{r \in\,[0,T]}\Vert   \partial_t u_\mu(r) \Vert _{H^1}^2+\mu^3\,\mathbb{E}\sup_{t \in\,[0,T]}\Vert \partial_t u_\mu(t)\Vert_{H^1}^4\\[10pt]
 &+ \mu\int_0^T \mathbb{E}\Vert  \partial_t u_\mu(s)\Vert _{H^1}^2 ds+\mu^2\int_0^T\mathbb{E}\,\Vert \partial_t u_\mu(t)\Vert_{H^1}^4\,dt+\mu\int_0^T\mathbb{E}\,\Vert u_\mu(t)\Vert_{H^2}^2\Vert \partial_t u_\mu(t)\Vert_{H^1}^2\,dt\leq c_{T}.\end{split}
 \end{align}

We start with the following lemma, that is an analogous to Lemma \ref{lemma1-bis}.

 \begin{Lemma}
 \label{lemma1}
 Assume Hypotheses \ref{as1} to \ref{as4} and fix $T>0$ and $(u^\mu_0,v^\mu_0) \in\,\mathcal{H}_2$ satisfying \eqref{finex} and \eqref{xsa11}. Then, there exists $c_T>0$, depending on $\Lambda_1$ and $\lambda_2$,  such that for every $\mu \in\,(0,1)$ and $t \in\,[0,T]$
 \begin{align} \begin{split}
 \label{sa1}
 \mu^2\,&\mathbb{E}\sup_{s \in\,[0,t]}\Vert \partial_t u_\mu(s)\Vert_{H^1}^4+\mathbb{E} \sup_{s \in\,[0,t]}\Vert u_\mu(s)\Vert_{H^2}^4+\mu\int_0^t\mathbb{E}\,\Vert \partial_t u_\mu(s)\Vert_{H^1}^4\,ds\\[10pt]
 &+\int_0^t\mathbb{E}\,\Vert u_\mu(s)\Vert_{H^2}^2\Vert \partial_t u_\mu(s)\Vert_{H^1}^2\,ds\leq \frac{c_T}\mu\left(\int_0^t\mathbb{E}\Vert u_\mu(s)\Vert_{H^2}^2\,ds+1\right)+ c_T\mu\,\mathbb{E}\sup_{s \in\,[0,t]}\Vert \partial_t u_\mu(s)\Vert_{H^1}^2. \end{split}\end{align}
 	
 \end{Lemma}

\begin{proof}
If we define $K(v)=\Vert v\Vert^4_{H^1}$, we have
\[DK(v)=4\,\Vert v\Vert_{H^1}^2 (-\Delta)v,\ \ \ \ \ \ D^2 K(v)=8\, (-\Delta)v\otimes (-\Delta)v+4\,\Vert v\Vert_{H^1}^2\,(-\Delta),\]
and hence
\begin{align*}
\sum_{k=1}^\infty D^2K(v)(\sigma(u)Q e_k,\sigma(u)Q e_k)=	8\,\Vert [\sigma(u)Q]^\star (-\Delta)v\Vert_{H}^2+4\,\Vert v\Vert_{H^1}^2 \Vert\sigma(u)\Vert^2_{\mathcal{L}_2(H_Q,H^1)}.
\end{align*}
For every $h \in\,H$, we have
\[\left|\langle [\sigma(u)Q]^\star (-\Delta)v,h\rangle_H\right|=\left|\langle v,[\sigma(u)Q] h\rangle_{H^1}\right|\leq \Vert \sigma(u)\Vert_{\mathcal{L}(H_Q,H^1)}\Vert v\Vert_{H^1}\Vert h\Vert_H,\]
so that, thanks  to \eqref{xfine127}, we get
\[\sup_{u \in\,H^1}\Vert[\sigma(u)Q]^\star (-\Delta)v\Vert_{H}\leq c\,\Vert v\Vert_{H^1},\]
and
\[\sum_{k=1}^\infty D^2K(v)(\sigma(u)Q e_k,\sigma(u)Q e_k)\leq c\,\Vert v\Vert_{H^1}^2.\]

Now, if we apply
It\^o's formula to $K$ and $\partial_t u_\mu(t)$, by proceeding as in the proof of Lemma \ref{lemma1-bis} we have	
\begin{align}
\begin{split}   \label{sa2}
	&d\Vert \partial_t u_\mu(t)\Vert_{H^1}^4+\frac{2\gamma_0}{\mu}\,\Vert \partial_t u_\mu(t)\Vert_{H^1}^4\,dt+\frac 2\mu\,d\left(\Vert \partial_t u_\mu(t)\Vert_{H^1}^2\,\Vert u_\mu(t)\Vert_{H^2}^2\right)\leq  \frac 2\mu\,d\,\Vert \partial_t u_\mu(t)\Vert_{H^1}^2\,\Vert u_\mu(t)\Vert_{H^2}^2\\[10pt]
	&+\frac{c}\mu\,\left(1+\Vert u_\mu(t)\Vert_{H^1}^4\right)\,dt+\frac 4\mu\langle\sigma(u_\mu(t))dw^Q(t),\partial_t u_\mu(t)\rangle_{H^1}\Vert \partial_t u_\mu(t)\Vert_{H^1}^2+\frac c{\mu^3}\,dt.
	\end{split}
\end{align}
Next, as we did in the proof of Lemma \ref{lemma1-bis} for $\Vert \partial_t u_\mu(t)\Vert_{H}^2$, we apply It\^o's formula to $\Vert \partial_t u_\mu(t)\Vert_{H^1}^2$ and we get
\begin{align*}
	\frac 2\mu\,d\,&\Vert \partial_t u_\mu(t)\Vert_{H^1}^2\,\Vert u_\mu(t)\Vert_{H^2}^2\leq -\frac 1{\mu^2}\,d\,\Vert u_\mu(t)\Vert_{H^2}^4-\frac{2\gamma_0}{\mu^2}\,\Vert \partial_t u_\mu(t)\Vert_{H^1}^2 \Vert u_\mu(t)\Vert_{H^2}^2\,dt\\[10pt]
	&\hsllp+\frac{c}{\mu^2}\,\Vert u_\mu(t)\Vert_{H^2}^4\,dt+\frac c{\mu^3}\,\Vert u_\mu(t)\Vert_{H^2}^2\,dt+\frac 4{\mu^2}\langle \sigma(u_\mu(t))dw^Q(t),\partial_t u_\mu(t)\rangle_{H^1}\Vert u_\mu(t)\Vert_{H^2}^2.
\end{align*}
Therefore,  we replace the inequality above into \eqref{sa2} and we get
\begin{align}
\begin{split}   \label{sa3}
	&d\Vert \partial_t u_\mu(t)\Vert_{H^1}^4+\frac {2}{\mu}\,d\left(\Vert \partial_t u_\mu(t)\Vert_{H^1}^2\,\Vert u_\mu(t)\Vert_{H^2}^2\right)+\frac 1{\mu^2}\,d\,\Vert u_\mu(t)\Vert_{H^2}^4\\[10pt]
	&\hslp+\frac{2\gamma_0}{\mu}\,\Vert \partial_t u_\mu(t)\Vert_{H^1}^4\,dt+\frac{2\gamma_0}{\mu^2}\,\Vert \partial_t u_\mu(t)\Vert_{H^1}^2 \Vert u_\mu(t)\Vert_{H^2}^2\,dt\\[10pt]
	&\leq\frac{c}{\mu^2}\,\Vert u_\mu(t)\Vert_{H^2}^4\,dt+\frac c{\mu^3}\,\Vert u_\mu(t)\Vert_{H^2}^2\,dt+\frac c{\mu^3}\,dt\\[10pt]
	&+\frac 4{\mu^2}\langle \sigma(u_\mu(t))dw^Q(t),\partial_t u_\mu(t)\rangle_{H^1}\Vert u_\mu(t)\Vert_{H^2}^2+\frac 4\mu\langle\sigma(u_\mu(t))dw^Q(t),\partial_t u_\mu(t)\rangle_{H^1}\Vert \partial_t u_\mu(t)\Vert_{H^1}^2.  
		\end{split}
\end{align}
Thanks to \eqref{gsb43}, we have
\begin{align*}\frac 4{\mu^2}\,\mathbb{E}&\sup_{s \in\,[0,t]}\left\vert\int_0^s \langle \sigma(u_\mu(r))dw^Q(r),\partial_t u_\mu(r)\rangle_{H^1}\Vert u_\mu(r)\Vert_{H^2}^2\right\vert\\[10pt]
&\leq \frac c{\mu^2}\,\mathbb{E}\left(\int_0^t\Vert \partial_t u_\mu(r)\Vert_{H^1}^2\left(\Vert u_\mu(r)\Vert_{H^1}^2+1\right)\Vert u_\mu(r)\Vert_{H^2}^4\,dr\right)^{\frac 12}\\[10pt]
&\hsl\leq \frac {\gamma_0}{\mu^2}\,	 \int_0^t \mathbb{E}\Vert \partial_t u_\mu(r)\Vert_{H^1}^2\Vert u_\mu(r)\Vert_{H^2}^2\,dr+\frac{c}{\mu^2}\mathbb{E}\sup_{r \in\,[0,t]}\left(\Vert u_\mu(r)\Vert_{H^1}^2+1\right)\Vert u_\mu(r)\Vert_{H^2}^2,\\[10pt]
&\leq \frac {\gamma_0}{\mu^2}\,	 \int_0^t \mathbb{E}\Vert \partial_t u_\mu(r)\Vert_{H^1}^2\Vert u_\mu(r)\Vert_{H^2}^2\,dr+\frac 1{2\mu^2}\mathbb{E}\sup_{r \in\,[0,t]}\mathbb{E}\,\Vert u_\mu(r)\Vert_{H^2}^4+\frac c{\mu^2}.
\end{align*}

Moreover,
\begin{align*}
\frac 4\mu\,\mathbb{E}&\sup_{s \in\,[0,t]}\left\vert\int_0^s\Vert \partial_t u_\mu(r)\Vert_{H^1}^2\langle\sigma(u_\mu(r))dw^Q(r),\partial_t u_\mu(r)\rangle_{H^1}\right\vert	\\[10pt]
&\leq \frac c\mu\, \mathbb{E}\left(\int_0^t\Vert \partial_t u_\mu(r)\Vert_{H^1}^6\left(\Vert u_\mu(r)\Vert_{H^1}^2+1\right)\,dr\right)^{\frac 12}\\[10pt]
&\leq \frac c\mu\, \mathbb{E}\left(\sup_{r \in\,[0,t]}\Vert \partial_t u_\mu(r)\Vert_{H^1}^2\left(\Vert u_\mu(r)\Vert_{H^1}^2+1\right)\int_0^t\Vert \partial_t u_\mu(r)\Vert_{H^1}^4\,dr\right)^{\frac 12}\\[10pt]
&\leq \frac{\gamma_0}{4\mu} \mathbb{E}\int_0^t\Vert \partial_t u_\mu(r)\Vert_{H^1}^4\,dr+\frac c\mu\,\mathbb{E}\sup_{r \in\,[0,t]}\Vert \partial_t u_\mu(r)\Vert_{H^1}^2\left(\Vert u_\mu(r)\Vert_{H^1}^2+1\right),
\end{align*}
so that we have
\begin{align*}
\frac 4\mu\,\mathbb{E}&\sup_{s \in\,[0,t]}\left\vert\int_0^s\Vert \partial_t u_\mu(r)\Vert_{H^1}^2\langle\sigma(u_\mu(r))dw^Q(r),\partial_t u_\mu(r)\rangle_{H^1}\right\vert	\\[10pt]
&\leq \frac{\gamma_0}{\mu} \mathbb{E}\int_0^t\Vert \partial_t u_\mu(r)\Vert_{H^1}^4\,dr+\frac {1}{2}\,\mathbb{E}\sup_{r \in\,[0,t]}\Vert \partial_t u_\mu(r)\Vert_{H^1}^4+\frac c{\mu^2}\,\mathbb{E}\sup_{r \in\,[0,t]}\Vert  u_\mu(r)\Vert_{H^1}^4+\frac c{\mu^2}.
\end{align*}
In particular, if, as we did in the proof of Lemma \ref{lemma1-bis}, we  first integrate with respect to time both sides in \eqref{sa3}, then  take the supremum with respect to time,  take the expectation and finally apply Gronwall's lemma, thanks to \eqref{xsa11} we get
\begin{align}\begin{split}  \label{xsa2}
	\mathbb{E}\sup_{s \in\,[0,t]}&\Vert \partial_t u_\mu(s)\Vert_{H^1}^4+\frac 1{\mu^2}\mathbb{E}\sup_{s \in\,[0,t]}\Vert u_\mu(s)\Vert_{H^2}^4+\frac{1}{\mu}\,\int_0^t\mathbb{E}\Vert \partial_t u_\mu(s)\Vert_{H^1}^4\,ds\\[10pt]
	&\hsl+\frac 1{\mu^2}\int_0^t\mathbb{E}\,\Vert u_\mu(s)\Vert_{H^2}^2\Vert \partial_t u_\mu(s)\Vert_{H^1}^2\,ds\leq \frac{\Lambda_2^2}{\mu^{2(1+\delta)}}+\frac{c_T}{\mu^3}\\[10pt]
	&+ \frac {c_T}{\mu^3}\int_0^t\mathbb{E}\Vert u_\mu(s)\Vert_{H^2}^2\,ds+\frac {c_T}\mu\,\mathbb{E}\sup_{s \in\,[0,t]}\Vert \partial_t u_\mu(s)\Vert_{H^1}^2+\frac {c_T}{\mu^2}\,\mathbb{E}\sup_{s \in\,[0,t]}\Vert  u_\mu(s)\Vert_{H^1}^4. \end{split}
\end{align}
According to \eqref{sa1-bis} and \eqref{sg31-tris} we have
\begin{equation}   \label{xfine140}\mathbb{E}\sup_{s \in\,[0,t]}\Vert  u_\mu(s)\Vert_{H^1}^4\leq \frac{c_T}\mu\left(1+\mu^2\, \mathbb{E}\,\sup_{s \in\,[0,t]}\Vert \partial_t u_\mu(s)\Vert_{H}^2\right).\end{equation}
Thus, recalling that $\delta<1/2$, \eqref{sa1} follows once we replace the inequality above in \eqref{xsa2} and multiply both sides by $\mu^2$.
\end{proof}

\begin{Lemma}
\label{prop4.1-bis}
Assume Hypotheses \ref{as1} to  \ref{as4}, and fix $T>0$ and $(u^\mu_0,v^\mu_0) \in\,\mathcal{H}_2$ satisfying \eqref{finex} and \eqref{xsa11}. Then, there exists $c_{T}>0$, depending on $\Lambda_1$ and $\Lambda_2$,  such that for every $\mu \in\,(0,1)$ and $t \in\,[0,T]$
\begin{equation}
\label{sg27-bis}
\begin{aligned}
&\mathbb{E}\sup_{s \in\,[0,t]}\Vert  u_\mu(s) \Vert _{H^1}^2 +\int_0^t \mathbb{E}\Vert  u_\mu(s)\Vert _{H^2}^2 ds\\[10pt]
& \hsp\leq c_{T} \le(1+  \mu \int_0^t \mathbb{E}\Vert  \partial_t u_\mu(s)\Vert _{H^1}^2 ds +\mu^2\,  \mathbb{E}\sup_{r \in\,[0,t]}\Vert  \partial_t u_\mu(r)\Vert _{H^1}^2\r). 
\end{aligned}
\end{equation}
\end{Lemma}
\begin{proof}
We have
\begin{equation}
\label{E3-bis}
\begin{aligned}
\ds{\langle \eta(u_\mu(t)),\mu\,   \partial_t u_\mu(t)\rangle_{H^1} =}  &  \ds{\langle \eta(u_0),\mu\,  v_0\rangle_{H^1} +\int_0^t\langle \partial_t \eta(u_\mu(s)),\mu\,\partial_t u_\mu(s)\rangle_{H^1}\,ds}\\[10pt]
 &\ds{\hs+\int_0^t\langle \eta(u_\mu(s)),\Delta u_\mu(s)\rangle_{H^1}\,ds-\int_0^t\langle \eta(u_\mu(s)),\gamma(u_\mu(s))\partial_t u_\mu(s)\rangle_{H^1}\,ds}\\[10pt]
  &\ds{\hs+\int_0^t\langle \eta(u_\mu(s)),f(u_\mu(s))\rangle_{H^1}\,ds+\int_0^t\langle \eta(u_\mu(s)),\sigma(u_\mu(s))\,dw^Q(s)\rangle_{H^1}}\\[10pt]
 \ds{=:}  &  \ds{\langle \eta(u_0),\mu\,  v_0\rangle_{H^1}+\sum_{k=1}^4 I_k(t),}
	\end{aligned} 
 \end{equation}
 where, as in Lemma \ref{prop4.1-tris}, we defined 
$\eta(h):=\gamma^{-1}(h)h$ for every $h \in\,H$.

By proceeding as in the proof of Proposition \ref{prop4.1-tris},  for every $\e \in\,(0,1)$ there exists $c_\e>0$ such that 
\begin{equation}
\label{sg10-bis}	
\begin{array}{l}
\ds{|I_1(t)|\leq\mu\,\int_0^t \left|\langle [D\gamma^{-1}(u_\mu(s))\partial_t u_\mu(s)]u_\mu(s),\partial_t u_\mu(s)\rangle_{H^1}+\langle \gamma^{-1}(u_\mu(s))\partial_t u_\mu(s), \partial_t u_\mu(s)\rangle_{H^1}\right| \,ds}\\[18pt]
\ds{\leq 	\e\, \mu^2 \int_0^t\Vert \partial_t u_\mu(s)\Vert_{H^1}^4\,ds+ c_\e\int_0^t\Vert u_\mu(s)\Vert_{H^1}^2\,ds+c\, \mu \int_0^t\Vert \partial_t u_\mu(s)\Vert_{H^1}^2\,ds.}
\end{array}
\end{equation}

For $I_2(t)$ and $I_3(t)$,  we have
\begin{equation}
\label{sg11-bis}
I_2(t)\leq -\tilde{\gamma_0}\int_0^t \Vert u_\mu(s)\Vert^2_{H^2}\,ds,	\end{equation}
and
\begin{equation}
\label{sg 12-bis}
I_3(t)=	-\int_0^t\langle u_\mu(t),\partial_t u_\mu(s))\rangle_{H^1}\,ds=-\Vert  u_\mu(t)\Vert_{H^1}^2+\Vert u^\mu_0\Vert_{H^1}^2.
\end{equation}

For $I_4(t)$, due Hypothesis \ref{as4} we have
\[\begin{array}{l}
\ds{I_4(t)=\int_0^t\langle \eta(u_\mu(s)),f(u_\mu(s))\rangle_{H^1}\,ds=-\int_0^t\langle \gamma^{-1}(u_\mu(s))\nabla u_\mu(s),\nabla [f(u_\mu(s))]\rangle_H\,ds,}
\end{array}\]
and this yields
\begin{equation}
\label{sg21-bis}
|I_4(t)|\leq 	c\int_0^t \Vert u_\mu(s)\Vert_{H^1}^2\,ds+c\,t.
\end{equation}

Therefore, since
\[
\begin{array}{ll}
\ds{\left|\langle\eta(u_\mu(t)),\mu\,   \partial_t u_\mu(t)\rangle_{H^1}\right|\leq \frac 12\,\Vert u_\mu(t)\Vert_{H^1}^2+c\,\mu^2\,\Vert \partial_tu_\mu(t)\Vert_{H^1}^2,}
	\end{array}
	\]
	collecting together \eqref{sg10-bis}, \eqref{sg11-bis}, \eqref{sg 12-bis}, and \eqref{sg21-bis},  from \eqref{E3-bis} we get
\begin{equation}
\label{sg25-bis-bis}\begin{array}{l}
\ds{\frac 12\Vert u_\mu(t)\Vert_{H^1}^2+\tilde{\gamma_0} \int_0^t \Vert u_\mu(s)\Vert_{H^2}^2\,ds}\\[14pt]
\ds{\leq  \langle \eta(u^\mu_0),\mu\,  v^\mu_0\rangle_{H^1}+\Vert u^\mu_0\Vert_{H^1}^2+c\,\mu\int_0^t \,\Vert \partial_t u_\mu(s)\Vert^2_{H^1}\,ds+\e\, \mu^2 \int_0^t\Vert \partial_t u_\mu(s)\Vert_{H^1}^4\,ds}\\[14pt]
\ds{\hsp\hslp+c_\e\int_0^t \Vert u_\mu(s)\Vert_{H^1}^2\,ds+c\,\mu^2\,\Vert \partial_tu_\mu(t)\Vert_{H^1}^2+\left|I_5(t)\right|.}
\end{array}\end{equation}
As a consequence of \eqref{gsb43} and \eqref{sg31-tris}, 
 we have
\begin{align*}
\mathbb{E}&\sup_{r \in\,[0,t]}
\left|I_{5}(t)\right|\leq  c\,\mathbb{E}\left(\int_0^t\left(1+\Vert u_\mu(s)\Vert_{H^1}^2\right)\Vert u_\mu(s)\Vert_{H^1}^2\,ds\right)^\frac 12\leq  \frac 14 \mathbb{E}\sup_{s \in\,[0,t]}\Vert u_\mu(s)\Vert_{H^1}^2+c_T,	
\end{align*}
and then, if we take  first the supremum with respect to $t$ and then the expectation in both sides of \eqref{sg25-bis-bis}, thanks to \eqref{finex}, \eqref{xsa11} and  \eqref{sg31-tris} we get
\begin{equation}
\begin{aligned}
\label{sa25}
\ds{\mathbb{E}}&\ds{\sup_{s \in\,[0,t]}\Vert u_\mu(s)\Vert^2_{H^1}+	\int_0^t \mathbb{E}\Vert u_\mu(s)\Vert_{H^2}^2\,ds}\\[10pt]
& \ds{\leq c_T+c\,\mu\int_0^t \mathbb{E}\Vert\partial_t u_\mu(s)\Vert^2_{H^1}\,ds+\e\, \mu^2 \int_0^t\mathbb{E}\Vert \partial_t u_\mu(s)\Vert_{H^1}^4\,ds+c\,\mu^2\mathbb{E}\sup_{s \in\,[0,t]}\Vert \partial_t u_\mu(s)\Vert_{H^1}^2,}\end{aligned}
 \end{equation}
 for some constant $c_T>0$ depending on $\Lambda_1$ and $\Lambda_2$.
In view of  \eqref{sa1},  this implies that
\begin{align*}
\mathbb{E}&\sup_{s \in\,[0,t]}\Vert u_\mu(s)\Vert^2_{H^1}+	\int_0^t \mathbb{E}\Vert u_\mu(s)\Vert_{H^2}^2\,ds\\[10pt]
&\leq c_T\,	\mu\,	 \int_0^t\mathbb{E}\,\Vert \partial_t u_\mu(s)\Vert_{H^1}^2\,ds+c_T\,\mu^2\,\mathbb{E}\sup_{s \in\,[0,t]}\Vert \partial_t u_\mu(s)\Vert_H^2+\e\,c_T\int_0^t \mathbb{E}\Vert u_\mu(s)\Vert_{H^2}^2\,ds+c_T.
\end{align*}
Thus, if we take $\bar{\e} \in\,(0,1)$ such that $\bar{\e} c_T<1/2$, we get
\begin{align*}
\mathbb{E}&\sup_{s \in\,[0,t]}\Vert u_\mu(s)\Vert^2_{H^1}+	\int_0^t \mathbb{E}\Vert u_\mu(s)\Vert_{H^2}^2\,ds\\[10pt]
&\hsp\leq c_T\,	\mu\,	 \int_0^t\mathbb{E}\,\Vert \partial_t u_\mu(s)\Vert_{H^1}^2\,ds+c_T\,\mu^2\,\mathbb{E}\sup_{s \in\,[0,t]}\Vert \partial_t u_\mu(s)\Vert_H^2+c_T,
\end{align*}
and  \eqref{sg27-bis} follows.

\end{proof}

\begin{Lemma}
\label{prop4.2-bis}
Under Hypotheses \ref{as1} to \ref{as4}, for every $T>0$ and  $(u^\mu_0, v^\mu_0) \in\,\mathcal{H}_2$ satisfying \eqref{finex} and \eqref{xsa11} there exists $c_{T}>0$, depending on $\Lambda_1$ and $\Lambda_2$,  such that for every $\mu \in\,(0,1)$ and $t \in\,[0,T]$
\begin{equation}
\label{sg30-bis}
\begin{aligned}
\begin{split}
 \mathbb{E}\sup_{s \in\,[0,t]}&\Vert  u_\mu(s) \Vert _{H^2}^2 +\mu\, \mathbb{E}\sup_{s \in\,[0,t]}\Vert   \partial_t u_\mu(s) \Vert _{H^1}^2+ \int_0^t \mathbb{E}\Vert  \partial_t u_\mu(s)\Vert _{H^1}^2 ds\leq \frac{c_{T}}\mu.\end{split}\end{aligned}
\end{equation}
\end{Lemma}
\begin{proof}
We have
\[\begin{array}{l}
\ds{\frac 12 d \le(\Vert  u_\mu(t)\Vert _{H^2}^2+\mu\, \Vert  \partial_t u_\mu(t)\Vert _{H^1}^2 \r)=-\langle \gamma(u_\mu(t))\partial_t u_\mu(t),\partial_tu_\mu(t)\rangle_{H^1}\,dt}\\[14pt]
\ds{\hsllp+\langle f(u_\mu(t)),\partial_tu_\mu(t)\rangle_{H^1}\,dt+\frac1{2\mu}\Vert \sigma(u_\mu(t))\Vert^2_{\mathcal{L}_2(H_Q,H^1)}\,dt+\langle \sigma(u_\mu(t))dw^Q(t),\partial_tu_\mu(t)\rangle_{H^1}.}\end{array}\]
For every $\mu \in\,(0,1)$  this gives
\begin{equation}\label{xfine200}    \begin{aligned}
    \ds{\frac 12 d } & \ds{\le(\Vert  u_\mu(t)\Vert _{H^2}^2+\mu\, \Vert  \partial_t u_\mu(t)\Vert _{H^1}^2 \r) }\\[10pt]
    &\ds{\hsllp \leq -\frac {\gamma_0}2 \Vert  \partial_t u_\mu(t)\Vert _{H^1}^2\,dt+c\,\Vert  u_\mu(t)\Vert _{H^1}^2\,dt+\frac c\mu\,dt+\langle \partial_t u_\mu(t),\sigma(u_\mu(t))dw^Q(t)\rangle_{H^1},}
\end{aligned}
\end{equation}
so that, thanks to \eqref{xsa11},
\begin{equation}  \label{sa26}
    \begin{array}{l}
 \ds{  \Vert  u_\mu(t)\Vert _{H^2}^2+\mu\, \Vert  \partial_t u_\mu(t)\Vert _{H^1}^2+\gamma_0\int_0^t\Vert  \partial_t u_\mu(s)\Vert _{H^1}^2\,ds}\\[14pt]
 \ds{\hsp\leq \frac{\Lambda_2}{\mu^\delta}+\frac{c}{\mu} t +c\,\int_0^t \Vert u_\mu(s)\Vert_{H^1}^2\,ds+2\int_0^t \langle \partial_t u_\mu(s),\sigma(u_\mu(s))dw^Q(s)\rangle_{H^1}.}
    \end{array}
\end{equation}
Thanks to \eqref{gsb43}, we have
\begin{align*}  
\mathbb{E}\sup_{s \in\,[0,t]}&\left|\int_0^t \langle \partial_t u_\mu(s),\sigma(u_\mu(s))dw^Q(s)\rangle_{H^1}\right|\leq c\,\mathbb{E}\left(\int_0^t\Vert \partial_t u_\mu(s)\Vert_{H^1}^2\left(1+\Vert u_\mu(s)\Vert_{H^1}^2\right)\,ds\right)^{\frac 12}\\[10pt]
&\leq \frac{\gamma_0}{4}\int_0^t\mathbb{E}\,\Vert  \partial_t u_\mu(s)\Vert _{H^1}^2\,ds+\frac 14\, \mathbb{E}\,\sup_{s \in\,[0,t]}\Vert u_\mu(s)\Vert^2_{H^2}+c\,\mathbb{E}\,\sup_{s \in\,[0,t]}\Vert u_\mu(s)\Vert^2_{H}+c_T.
\end{align*}
Hence, by taking the supremum with respect to time and the expectation in both sides of \eqref{sa26}, due to \eqref{sg31-tris} we obtain
\[\begin{array}{l}
\ds{ \frac 12\,\mathbb{E}\sup_{s \in\,[0,t]}\Vert  u_\mu(s)\Vert _{H^2}^2+\mu\, \mathbb{E}\sup_{s \in\,[0,t]}\Vert  \partial_t u_\mu(s)\Vert _{H^1}^2+\frac{\gamma_0}{2}\int_0^t\mathbb{E}\,\Vert  \partial_t u_\mu(s)\Vert _{H^1}^2\,ds}\\[14pt]
\ds{\hsp\leq \frac{c_T}\mu +c\int_0^t \mathbb{E}\Vert u_\mu(s)\Vert_{H^1}^2\,ds+c\,\mathbb{E}\,\sup_{s \in\,[0,t]}\Vert u_\mu(s)\Vert^2_{H}}\leq c_T,	
\end{array}	
\]
and \eqref{sg30-bis} follows. 
\end{proof}

 From the combination of Propositions  \ref{prop4.1-bis} and \ref{prop4.2-bis},  we obtain that for every $T>0$  there exists a positive constant $c_{T}$, depending on $\Lambda_1$ and $\Lambda_2$, such that
\begin{align*}
\mathbb{E}&\sup_{s \in\,[0,T]}\Vert  u_\mu(s) \Vert _{H^1}^2 +\int_0^T \mathbb{E}\Vert  u_\mu(s)\Vert _{H^2}^2 ds\leq c_{T}, 
	\end{align*}
and  we obtain \eqref{sg31-bis}.
In particular, thanks to \eqref{sg30-bis}, this implies that there exists a positive constant $c_T>0$, depending on $\Lambda_1$ and $\Lambda_2$, such that 
	\eqref{sa29} holds, for all $\mu \in\,(0,1)$.
Finally, if we combine  \eqref{sg31-bis} and \eqref{sa29} together with \eqref{sa1}, we get
 \eqref{sa1-final}.
 
 \medskip
 
 By a  suitable  modification of the argument used in the proof of Proposition 4.4 in \cite{CX} we can prove the following result.
\begin{Proposition}
\label{prop4.4}	Assume Hypotheses \ref{as1} to \ref{as4} and fix $T>0$ and $(u^\mu_0, v^\mu_0) \in\,\mathcal{H}_2$ satisfying \eqref{finex} and \eqref{xsa11}. Then there exists a constant $c_T>0$, depending on $\Lambda_1$ and $\Lambda_2$, such \eqref{sg32} holds.

\end{Proposition}

\begin{proof}
Assume that \eqref{sg32} is not true. Then, there exists a sequence $\{\mu_k\}_{k \in\,\mathbb{N}}\subset (0,1)$ converging to zero such that 
\begin{equation}
\label{n21}
\lim_{k\to \infty}\sqrt{\mu_k}\left(\mathbb{E}\sup_{t \in\,[0,T]}\Vert u_{\mu_k}(t)\Vert^2_{H^2}+\mu_k\,\mathbb{E}\sup_{t \in\,[0,T]}\Vert \partial_t u_{\mu_k}(t)\Vert^2_{H^1}\right)	=\infty.
\end{equation}
For every $k \in\,\mathbb{N}$, we denote
\[L_k(t):=\Vert u_{\mu_k}(t)\Vert^2_{H^2}+\mu_k \,\Vert \partial_t u_{\mu_k}(t)\Vert^2_{H^1},\]
and we fix a random time $\tau_k \in\,[0,T]$ such that
\[L_k(\tau_k)=\max_{t \in\,[0,T]}L_k(t).\]

According to \eqref{xfine200}, we have
\[d L_k(t)\leq c\,\Vert u_{\mu_k}(t)\Vert_{H^1}^2\,dt+\frac {\bar{c}}{\mu_k}\,dt+\langle \sigma(u_{\mu_k}(t))\,dw^Q(t),\partial_t u_{\mu_k}(t)\rangle_{H^1}.\]
Therefore,  for every random time $\sigma \in\,[0,T]$ such that $\mathbb{P}(\sigma\leq \tau_k)=1$, we have
\begin{equation}
\label{n23}	L_k(\tau_k)-L_k(\sigma)\leq c\int_\sigma^{\tau_k}\Vert  u_{\mu_k}(s)\Vert _{H^1}^2\,ds +\frac{\bar{c}}{{\mu_k}}(\tau_k-\sigma)+M_k(\tau_k)-M_k(\sigma),
\end{equation}
where 
\[M_k(t):=\int_0^t\langle \partial_t u_{\mu_k}(s),\sigma(u_{\mu_k}(s))dw^Q(s)\rangle_{H^1}.\]	
Thus, if we 
 define the two random variables
\[U^\star_k:=c\int_0^T\Vert u_{\mu_k}(t)\Vert_{H^1}^2\,dt,\ \ \ \ \ M^\star_k:=2\sup_{t \in\,[0,T]}|M_k(t)|,\]
we have
\begin{equation}  \label{xfine207}L_k(\tau_k)-L_k(\sigma)\leq U^\star_k+M^\star_k+\frac{\bar{c}}{{\mu_k}}(\tau_k-\sigma).\end{equation}

According to   \eqref{xsa11}, we have
\[L_k(0)=\Vert u_0^{\mu_k}\Vert_{H^2}^2+\mu_k\,\Vert v_0^{\mu_k}\Vert_{H^1}^2\leq \frac{\Lambda_2}{\mu_k^\delta}.\]
Then, if we take  $\sigma=0$, we get
\[\tau_k\geq \frac{\mu_k}{\bar{c}}\left(L_k(\tau_k)-L_k(0)-U^\star_k-M^\star_k\right)\geq \frac{\mu_k}{\bar{c}}\left(L_k(\tau_k)-\frac{\Lambda_2}{\mu_k^\delta}-U^\star_k-M^\star_k\right)=:\frac{\mu_k}{\bar{c}}\,\theta_k\]
On the set $E_k:=\{\theta_k>0\}$, if we fix  any $s\in[\tau_k-\frac{\mu_k}{2\bar{c}}\,\theta_k,\tau_k]$,  we have
\[L_k(s)\geq  L_k(\tau_k) -U^\star_k-M^\star_k-\frac{\theta_k}2=\frac{\theta_k}2+\frac{\Lambda_2}{\mu_k^\delta},\]
so that
\[I_k := \int_{0}^{T} L_k(s)\,ds  \geq \int_{\tau_k-\frac{\mu_k}{2\bar{c}}\,\theta_k}^{\tau_k} L_k(s)\,ds 
    \geq \frac{\mu_k}{4\bar{c}}\,\theta_k^2 +\frac{\Lambda_2}{2\bar{c}}\,\mu_k^{1-\delta}\,\theta_k.
\end{equation*}
Thus, by taking expectation on both sides,  we get 
\begin{equation}\label{gx007}
\begin{aligned}
    \mathbb{E}(I_k) &\geq \mathbb{E}(I_k \,;\,E_k)\geq \mathbb{E}\le(\frac{\mu_k}{4\bar{c}}\,\theta_k^2 +\frac{c_0}{2\bar{c}}\,\mu_k^{1-\delta}\,\theta_k ; E_k\r)\geq \frac 1{4\bar{c}}\mathbb{E}\le(\mu_k\,\theta_k^2; E_k\r). 
\end{aligned}
\end{equation}

Now, thanks to \eqref{sg31-tris} and \eqref{sa1-final} we have
\begin{equation}
\label{xfine210}
\sup_{k \in\,\mathbb{N}}\,\mathbb{E}U^\star_k\leq c_T,\ \ \ \ \ \ \sup_{k \in\,\mathbb{N}}\,\mathbb{E}M^\star_k	\leq \frac{c_T}{\sqrt{\mu_k}},	
\end{equation}
and, since we are assuming $\delta<1/2$, \eqref{n21} yields
\[\lim_{k\rightarrow \infty}\sqrt{\mu_k}\,\mathbb{E}\,\theta_k=+\infty.\]
Moreover, 
\begin{equation}\label{gx008}
    \sqrt{\mu_k}\,\mathbb{E}\theta_k\leq \left((\mathbb{E}(\mu_k\,\theta_k^2 \,;\,E_k)\right)^{1/2},
\end{equation}
so that 
\begin{equation*}
    \limsup_{k\to\infty}\mathbb{E}(I_k)\geq \frac{1}{4\bar{c}}\limsup_{k\to\infty}\mathbb{E}(\mu_k\,\theta_k^2 \,;\,E_k)\geq \frac{1}{4\bar{c}}\lim_{k\to\infty}\left(\sqrt{\mu_k}\,\mathbb{E}\,\theta_k\right)^2=+\infty.
\end{equation*}

However, this is not possible because, thanks to \eqref{sg31-bis} and \eqref{sa29}, we have
\[\sup_{k \in\,\mathbb{N}}\mathbb{E} I_k\leq c_T.\] 
This means that 
we get to a contradiction of \eqref{n21} and we conclude that \eqref{sg32} has to be true.

\end{proof}
 
\subsection{Proof of the $\mathcal{H}_2$-bounds in the case of unbounded $\sigma$}
Since we are not assuming that $\sigma:H^1\to\mathcal{L}_2(H_Q,H^1)$ is bounded, we cannot prove Lemma \ref{lemma1}, not even for the truncated problem \eqref{SPDE-R}, so that we have to find a different way to prove an analogous of Propositions \ref{prop4.1-bis} and \ref{prop4.2-bis}, that hold at least for each  truncated problems.

To this purpose,  
we notice that, due to the unboundedness of  $\sigma$, according to Hypothesis \ref{as4} we can assume that
there exists $\bar{s}<1$ such that $\mathfrak{g}:H^{\bar{s}}\to \mathbb{R}^{r\times r}$ is differentiable, with
\begin{equation}
\label{gsb15-4th}
\sup_{h \in\,H^1}\Vert D\mathfrak{g}_R(h)\Vert_{\mathcal{L}(H^{\bar{s}},	\mathbb{R}^{r\times r})}<+\infty,\ \ \ \ \ \inf_{h \in\,H^{\bar{s}}}\,\langle\mathfrak{g}_R(h)\xi,\xi\rangle_{\mathbb{R}^{r}}\geq \gamma_0\,\Vert\xi\Vert_{\mathbb{R}^r}^2,\ \ \ \xi \in\,\mathbb{R}^r. 
\end{equation}

\begin{Lemma}
\label{prop4.1-4th}
Assume Hypotheses \ref{as1} to  \ref{as4}, and fix $R\geq 1$,  $T>0$ and $(u^\mu_0,v^\mu_0) \in\,\mathcal{H}_2$ satisfying \eqref{finex}. Then, there exists $c_{T,R}>0$, depending on $\Lambda_1$ and $\Lambda_2$,  such that for every $\mu \in\,(0,1)$ and $t \in\,[0,T]$
\begin{equation}
\label{sg27-tris}
\begin{aligned}
&\mathbb{E}\sup_{s \in\,[0,t]}\Vert  u^R_\mu(s) \Vert _{H^1}^2 +\int_0^t \mathbb{E}\Vert  u^R_\mu(s)\Vert _{H^2}^2 ds\\[10pt]
& \hsp\leq c_{T,R} \le(1+  \mu \int_0^t \mathbb{E}\Vert  \partial_t u^R_\mu(s)\Vert _{H^1}^2 ds +\mu^2\,  \mathbb{E}\sup_{s \in\,[0,t]}\Vert  \partial_t u^R_\mu(s)\Vert _{H^1}^2\r). 
\end{aligned}
\end{equation}
\end{Lemma}

\begin{proof}
As in the proof of Lemma \ref{prop4.1-bis}, we have
\begin{equation}
\label{E3-tris}
\begin{aligned}
\ds{\langle \eta_R(u^R_\mu(t)),\mu\,   \partial_t u^R_\mu(t)\rangle_{H^1} =}  &  \ds{\langle \eta_R(u^\mu_0),\mu\,  v^\mu_0\rangle_{H^1} +\int_0^t\langle \partial_t \eta_R(u^R_\mu(s)),\mu\,\partial_t u^R_\mu(s)\rangle_{H^1}\,ds}\\[10pt]
 &\ds{\hs+\int_0^t\langle \eta_R(u^R_\mu(s)),\Delta u^R_\mu(s)\rangle_{H^1}\,ds-\int_0^t\langle \eta_R(u^R_\mu(s)),\gamma_R(u^R_\mu(s))\partial_t u^R_\mu(s)\rangle_{H^1}\,ds}\\[10pt]
  &\ds{\hs+\int_0^t\langle \eta_R(u^R_\mu(s)),f_R(u^R_\mu(s))\rangle_{H^1}\,ds+\int_0^t\langle \eta_R(u^R_\mu(s)),\sigma_R(u^R_\mu(s))\,dw^Q(s)\rangle_{H^1}}\\[10pt]
 \ds{=:}  &  \ds{\langle \eta_R(u^\mu_0),\mu\,  v^\mu_0\rangle_{H^1}+\sum_{k=1}^4 I_{R,k}(t),}
	\end{aligned} 
 \end{equation}
 where 
$\eta_R(h):=\gamma_R^{-1}(h)h$, for every $h \in\,H$.

Due to \eqref{gsb15-4th} and  an interpolation argument,  we have
\begin{align}
\begin{split}
\Vert [D\gamma_R^{-1}&(u^R_\mu(s))\partial_t u^R_\mu(s)]u^R_\mu(s)\Vert_{H^1}\leq \Vert D\gamma_R^{-1}(u^R_\mu(s))\Vert_{\mathcal{L}(H^{\bar{s}})}\Vert \partial_t u^R_\mu(s)\Vert_{H^{\bar{s}}}\Vert u^R_\mu(s)\Vert_{H^1}\\[10pt]
&\hslp\leq \Vert D\gamma_R^{-1}(u^R_\mu(s))\Vert_{\mathcal{L}(H^{\bar{s}})}	\Vert \partial_t u^R_\mu(s)\Vert^{\bar{s}}_{H^{1}}\Vert \partial_t u^R_\mu(s)\Vert^{1-\bar{s}}_{H}\Vert u^R_\mu(s)\Vert^{\frac{1-\bar{r}}{2-\bar{r}}}_{H^2}\Vert u^R_\mu(s)\Vert^{\frac{1}{2-\bar{r}}}_{H^{\bar{r}}}.
\end{split}
	\end{align}
Recalling the definition \eqref{xfine130} of $\bar{r}$, we have $\bar{r}\geq 2\bar{s}/(1+\bar{s})$, so that
\[\frac{1-\bar{r}}{2-\bar{r}}\leq \frac{1-\bar{s}}2.\]
Hence, the inequality above  gives
\begin{equation}
\Vert [D\gamma_R^{-1}(u^R_\mu(s))\partial_t u^R_\mu(s)]u^R_\mu(s)\Vert_{H^1}\leq 	c_R\,\Vert \partial_t u^R_\mu(s)\Vert^{\bar{s}}_{H^{1}}\Vert \partial_t u^R_\mu(s)\Vert^{1-\bar{s}}_{H}\Vert u^R_\mu(s)\Vert^{\frac{1-\bar{s}}{2}}_{H^2}.
\end{equation}
In particular, since
we have 
\begin{align*}I_{R,1}(t)&=\mu\,\int_0^t \langle [D\gamma_R^{-1}(u^R_\mu(s))\partial_t u^R_\mu(s)]u^R_\mu(s),\partial_t u^R_\mu(s)\rangle_{H^1}\,ds\\[10pt]
&\hsp\hsp+\langle\gamma_R^{-1}(u^R_\mu(s))\partial_t u^R_\mu(s),\partial_t u^R_\mu(s)\rangle_{H^1} \,ds,\end{align*}
for every $\e>0$ we can find $c_{R, \e}>0$ such that
\begin{align}
\begin{split}
\vert I_{R,1}(t)	\vert&\leq 	c_R\,\mu\int_0^t\Vert \partial_t u^R_\mu(s)\Vert^{1+\bar{s}}_{H^{1}}\Vert \partial_t u^R_\mu(s)\Vert^{1-\bar{s}}_{H}\Vert u^R_\mu(s)\Vert^{\frac{1-\bar{s}}{2}}_{H^2}\,ds+c\,\mu \int_0^t\Vert \partial_t u^R_\mu(s)\Vert^{2}_{H^{1}}\,ds\\[10pt]
&\hsllp\leq \e\int_0^t \Vert u^R_\mu(s)\Vert^{2}_{H^2}\,ds+c_{R, \e}\,\mu^2 \int_0^t \Vert \partial_t u^R_\mu(s)\Vert^{4}_{H}\,ds+c_{R, \e}\,\mu \int_0^t \Vert \partial_t u^R_\mu(s)\Vert^{2}_{H^1}\,ds.
\end{split}
	\end{align}
All the other terms $I_{R, k}(t)$, for $k=2,\ldots,5$ can be treated as in the proof of Lemma \ref{prop4.1-bis}, so that
by proceeding as in the proof of Lemma \ref{prop4.1-bis} we get
\begin{equation}
\begin{aligned}
\label{sa25-bis}
\ds{\mathbb{E}}&\ds{\sup_{s \in\,[0,t]}\Vert u^R_\mu(s)\Vert^2_{H^1}+	\int_0^t \mathbb{E}\Vert u^R_\mu(s)\Vert_{H^2}^2\,ds\leq \e\,c\int_0^t \mathbb{E}\Vert u^R_\mu(s)\Vert_{H^2}^2\,ds+c_T+c\int_0^t \mathbb{E}\,\Vert u^R_\mu(s)\Vert_{H^1}^2\,ds}\\[10pt]
& \ds{+c_{R, \e}\,\mu\int_0^t \mathbb{E}\Vert\partial_t u^R_\mu(s)\Vert^2_{H^1}\,ds+c_{R,\e} \mu^2 \int_0^t\mathbb{E}\Vert \partial_t u^R_\mu(s)\Vert_{H}^4\,ds+c\,\mu^2\,\mathbb{E}\sup_{s \in\,[0,t]}\Vert \partial_t u^R_\mu(s)\Vert_{H^1}^2.}\end{aligned}
 \end{equation}
 This means that if we pick $\bar{\e}>0$ sufficiently small, in view of \eqref{sa1-bis} we have
 \begin{equation}
\begin{aligned}
\label{sa25-4th}
\ds{\mathbb{E}}&\ds{\sup_{s \in\,[0,t]}\Vert u^R_\mu(s)\Vert^2_{H^1}+	\int_0^t \mathbb{E}\Vert u^R_\mu(s)\Vert_{H^2}^2\,ds\leq c_T+c\int_0^t \mathbb{E}\,\Vert u^R_\mu(s)\Vert_{H^1}^2\,ds}\\[10pt]
& \ds{+c_{R, \bar{\e}}\,\mu\int_0^t \mathbb{E}\Vert\partial_t u^R_\mu(s)\Vert^2_{H^1}\,ds+c\,\mu^2\,\mathbb{E}\sup_{s \in\,[0,t]}\Vert \partial_t u^R_\mu(s)\Vert_{H^1}^2,}\end{aligned}
 \end{equation}
and we obtain \eqref{sg27-tris}.

\end{proof}

By proceeding as in the proof of Lemma \ref{prop4.2-bis}, we have that there exists  $c_{T}>0$, depending on $\Lambda_1$ and $\Lambda_2$, such that for every $\mu \in\,(0,1)$ and $t \in\,[0,T]$
\begin{equation}
\label{sg30-tris}
\begin{aligned}
\begin{split}
 \mathbb{E}\sup_{s \in\,[0,T]}&\Vert  u^R_\mu(s) \Vert _{H^2}^2 +\mu\, \mathbb{E}\sup_{s \in\,[0,T]}\Vert   \partial_t u^R_\mu(s) \Vert _{H^1}^2+ \int_0^T \mathbb{E}\Vert  \partial_t u^R_\mu(s)\Vert _{H^1}^2 ds\\[10pt]
 &\hsp \leq \frac{c_{T}}\mu\left(1+\int_0^t \mathbb{E}\Vert u^R_\mu(s)\Vert_{H^1}^2\,ds\right).\end{split}\end{aligned}
\end{equation}
As we did in Subsection \ref{ssec6.1}, from the combination of \eqref{sg27-tris} and \eqref{sg30-tris}, we obtain \eqref{sg31-bis}.

In order to prove \eqref{sg32}, we want to proceed as in the proof of Lemma \ref{prop4.4}. However, we see that one of the fundamental facts in such proof is the fact that $\sigma:H^1\to \mathcal{L}_2(H_Q,H^1)$ is bounded. In order to make up with the fact that here $\sigma:H^1\to \mathcal{L}_2(H_Q,H^1)$ is not bounded, we need the following modification of Lemma \ref{prop4.4}.

\begin{Lemma}
\label{prop4.2-bis-bis}
Assume Hypotheses \ref{as1} to \ref{as4} and $(u^\mu_0, v^\mu_0) \in\,\mathcal{H}_2$ satisfying \eqref{finex} and \eqref{xsa11}. Then, for every $R\geq 1$ and $T>0$ there exists a constant $c_{T, R}>0$, depending on $\Lambda_1$ and $\Lambda_2$, such \eqref{sg32} holds.

\end{Lemma}
\begin{proof}
We have
\[\begin{array}{l}
\ds{\frac 12 d \le(\Vert  u^R_\mu(t)\Vert _{H^{1+\bar{r}}}^2+\mu\, \Vert  \partial_t u^R_\mu(t)\Vert _{H^{\bar{r}}}^2 \r)=-\langle \gamma_R(u^R_\mu(t))\partial_t u^R_\mu(t),\partial_t u^R_\mu(t)\rangle_{H^{\bar{r}}}\,dt}\\[14pt]
\ds{\hsllp+\langle f_R(u^R_\mu(t)),\partial_tu_\mu(t)\rangle_{H^{\bar{r}}}\,dt+\frac1{2\mu}\Vert \sigma_R(u^R_\mu(t))\Vert^2_{\mathcal{L}_2(H_Q,H^{\bar{r}})}\,dt+\langle \sigma_R(u^R_\mu(t))dw^Q(t),\partial_t u^R_\mu(t)\rangle_{H^{\bar{r}}}.}\end{array}\]
Thanks to \eqref{xfine204}, 
for every $u \in\,H^{\bar{r}}$ we have
\[\Vert f_R(u)\Vert_{H^{\bar{r}}}\leq c\,\Phi_R(\Vert u\Vert_{H^{\bar{r}}})\Vert u\Vert_{H^{\bar{r}}}\leq c_R,\]
and
\[\Vert \sigma_R(u)\Vert_{\mathcal{L}_2(H_Q,H^{\bar{r}})}\leq c\,\Phi_R(\Vert u\Vert_{H^{\bar{r}}})\left(1+\Vert u\Vert_{H^{\bar{r}}}\right)\leq c_R,\]
and this allows to obtain
\[   \begin{aligned}
    &\ds{\frac 12 d \le(\Vert  u^R_\mu(t)\Vert _{H^{1+\bar{r}}}^2+\mu\, \Vert  \partial_t u^R_\mu(t)\Vert _{H^{\bar{r}}}^2 \r)}\\[10pt]
    &\hsp\ds{\leq -\frac{\gamma_0}{2}\,\Vert \partial_t u^R_\mu(t)\Vert^2_{H^{\bar{r}}}+\frac{c_R}\mu\,dt+\vert \langle \sigma_R(u^R_\mu(t))dw^Q(t),\partial_t u^R_\mu(t)\rangle_{H^{\bar{r}}}\vert.}
\end{aligned}\]
Thus, if we define
\[L^R_\mu(t):=\Vert  u^R_\mu(t)\Vert _{H^{1+\bar{r}}}^2+\mu\, \Vert  \partial_t u^R_\mu(t)\Vert _{H^{\bar{r}}}^2,
\]
for every $0\leq s\leq t\leq T$ we have
\[L^R_\mu(t)-L^R_\mu(s)\leq \frac{c_R}\mu\,(t-s)+2\,M_{R, \mu}^\star,\]
where
\[M_{\mu, R}^\star:=\sup_{t \in\,[0,T]}\left\vert \int_0^t\langle \sigma_R(u^R_\mu(s))dw^Q(s),\partial_t u^R_\mu(s)\rangle_{H^{\bar{r}}}\right\vert.\]
At this point, the rest proof of the proof proceed exactly as in the proof of Lemma  \ref{prop4.4}.
 
\end{proof}

\section{Uniform bounds in $\mathcal{H}_3$}  \label{ssec5.3}
Here we do not prove uniform bounds in $\mathcal{H}_3$ with respect to $\mu \in\,(0,1)$, as we did for the bounds in $\mathcal{H}_1$ and $\mathcal{H}_2$. Nonetheless,  we are able to control the growth of those bounds with respect to $\mu \in\,(0,1)$ in a proper way. Namely, we are going to prove that \begin{equation}
\label{xsa17}
\mathbb{E}\sup_{s \in\,[0,T]}\Vert u^R(s)\Vert_{H^2}^2 +\int_0^T \mathbb{E}\,\Vert  u^R_\mu(s)\Vert _{H^3}^2\, ds\leq \frac{\rho_{T,R}(\mu)}{\mu^\delta},\ \ \ \ \ \mu \in\,(0,1),	\end{equation}
and
\begin{align}\begin{split} \label{xfine141-final}  &\mu^3\,\mathbb{E}\sup_{s \in\,[0,T]}\Vert \partial_t u^R_\mu(s)\Vert_{H^1}^4+\mu\,\mathbb{E}\sup_{s \in\,[0,T]}\Vert u^R_\mu(s)\Vert_{H^2}^4+\mu^2\,\int_0^T\mathbb{E}\Vert \partial_t u^R_\mu(s)\Vert_{H^1}^4\,ds\\[10pt]
	&\hsp+\mu\int_0^T\mathbb{E}\,\Vert u^R_\mu(s)\Vert_{H^2}^2\Vert \partial_t u^R_\mu(s)\Vert_{H^1}^2\,ds\leq \frac{\rho_{T,R}(\mu)}{\mu^\delta},\ \ \ \ \ \mu \in\,(0,1),
	\end{split}\end{align}
where $\delta \in\,(0,1/2)$ is the constant introduced in \eqref{xsa11} and   $\rho_{T,R}:(0,1)\to  [0,+\infty)$ is some function such that
\begin{equation}
\label{xsa19}
\lim_{\mu\to 0}	\rho_{T,R}(\mu)=0.
\end{equation}
Notice that, unlike in the case of the $\mathcal{H}_2$-bounds, here we  do not need to distinguish the case $\sigma$ is bounded and the case it is not.  

In Lemma \ref{lemma1} we have shown that if $\sigma$ is bounded in $\mathcal{L}_2(H_Q,H^1)$, then $\boldsymbol{u}_\mu$, and hence all $\boldsymbol{u}^R_\mu$ satisfy estimate \eqref{sa1}, which involves  fourth moments of $u_\mu$ and $\partial_t u_\mu$ in $\mathcal{H}_2$. The same bounds cannot be proven in case of unbounded $\sigma$. However, going through all steps of the proof of the lemma, we see that even if $\sigma$ is not bounded,  we still have the inequality
\begin{align}\begin{split}  \label{xsa2-ub}
	\mathbb{E}\sup_{s \in\,[0,t]}&\Vert \partial_t u^R_\mu(s)\Vert_{H^1}^4+\frac 1{\mu^2}\mathbb{E}\sup_{s \in\,[0,t]}\Vert u^R_\mu(s)\Vert_{H^2}^4+\frac{1}{\mu}\,\int_0^t\mathbb{E}\Vert \partial_t u^R_\mu(s)\Vert_{H^1}^4\,ds\\[10pt]
	&\hsl+\frac 1{\mu^2}\int_0^t\mathbb{E}\,\Vert u^R_\mu(s)\Vert_{H^2}^2\Vert \partial_t u^R_\mu(s)\Vert_{H^1}^2\,ds\leq \frac{c}{\mu^{1+\delta}}+\frac {c_T}{\mu^2}\int_0^t\mathbb{E}\,\Vert \sigma_R(u^R_\mu(s))\Vert^2_{\mathcal{L}_2(H_Q,H^1)}\Vert \partial_t u^R_\mu(s)\Vert_{H^1}^2\,ds\\[10pt]
	&
	+\frac {c_T}{\mu^3}\int_0^t\mathbb{E}\,\Vert \sigma_R(u^R_\mu(s))\Vert^2_{\mathcal{L}_2(H_Q,H^1)}\Vert u^R_\mu(s)\Vert_{H^2}^2\,ds\\[10pt]
	&\hsp+\frac{c_T}{\mu^2}+\frac {c_T}\mu\,\mathbb{E}\sup_{s \in\,[0,t]}\Vert \partial_t u^R_\mu(s)\Vert_{H^1}^2+\frac {c_T}{\mu^2}\,\mathbb{E}\sup_{s \in\,[0,t]}\Vert  u^R_\mu(ss)\Vert_{H^1}^4.\\[10pt]
	& \end{split}
\end{align}

Now, by interpolation for every $u \in\,H^2$ and $v \in\,H^1$ we have
\begin{align}\begin{split}  \label{xfine133}
\Vert \sigma_R&(u)\Vert^2_{\mathcal{L}_2(H_Q,H^1)}\Vert v\Vert_{H^1}^2\leq c\,\Vert v\Vert_{H^1}^2+c\,\Phi^2_R(	\Vert u\Vert_{H^{\bar{r}}})\,\Vert u\Vert^2_{H^1}\Vert v\Vert_{H^1}^2\\[10pt]
&
\hsp\leq c\,\Vert v\Vert_{H^1}^2+c_R\,\Vert u\Vert_{H^2}\Vert v\Vert_{H^1}^2\leq c_R\,\Vert v\Vert_{H^1}^2+\frac 12 \,\Vert u\Vert^2_{H^2}\Vert v\Vert^2_{H^1},\end{split}
\end{align}
and if $u \in\,H^3$, for every $\e>0$, 
\begin{align*}
\Vert\sigma_R(u)&\Vert^2_{\mathcal{L}_2(H_Q,H^1)}\Vert u\Vert_{H^2}^2\leq c\,\Vert u\Vert_{H^2}^2+c\,\Phi^2_R(	\Vert u\Vert_{H^{\bar{r}}})\,\Vert u\Vert^2_{H^1}\Vert u\Vert_{H^2}^2\\[10pt]
&\leq c\,\Vert u\Vert_{H^2}^2+\Phi_R(	\Vert u\Vert_{H^{\bar{r}}})\,\Vert u\Vert^{}_{H^3}\Vert u\Vert_{H^1}^3\leq c\,\Vert u\Vert_{H^2}^2+\e\,\Vert u\Vert^{2}_{H^3}+c_{R,\e}\,\Vert u\Vert^{\frac{6(1-\bar{r})}{2-\bar{r}}}_{H^2}.
\end{align*}
Since without any loss of generality we can assume that $\bar{r}>1/2$, we have that $6(1-\bar{r})/(2-\bar{r})\leq 2$, so that we conclude
\[\Vert\sigma_R(u)\Vert^2_{\mathcal{L}_2(H_Q,H^1)}\Vert u\Vert_{H^2}^2\leq c_{R,\e}\,\Vert u\Vert_{H^2}^2+\e\,\Vert u\Vert^{2}_{H^3}.\]
By using this and \eqref{xfine133} into \eqref{xsa2-ub}, we obtain
\begin{align*}\begin{split}  \mathbb{E}\sup_{s \in\,[0,t]}&\Vert \partial_t u^R_\mu(s)\Vert_{H^1}^4+\frac 1{\mu^2}\mathbb{E}\sup_{s \in\,[0,t]}\Vert u^R_\mu(s)\Vert_{H^2}^4+\frac{1}{\mu}\,\int_0^t\mathbb{E}\Vert \partial_t u^R_\mu(s)\Vert_{H^1}^4\,ds\\[10pt]
	&\hsl+\frac 1{2\mu^2}\int_0^t\mathbb{E}\,\Vert u^R_\mu(s)\Vert_{H^2}^2\Vert \partial_t u^R_\mu(s)\Vert_{H^1}^2\,ds\leq \frac {c_{T,R}}{\mu^2}\int_0^t\mathbb{E}\,\Vert \partial_t u^R_\mu(s)\Vert_{H^1}^2\,ds+\frac {c_{T,R,\e}}{\mu^3}\int_0^t\mathbb{E}\,\Vert u^R_\mu(s)\Vert_{H^2}^2\,ds\\[10pt]
	&
	+\e\,\frac {c_T}{\mu^3}\int_0^t\mathbb{E}\,\Vert u^R_\mu(s)\Vert^{2}_{H^3}\,ds+\frac{c_T}{\mu^2}+\frac {c_T}\mu\,\mathbb{E}\sup_{s \in\,[0,t]}\Vert \partial_t u^R_\mu(s)\Vert_{H^1}^2+\frac {c_T}{\mu^2}\,\mathbb{E}\sup_{s \in\,[0,t]}\Vert  u^R_\mu(s)\Vert_{H^1}^4+\frac{c}{\mu^{1+\delta}}.\\[10pt]
	& \end{split}
\end{align*}
Therefore, recalling \eqref{sa1-4th-mom-H} and \eqref{sg31-bis}, after we multiply both sides by $\mu^3$ we get
\begin{align}\begin{split} \label{xfine141}  &\mu^3\,\mathbb{E}\sup_{s \in\,[0,t]}\Vert \partial_t u^R_\mu(s)\Vert_{H^1}^4+\mu\,\mathbb{E}\sup_{s \in\,[0,t]}\Vert u^R_\mu(s)\Vert_{H^2}^4+\mu^2\,\int_0^t\mathbb{E}\Vert \partial_t u^R_\mu(s)\Vert_{H^1}^4\,ds\\[10pt]
	&\hsp+\mu\int_0^t\mathbb{E}\,\Vert u^R_\mu(s)\Vert_{H^2}^2\Vert \partial_t u^R_\mu(s)\Vert_{H^1}^2\,ds\\[10pt]
	&\leq c_{T,R,\e}+c_{T,R}\left(1+\mu\int_0^t\mathbb{E}\,\Vert \partial_t u^R_\mu(s)\Vert_{H^1}^2\,ds+\mu^2\,\mathbb{E}\sup_{s \in\,[0,t]}\Vert \partial_t u^R_\mu(s)\Vert_{H^1}^2\right)+\e\,c_T\int_0^t\mathbb{E}\,\Vert u^R_\mu(s)\Vert^{2}_{H^3}\,ds. \end{split}
\end{align}

\begin{Lemma}
\label{prop4.1-4th}
Assume Hypotheses \ref{as1} to  \ref{as4}, and fix $R\geq 1$ and  $(u^\mu_0,v^\mu_0) \in\,\mathcal{H}_3$ satisfying \eqref{finex} and \eqref{xsa11}. Then, for every $T>0$ there exists $c_{T}>0$ such that for every $\mu \in\,(0,1)$ and $t \in\,[0,T]$
\begin{equation}
\label{sg27-4th}
\begin{aligned}
&\mathbb{E}\sup_{s \in\,[0,t]}\Vert  u^R_\mu(s) \Vert _{H^2}^2 +\int_0^t \mathbb{E}\,\Vert  u^R_\mu(s)\Vert _{H^3}^2\, ds\\[10pt]
& \hsp\leq \frac{\rho_{1}(\mu)}{\mu^\delta}+c_{T}\,\mu \le(\int_0^t \mathbb{E}\,\Vert  \partial_t u^R_\mu(s)\Vert _{H^2}^2 ds +\mu\,  \mathbb{E}\sup_{s \in\,[0,t]}\Vert  \partial_t u^R_\mu(s)\Vert _{H^2}^2\r), 
\end{aligned}
\end{equation}
for some $\rho_{1}(\mu)>0$, depending on $\Lambda_1$, $\Lambda^\mu_2$ and $\Lambda^\mu_3$, such that
\begin{equation}
\label{xsa12}
\lim_{\mu\to 0}\rho_{1}(\mu)=0.	
\end{equation}

\end{Lemma}

\begin{proof}
As in the proof of Lemma \ref{prop4.1-tris} and Lemma \ref{prop4.1-bis}, we have
\begin{align}\begin{split}
\label{E3-4th}
\langle \eta&(u^R_\mu(t)),\mu\,   \partial_t u^R_\mu(t)\rangle_{H^2} =\langle \eta(u_0),\mu\,  v_0\rangle_{H^2} +\int_0^t\langle \partial_t \eta(u^R_\mu(s)),\mu\,\partial_t u^R_\mu(s)\rangle_{H^2}\,ds\\[10pt]
 &\hsllp+\int_0^t\langle \eta(u^R_\mu(s)),\Delta u^R_\mu(s)\rangle_{H^2}\,ds-\int_0^t\langle \eta(u^R_\mu(s)),\gamma(u^R_\mu(s))\partial_t u^R_\mu(s)\rangle_{H^2}\,ds\\[10pt]
  &\hsllp+\int_0^t\langle \eta(u^R_\mu(s)),f(u^R_\mu(s))\rangle_{H^2}\,ds+\int_0^t\langle \eta(u^R_\mu(s)),\sigma_R(u_\mu(s))\,dw^Q(s)\rangle_{H^2}\\[10pt]
 &\hsp\hsp=:\langle \eta(u_0),\mu\,  v_0\rangle_{H^1}+\sum_{k=1}^4 I^R_k(t).\end{split}
	\end{align} 
 
For $I^R_1$, we have
\begin{equation}
\label{sg10-4th}	
\begin{array}{ll}
\ds{|I^R_1(t)|} & \ds{\leq 	c\,\mu \int_0^t\Vert \partial_t u^R_\mu(s)\Vert_{H^1}\,\Vert \partial_t u^R_\mu(s)\Vert_{H^2}\,\Vert  u^R_\mu(s)\Vert_{H^2}\,ds+ c\,\mu\int_0^t\Vert \partial_t u^R_\mu(s)\Vert_{H^2}^2\,ds}\\[18pt]
&\ds{\hslp\leq 	c\,\mu \int_0^t\Vert \partial_t u^R_\mu(s)\Vert_{H^1}^2\,\Vert  u^R_\mu(s)\Vert_{H^2}^2\,\ds+c\,\mu\int_0^t\Vert   \partial_t u^R_\mu(s)\Vert_{H^2}^2\,ds.}
\end{array}
\end{equation}
For $I^R_2(t)$ and $I^R_3(t)$,  we have
\begin{equation}
\label{sg11-4th}
I^R_2(t)\leq -\tilde{\gamma_0}\int_0^t \Vert u^R_\mu(s)\Vert^2_{H^3}\,ds,	\end{equation}
and
\begin{equation}
\label{sg 12-4th}
I^R_3(t)=	-\int_0^t\langle u^R_\mu(t),\partial_t u^R_\mu(s))\rangle_{H^2}\,ds=-\Vert  u^R_\mu(t)\Vert_{H^2}^2+\Vert u^\mu_0\Vert_{H^2}^2.
\end{equation}
For $I^R_4(t)$, we have
\begin{align*}\begin{split}
|I^R_4(t)|\leq c\,t&+c\int_0^t \Phi_R(	\Vert u^R_\mu(r)\Vert_{H^{\bar{r}}})\Vert u^R_\mu(s)\Vert_{H^2}^3\,ds\leq c\,t+c_R \int_0^t \Vert u^R_\mu(s)\Vert_{H^3}^{\frac{3(2-\bar{r})}{3-\bar{r}}}\,ds.
\end{split} \end{align*}
Then, since $3(2-\bar{r})/(3-\bar{r})<2$,  we have
\begin{equation}
\label{sg21-4th}
|I^R_4(t)|\leq c_{T,R}+	\frac{\tilde{\gamma_0}}{2} \int_0^t \Vert u^R_\mu(s)\Vert_{H^3}^{2}\,ds
\end{equation}
Finally, we have 
\begin{equation}  \label{xsa10}
\left|\langle\eta(u^R_\mu(t)),\mu\,   \partial_t u^R_\mu(t)\rangle_{H^2}\right|\leq \frac 12\,\Vert u^R_\mu(t)\Vert_{H^2}^2+c\,\mu^2\,\Vert \partial_tu^R_\mu(t)\Vert_{H^2}^2.
	\end{equation}
Therefore, due to	\eqref{sg10-4th}, \eqref{sg11-4th}, \eqref{sg 12-4th}, \eqref{sg21-4th} and \eqref{xsa10},  if we integrate both sides in  \eqref{E3-4th} and then take the expectation, we get
\begin{equation}
\label{sg25-bis}\begin{array}{l}
\ds{\frac 12\,\mathbb{E}\,\Vert u^R_\mu(t)\Vert_{H^2}^2+\frac{\tilde{\gamma_0}}{2} \int_0^t \mathbb{E}\,\Vert u^R_\mu(s)\Vert_{H^3}^2\,ds\leq  \langle \eta(u^\mu_0),\mu\,  v^\mu_0\rangle_{H^2}+\Vert u^\mu_0\Vert_{H^2}^2+c_{T,R}}\\[14pt]
\ds{+c\,\mu\int_0^t \mathbb{E}\,\Vert \partial_t u^R_\mu(s)\Vert^2_{H^2}\,ds+c\,\mu \int_0^t\mathbb{E}\,\Vert \partial_t u^R_\mu(s)\Vert_{H^1}^2\,\Vert \partial_t u^R_\mu(s)\Vert_{H^2}^2\,\ds+c\,\mu^2\,\mathbb{E}\,\Vert \partial_tu^R_\mu(t)\Vert_{H^2}^2.}\end{array}\end{equation}
In view of \eqref{xsa11} we have
\[\mu^\delta\left(\langle \eta(u^\mu_0),\mu\,  v^\mu_0\rangle_{H^2}+\Vert u^\mu_0\Vert_{H^2}^2\right)\leq \left(\Lambda_3^\mu \Lambda_2^\mu\right)^{1/2}+\Lambda_2^\mu=:\rho_{1}(\mu)\to 0,\ \ \ \ \ \text{as}\ \mu\to 0.\]
Therefore,  if in \eqref{xfine141} we  fix $\bar{\e}>0$ sufficiently small so that 
$\bar{\e} c_T\leq \tilde{\gamma_0}/4$ we have
\begin{align*}
&\frac 12\,\mathbb{E}\,\Vert u^R_\mu(t)\Vert_{H^2}^2+\frac{\tilde{\gamma_0}}{4} \int_0^t \mathbb{E}\,\Vert u^R_\mu(s)\Vert_{H^3}^2\,ds\\[10pt]	
&\hsp\leq \frac{\rho_1(\mu)}{\mu^\delta}+ c_T\left(1+	 \mu\int_0^t\mathbb{E}\,\Vert \partial_t u^R_\mu(s)\Vert_{H^2}^2\,ds+\mu^2\,\mathbb{E}\sup_{s \in\,[0,t]}\Vert \partial_t u^R_\mu(s)\Vert_{H^2}^2\right),
\end{align*}
and \eqref{sg27-4th} follows.

\end{proof}

In what follows, we shall denote
\[\Lambda_3:=\sup_{\mu \in\,(0,1)} \mu^{1+\delta}\,\Lambda_3^\mu.\]
\begin{Lemma}
\label{prop4.2-4th}
Assume Hypotheses \ref{as1} to \ref{as4}, and fix $R\geq 1$ and  $(u^\mu_0, v^\mu_0) \in\,\mathcal{H}_3$ satisfying \eqref{finex} and \eqref{xsa11}. Then, for every $T>0$ there exists $c_{T}>0$, depending on $\Lambda_1$, $\Lambda_2$ and $\Lambda_3$,  such that for every $\mu \in\,(0,1)$
\begin{equation}
\label{sg30-R}
 \begin{array}{l}
 \ds{\mathbb{E}\sup_{s \in\,[0,T]}\Vert  u^R_\mu(s) \Vert _{H^3}^2 +\mu\, \mathbb{E}\sup_{s \in\,[0,T]}\Vert   \partial_t u^R_\mu(s) \Vert _{H^2}^2+ \int_0^T \mathbb{E}\Vert  \partial_t u^R_\mu(s)\Vert _{H^2}^2 ds}\\[18pt]
 \ds{\hsp\hslp\leq \frac{\rho_{2,T,R}(\mu)}{\mu^{1+\delta}}+\frac{c_R}{\mu}\,\int_0^t \mathbb{E}\,\Vert u^R_\mu(s)\Vert_{H^3}^{\frac{4(\bar{\kappa}-\bar{r})}{3-\bar{r}}}\,ds,}	
 \end{array}
\end{equation} for some $\rho_{2,T,R}(\mu)>0$, depending on $\Lambda_1$, $\Lambda^\mu_2$ and $\Lambda^\mu_3$, such that
 \begin{equation}
 \label{xsa16}
 \lim_{\mu\to 0}\rho_{2,T,R}(\mu)=0.
\end{equation}
\end{Lemma}
\begin{proof}
Here, we assume that $\sigma$ is not bounded, so that $f$ satisfies \eqref{xfine120-f}. We omit the proof in the case $\sigma$ is bounded; it  is analogous to the proof below and relies on estimate \eqref{xfine141} instead of \eqref{xfine120-f}. We have
\[\begin{array}{l}
\ds{\frac 12 d \le(\Vert  u^R_\mu(t)\Vert _{H^3}^2+\mu\, \Vert  \partial_t u^R_\mu(t)\Vert _{H^2}^2 \r)=-\langle \gamma_R(u^R_\mu(t))\partial_t u^R_\mu(t),\partial_tu^R_\mu(t)\rangle_{H^2}\,dt}\\[14pt]
\ds{\hsllp+\langle f_R(u^R_\mu(t)),\partial_tu^R_\mu(t)\rangle_{H^2}\,dt+\frac1{2\mu}\Vert \sigma_R(u^R_\mu(t))\Vert^2_{\mathcal{L}_2(H_Q,H^2)}\,dt+\langle \sigma_R(u^R_\mu(t))dw^Q(t),\partial_tu^R_\mu(t)\rangle_{H^2}.}\end{array}\]
Now, due to \eqref{xfine120} and \eqref{xfine120-f}, for every $\mu \in\,(0,1)$ we have
\begin{align*}
\Vert f_R(u)\Vert_{H^2}^2+\frac1{2\mu}\Vert \sigma_R(u)\Vert^2_{\mathcal{L}_2(H_Q,H^2)}	\leq \frac{c_R}{\mu}\left(1+\Vert u\Vert_{H^2}^2+\Vert u\Vert_{H^3}^{\frac{4(\bar{\kappa}-\bar{r})}{3-\bar{r}}}\right),
\end{align*}
so that
\[    \begin{aligned}
    \ds{\frac 12 d } & \ds{\le(\Vert  u^R_\mu(t)\Vert _{H^3}^2+\mu\, \Vert  \partial_t u^R_\mu(t)\Vert _{H^2}^2 \r)\leq -\frac {\gamma_0}2 \Vert  \partial_t u^R_\mu(t)\Vert _{H^2}^2\,dt+ \Vert  u^R_\mu(t)\Vert _{H^2}^{4}\,dt}\\[10pt]
    &\ds{ +\frac{c_R}{\mu}\,\Vert  u^R_\mu(t)\Vert _{H^2}^{2}\,dt+\frac {c_R}\mu\,dt+\frac{c_R}{\mu}\Vert  u^R_\mu(t)\Vert _{H^3}^{\frac{4(\bar{\kappa}-\bar{r})}{3-\bar{r}}}\,dt+\langle \partial_t u^R_\mu(t),\sigma_R(u^R_\mu(t))dw^Q(t)\rangle_{H^1}.}
\end{aligned}
\]
If we integrate both sides with respect to time, thanks to \eqref{xsa11} we get 
\begin{equation}  \label{sa26-R}
    \begin{array}{l}
 \ds{  \Vert  u^R_\mu(t)\Vert _{H^3}^2+\mu\, \Vert  \partial_t u^R_\mu(t)\Vert _{H^2}^2+\gamma_0\int_0^t\Vert  \partial_t u^R_\mu(s)\Vert _{H^2}^2\,ds\leq \frac{\Lambda_3^\mu\,\mu^{1+\delta}+c_{T,R}\,\mu^\delta}{\mu^{1+\delta}}}\\[14pt]
 \ds{\hslp+\frac {c_R}\mu\int_0^t \Vert u^R_\mu(s)\Vert_{H^2}^{2}\,ds+\frac{c_R}\mu\int_0^t \Vert u^R_\mu(s)\Vert_{H^3}^{\frac{4(\bar{\kappa}-\bar{r})}{3-\bar{r}}}\,ds+2\int_0^t \langle \partial_t u^R_\mu(s),\sigma_R(u^R_\mu(s))dw^Q(s)\rangle_{H^2}.}
    \end{array}
\end{equation}
In view again of  \eqref{xfine120}, we have
\begin{align*}  
\mathbb{E}\sup_{s \in\,[0,t]}&\left|\int_0^t \langle \partial_t u^R_\mu(s),\sigma_R(u^R_\mu(s))dw^Q(s)\rangle_{H^2}\right|\leq c_R\,\mathbb{E}\left(\int_0^t\Vert \partial_t u^R_\mu(r)\Vert_{H^2}^2\left(1+\Vert u^R_\mu(s)\Vert_{H^2}^{2}\right)\,dr\right)^{\frac 12}\\[10pt]
&\hsp\leq \frac {\mu}4\, \mathbb{E}\,\sup_{s \in\,[0,t]}\Vert \partial_t u^R_\mu(s)\Vert^2_{H^2}+\frac{c_R}{\mu}\,\int_0^t \mathbb{E}\,\Vert u^R_\mu(s)\Vert_{H^3}^{\frac{4(\bar{\kappa}-\bar{r})}{3-\bar{r}}}\,ds+\frac{c_T}{\mu}.
\end{align*}
Hence, by taking the supremum with respect to time and the expectation in both sides of \eqref{sa26-R}, according to \eqref{sg31-bis}, we obtain
\[\begin{array}{l}
\ds{ \mathbb{E}\sup_{s \in\,[0,t]}\Vert  u^R_\mu(s)\Vert _{H^3}^2+\mu\, \mathbb{E}\sup_{s \in\,[0,t]}\Vert  \partial_t u^R_\mu(s)\Vert _{H^2}^2+\int_0^t\mathbb{E}\,\Vert  \partial_t u^R_\mu(s)\Vert _{H^2}^2\,ds }\\[14pt]
\ds{\hsp\hslp\leq \frac{1}{\mu^{1+\delta}}\,\left(\Lambda_3^\mu\,\mu^{1+\delta}+c_{T,R}\,\mu^\delta\right)+\frac{c_R}{\mu}\,\int_0^t \mathbb{E}\,\Vert u^R_\mu(s)\Vert_{H^3}^{\frac{4(\bar{\kappa}-\bar{r})}{3-\bar{r}}}\,ds.	}
\end{array}	
\]
This implies \eqref{sg30-R}, with 
$\rho_{2, T,R}(\mu):=\Lambda_3^\mu\,\mu^{1+\delta}+c_{T, R}\,\mu^\delta$.
\end{proof}

Now, thanks to \eqref{sg27-4th}, we have
\begin{align*}
&\mathbb{E}\sup_{r \in\,[0,t]}\Vert u^R(s)\Vert_{H^2}^2 +\int_0^t \mathbb{E}\,\Vert  u^R_\mu(s)\Vert _{H^3}^2\, ds\\[10pt]
&\hsp \leq \frac{\rho_{1}(\mu)}{\mu^\delta}+c_{T}\,\mu \le(\int_0^t \mathbb{E}\,\Vert  \partial_t u^R_\mu(s)\Vert _{H^2}^2 ds +\mu\,  \mathbb{E}\sup_{r \in\,[0,t]}\Vert  \partial_t u^R_\mu(r)\Vert _{H^2}^2\r).\end{align*}
Hence, if we define
\[\hat{\rho}_{T,R}(\mu):=\rho_{1}(\mu)+c_T\,\rho_{2,T,R}(\mu),\]
 we have that $\hat{\rho}_{T,R}(\mu)\to 0$, as $\mu\to 0$, and thanks to \eqref{sg30-R}
\begin{align*}
&\mathbb{E}\sup_{r \in\,[0,t]}\Vert u^R(s)\Vert_{H^2}^2 +\int_0^t \mathbb{E}\,\Vert  u^R_\mu(s)\Vert _{H^3}^2\, ds\leq \frac{\hat{\rho}_{T, R}(\mu)}{\mu^\delta}+c_{T, R}\,\int_0^t \mathbb{E}\,\Vert u^R_\mu(s)\Vert_{H^3}^{\frac{4(\bar{\kappa}-\bar{r})}{3-\bar{r}}}\,ds.\end{align*}
Now, because of the way we have chosen $\bar{r}$ in \eqref{xfine130}, we have
$4(\bar{\kappa}-\bar{r})/(3-\bar{r})<2$, so that
\[c_{T, R}\,\int_0^t \mathbb{E}\,\Vert u^R_\mu(s)\Vert_{H^3}^{\frac{4(\bar{\kappa}-\bar{r})}{3-\bar{r}}}\,ds\leq \frac 12 \int_0^t \mathbb{E}\,\Vert u^R_\mu(s)\Vert_{H^3}^{2}\,ds+c_{T, R},\]
so that \eqref{xsa17} follows  with 
\[\rho_{T, R}(\mu):=2\,\left(\hat{\rho}_{T, R}(\mu)+\,c_{T,R}\,\mu^\delta\right).\]
 
Finally, \eqref{xfine141-final} follows from \eqref{xfine141}, due to \eqref{sa29} and \eqref{xsa17}.

\section{Tightness}
\label{sec6}

In this section, by using  the uniform bounds we have proved in Section \ref{sec4} and Section \ref{ssec5.2}, we will prove the tightness of the solution of the truncated problem \ref{SPDE-R} in  appropriate functional spaces.

\begin{Proposition}
\label{teo-tight}
Assume Hypotheses \ref{as1} to \ref{as4}, and fix any $R\geq 1$ and $T>0$ and any initial conditions $(u^\mu_0, v^\mu_0) \in\,\mathcal{H}_2$ satisfying conditions \eqref{finex} and \eqref{xsa11}. Then,  for any sequence $(\mu_k)_{k \in\,\mathbb{N}}$ converging to zero and for every $\varrho<1$, $\vartheta \in\,[1,2)$ and $p<2/(\vartheta-1)$, the family of probability measures  $\left(\mathcal{L}(u^R_{\mu_k})\right)_{k \in\,\mathbb{N}}$  is tight in 
\[C([0,T];H^\varrho)\cap L^p(0,T;H^\vartheta).\] 
\end{Proposition}

In order to prove Proposition \ref{teo-tight}, we will need the following preliminary result.

\begin{Lemma}
For every $R\geq 1$, $T>0$  and $\mu>0$, we define
\[\Phi^R_\mu(t):=u^R_\mu(t)+\mu\,\gamma_R^{-1}(u^R_\mu(t))\,\partial_t u^R_\mu(t),\ \ \ \ \ t \in\,[0,T].\]
Then, under the same assumptions of Proposition \ref{teo-tight}, there exist $\theta \in(0,1)$ and $c_{T, R}>0$ such that for every $0<h<T$  
\begin{equation}
\label{sg57}
\sup_{\mu \in\,(0,1)} h^{-\theta}\,\mathbb{E}\sup_{t \in\,[0,T-h]}\Vert \Phi^R_\mu(t+h)-\Phi^R_\mu(t)\Vert_{H}\,dt\leq c_{T, R}.
	\end{equation}
	\end{Lemma}
\begin{proof}
For every $\psi \in\,H^2$, we define
\[F^R_{\mu}(u,v):=\mu\,\langle \,\gamma_R^{-1}(u)v,\psi\rangle_H,\ \ \ \ \ (u,v) \in\,\mathcal{H}_1.\]
We have
\[D_u F^R_{\mu}(u,v)h=\mu\,\langle \,[D\gamma_R^{-1}(u)h]\,v,\psi\rangle_H,\ \ \ \ \ \ \ D_v F^R_{\mu}(u,v)h=\mu\,\langle \,\gamma_R^{-1}(u)h,\psi\rangle_H,\]
and the  It\^o formula gives
\[\begin{aligned}
&\ds{d F^R_{\mu}(u^R_\mu(t),\partial_t u^R_\mu(t))}	\\[10pt]
 &\ds{=\langle [D\gamma_R^{-1}(u^R_\mu(t))\mu\,\partial_t u^R_\mu(t)]\,\partial_t u^R_\mu(t),\psi\rangle_H\,dt+\langle \gamma_R^{-1}(u^R_\mu(t))\Delta u^R_\mu(t),\psi\rangle_H\,dt-\langle \partial_t u^R_\mu(t),\psi\rangle_H\,dt}\\[10pt]
 &\ds{\hsp+\langle \gamma_R^{-1}(u^R_\mu(t))f_R(u^R_\mu(t)),\psi\rangle_H\,dt+\langle \gamma_R^{-1}(u^R_\mu(t))\sigma_R(u^R_\mu(t))dw^Q(t),\psi\rangle_H.}
\end{aligned}\]
Thus, integrating with respect to time, we obtain
\[\begin{array}{l}
\ds{\langle \Phi^R_\mu(t),\psi\rangle_H=\langle u^R_\mu(t),\psi\rangle_H+F^R_{\mu}(u^R_\mu(t),\partial_t u^R_\mu(t))=\langle u^\mu_0+\mu\,\gamma_R^{-1}(u^\mu_0)v^\mu_0),\psi\rangle_H}\\[10pt]
\ds{\hslp+\mu\int_0^t\langle [D\gamma^{-1}(u^R_\mu(s))\,\partial_t u^R_\mu(s)]\,\partial_t u^R_\mu(s),\psi\rangle_H\,ds+\int_0^t\langle \gamma_R^{-1}(u^R_\mu(s))\Delta u^R_\mu(s),\psi\rangle_H\,ds}\\[14pt]
\ds{\hslp+\int_0^t\langle \gamma_R^{-1}(u^R_\mu(s))f_R(u^R_\mu(s)),\psi\rangle_H\,ds+\int_0^t\langle \gamma_R^{-1}(u^R_\mu(s))\sigma_R(u^R_\mu(s))dw^Q(s),\psi\rangle_H}\\[10pt]
\ds{\hsp=:\langle u^\mu_0+\mu\,\gamma_R^{-1}(u^\mu_0)v^\mu_0),\psi\rangle_H+\sum_{k=1}^4  \langle I^R_{\mu,k}(t),\psi\rangle_H.}
\end{array}\]

Thanks to \eqref{sa51}, we have
\[\begin{array}{l}
\ds{	|\langle I^R_{\mu,1}(t+h)-I^R_{\mu,1}(t),\psi\rangle_H|\leq c\int_{t}^{t+h} \mu\,\Vert \partial_t u^R_\mu(s)\Vert_{H}\Vert \partial_t u^R_\mu(s)\Vert_{H^1}\,ds\,\Vert \psi\Vert_{H}}\\[10pt]
\ds{\leq c\left(\mu\int_0^T \Vert \partial_t u^R_\mu(s)\Vert_{H^1}^2\,ds\right)^{\frac 12}\left(\mu^2\int_0^T \Vert \partial_t u^R_\mu(s)\Vert_{H}^4\,ds\right)^{\frac 14}h^{\frac 14}\Vert \psi\Vert_{H},}
\end{array}
\]
so that, in view of \eqref{sa29} and \eqref{sa1-4th-mom-H}, we get 
\begin{equation}
\label{sg511}	
\sup_{\mu \in\,(0,1)}\mathbb{E}\sup_{t \in\,[0,T-h]}\Vert  I^R_{\mu,1}(t+h)-I^R_{\mu,1}(t)\Vert_{H}\leq c_{T, R}\,h^{\frac 14}\,\Vert \psi\Vert_H.
\end{equation}
Concerning $I^R_{\mu,2}$, we have
\[\vert \langle I^R_{\mu,2}(t+h)-I^R_{\mu,2}(t),\psi\rangle \vert\leq c\int_t^{t+h}\Vert u^R_\mu(s)\Vert_{H^2}\,ds\,\Vert\psi\Vert_H\leq c\left(\int_0^T\Vert u^R_\mu(s)\Vert_{H^2}^2\,ds\right)^{\frac 12} h^{\frac 12}\Vert \psi\Vert_{H},\]
and then, according to \eqref{sg31-bis}, we obtain 
\begin{equation}
\label{sg522}
\sup_{\mu \in\,(0,1)}\mathbb{E}\sup_{t \in\,[0,T-h]}\Vert  I^R_{\mu,2}(t+h)-I^R_{\mu,2}(t)\Vert_{H}\leq c_{T, R}\,h^{\frac 12}\,\Vert\psi\Vert_H.
	\end{equation}
As for $I^R_{\mu,3}$, thanks to  \eqref{n30}, and \eqref{sg31-bis} we have
\begin{equation}
\label{sg533}	
\sup_{\mu \in\,(0,1)}\mathbb{E}\,\sup_{t \in\,[0,T-h]}\Vert I^R_{\mu,3}(t+h)-I^R_{\mu,3}(t)\Vert_{H}\leq c_{T, R}\, h\left(\mathbb{E}\,\sup_{t \in\,[0,T]}\Vert u^R_\mu(t)\Vert_{H^1}^2+1\right)\leq c_{T, R} h.
\end{equation}

Therefore, collecting together \eqref{sg511}, \eqref{sg522}, and \eqref{sg533}, we have
\begin{equation}
\label{sg75}	
\begin{array}{l}
\ds{	\mathbb{E}\sup_{t \in\,[0,T-h]}\Vert \Phi^R_\mu(t+h)-\Phi^R_\mu(t)\Vert_{H}\leq c_{T,R}\,h^{\frac 14}+\mathbb{E}\sup_{t \in\,[0,T-h]}\Vert I^R_{\mu,4}(t+h)-I^R_{\mu,4}(t)\Vert_{H},}
\end{array}
\end{equation}
for every $\mu \in\,(0,1)$ and $h \in\,(0,1)$.

By using a factorization argument as in \cite[Theorem 5.11 and Theorem 5.15]{DPZ}, due to \ the boundedness of $\sigma$ in $H$ we obtain that for some $\theta \in\,(0,1)$
\[
\sup_{\mu \in\,(0,1)} \mathbb{E}	\Vert I^R_{\mu,4}\Vert_{C^\theta([0,T];H)}<\infty.
\]
This and \eqref{sg75} allow  to obtain \eqref{sg57}.

\end{proof}

\begin{proof}[Proof of Proposition \ref{teo-tight}]
For every function $f:[0,T]\to H$ and every $0<h<T$ we define
\[\tau_h f(t) :=f(t+h),\ \ \ \ \ t \in\,[0,T-h].\]
Next,  for every $L>0$ we define
\[K_{L,1}:=\left\{f:[0,T]\times \mathcal{O}\to \mathbb{R}^r\,:\,\sup_{t \in\,[0,T]}\Vert f\Vert_{H^1}\leq L\right\},\]
and
\[K_{L,2}:=\left\{f:[0,T]\times \mathcal{O}\to \mathbb{R}^r\,:\,\sup_{t \in\,[0,T-h]}\Vert \tau_h f(t)-f(t)\Vert_{H}\,dt\leq L\,h^{\theta}\right\},\]
where $\theta \in\,(0,1)$ is the same as the constant  in \eqref{sg57}. Since $H^1$ is compactly embedded in $H^\varrho$, for every $\varrho<1$, due to \cite[Theorem 6]{Simon1986} we have that $K_{L,1}\cap K_{L,2}$ is compact in $C([0,T];H^{\varrho})$, for every  $\varrho<1$.

Now, according to  \eqref{sg31-bis} and \eqref{sa29}, we have 
\[\sup_{\mu \in\,(0,1)}\mathbb{E}\sup_{t \in\,[0,T]}\Vert u^R_\mu(t)+\mu\,\gamma_R^{-1}(u^R_\mu(t))\,\partial_t u^R_\mu(t)\Vert_{H^1}\leq c_{T, R},\]
and hence, for every $\e>0$ we can find $L_{1, R}^{\e}>0$ such that
\begin{equation}
\label{sg60}
\inf_{\mu \in\,(0,1)}\mathbb{P}\left(u^R_\mu+\mu\,\gamma_R^{-1}(u^R_\mu)\,\partial_t u^R_\mu \in\,K_{1,L_{1, R}^{\e}}\right)\geq 1-\e.	
\end{equation}
Moreover, thanks to \eqref{sg57}, we have that there exists $L_{2, R}^\e>0$ such that
\begin{equation}
\label{sg61}
\inf_{\mu \in\,(0,1)}\mathbb{P}\left(u^R_\mu+\mu\,\gamma_R^{-1}(u^R_\mu)\,\partial_t u^R_\mu \in\,K_{2,L_{2, R}^{\e}}\right)\geq 1-\e.\end{equation}
Next, since we are assuming $\bar{r}\geq \rho$,
thanks to \eqref{sg32}, we have 
\[\lim_{\mu\to 0}\, \mu\,\mathbb{E}\sup_{t \in\,[0,T]}\Vert \gamma_R^{-1}(u^R_\mu(t))\,\partial_t u_{\mu}(t)\Vert_{H^\rho}=0.\]
This means in particular that  for every sequence $(\mu_k)_{k \in\,\mathbb{N}}$ converging to zero, there exists a compact set $K^R_{3,\e}$ in $C([0,T];H^\rho)$ such that
\begin{equation}
\label{sg63}
\inf_{k \in\,\mathbb{N}}\mathbb{P}\left(-\mu_k \gamma_R^{-1}(u^R_{\mu_k})\,\partial_t u^R_{\mu_k}\in\,K^R_{3,\e}\right)>1-\e.
\end{equation}
From \eqref{sg60}, \eqref{sg61} and \eqref{sg63}, we have that
\begin{equation}
\label{sg65}
\inf_{k \in\,\mathbb{N}}\mathbb{P}\left(u^R_{\mu_k} \in\,
[K_{1,L_{1, R}^{\e}}\cap K_{2,L_{2, R}^{\e}}]+K^R_{3,\e}\right)\geq 1-3\e.\end{equation}
In particular, since
\[[K_{1,L_{1, R}^{\e}}\cap K_{2,L_{2, R}^{\e}}]+K^R_{3,\e}\subset C([0,T];H^\varrho)\]
is compact, due to the arbitrariness of $\e>0$ we have that the family $(\mathcal{L}(u^R_{\mu_k}))_{k \in\,\mathbb{N}}$ is tight in $C([0,T];H^\varrho) $.

Now, if we define 
\[K_{4,L}:=\left\{ f:[0,T]\times \mathcal{O}\to\mathbb{R}\,:\ \int_0^T\Vert f(t)\Vert_{H^2}^2\,dt\leq L\right\},\]
thanks to \eqref{sg31-bis} we have that there exists $L_{4, R}^\e>0$ such that
\[\inf_{\mu \in\,(0,1)}\mathbb{P}\left(u^R_\mu \in\,K_{4,L^\e_{4,R}}\right)\geq 1-\e.\]
This, together with \eqref{sg65} implies that
\begin{equation}
\label{sg71}\inf_{k \in\,\mathbb{N}}\mathbb{P}\left(u_{\mu_k} \in\,
K^R_\e\right)\geq 1-4\e,\end{equation}
where
\[K^R_\e:=\left([K_{1,L_{1, R}^{\e}}\cap K_{2,L_{2, R}^{\e}}]+K^R_{3,\e}\right)\cap K_{4,L^\e_{4,R}}.\]

Since the set $[K_{1,L_{1, R}^{\e}}\cap K_{2,L_{2, R}^{\e}}]+K^R_{3,\e}$ is compact in $C([0,T];H^\varrho)$,  due to \cite[Theorem 1]{Simon1986} we have that
\[
\lim_{h\to 0}\,\sup_{f \in\,[K_{1,L_{1, R}^{\e}}\cap K_{2,L_{2, R}^{\e}}]+K^R_{3,\e}}\,
\sup_{t \in\,[0,T-h]}\Vert \tau_h f(t)-f(t)\Vert_{H^\varrho}=0.	
\]

Now, for every $\vartheta \in\,[1,2)$, we have
\[ \a:=\frac{\vartheta-\varrho}{2-\varrho}\Longrightarrow \Vert h\Vert_{H^\vartheta}\leq \Vert h\Vert_{H^2}^{\alpha}\Vert h\Vert_{H^\varrho}^{1-\alpha}.\]
Recalling how $K_{4,L^\epsilon_{4, R}}$ was defined, we have that $K^R_\epsilon$ is bounded in $L^2(0,T;H^2)$. Therefore, due to \cite[Theorem 7]{Simon1986}, we have that $K^R_\epsilon$ is compact in $L^{p_{\varrho,\vartheta}}(0,T;H^{\delta})$, where
\[ p_{\varrho,\vartheta}:=\left(\frac \a{2}+\frac{1-\a}\infty\right)^{-1}=\frac{2(2-\varrho)}{\vartheta-\varrho},\]
and, in view of  \eqref{sg71}, this allows to conclude that  $\left(\mathcal{L}(u^R_{\mu_k})\right)_{k \in\,\mathbb{N}}$ is tight in $L^{p_{\varrho,\vartheta}}(0,T;H^\delta)$. Now, since $\varrho<1$ can be choosen arbitrarily close to $1$, we obtain the tightness in $L^p(0,T;H^\vartheta)$, for every $p<2/(\vartheta-1)$.

\end{proof}

\section{About the limiting equation (\ref{lim-eq}) and its truncated version}\label{sec7}

We recall that the function $S$ we have introduced in \eqref{sa160} is defined by
\[S(u)=\int_{H^1}[D\mathfrak{g}^{-1}(u)z]z\,d\nu^u(z),\ \ \ \ \ u \in\,H^1.\] 

\subsection{Properties of the mapping $S$}
\begin{Lemma}
\label{lemma7.1}
Under Hypotheses \ref{as1}, \ref{as2} and  \ref{as4}, we have that $S$ maps $H^1$ into itself, with
\begin{equation}
\label{xsa51}
\Vert S(u)\Vert_{H^1}\leq c\left(1+\Vert u\Vert_{H^1}^2\right),\ \ \ 
\Vert S(u)\Vert_{H}\leq c\left(1+\Vert u\Vert_{H^1}\right),\ \ \ \ \ u \in\,H^1.	
\end{equation}

Moreover, if we also assume Hypothesis \ref{as5}, we have that $S:H^2\to H^1$ is differentiable and 
\begin{equation}
\label{sa162}
\Vert D S(u)h\Vert_{H^1}\leq c\,\left(1+\Vert u\Vert^2_{H^2}\right)\,\Vert h\Vert_{H^1},\ \ \ \ \ u, h \in\,H^2.	
\end{equation}
	\end{Lemma}

	\begin{proof}
Due to \eqref{xfine150}, we have 
\begin{equation}
\label{sa110}
\Vert S(u)\Vert_{H^1}\leq c\int_{H^1}\Vert z\Vert_{H^1}^{2}\,d\nu^u(z)\leq c\,\text{Tr}_{H^1}\Lambda_u\leq c\,\left(1+\Vert u\Vert_{H^1}^{2}\right),	\end{equation}
and the first inequality in \eqref{xsa51} follows. Moreover, 
\begin{align*}
\Vert S(u)&\Vert_{H}\leq c\int_{H^1}\Vert z\Vert_{H^1}\Vert z\Vert_{H}\,d\nu^u(z)\leq c\,\left(\int_{H^1}\Vert z\Vert_{H^1}^{2}\,d\nu^u(z)\right)^{\frac 12}\left(\int_{H^1}\Vert z\Vert_{H}^{2}\,d\nu^u(z)\right)^{\frac 12}\\[10pt]
&\hsp\leq c\,\left(\text{Tr}_{H^1}\Lambda_u\,\text{Tr}_{H}\Lambda_u\right)^{\frac 12}\leq c\,\left(1+\Vert u\Vert_{H^1}\right),	\end{align*}
so that the second inequality in \eqref{xsa51} follows.

In order to study the differentiability of $S$, for every $u \in\,H^1$ we introduce the random variable 
\begin{equation}
	\label{sa141}
	\xi(u):=\int_{-\infty}^0 e^{\,\mathfrak{g}(u)s}\sigma(u)\,d\hat{w}^Q(s),
\end{equation}
where $\hat{w}^Q(t)$, $t \in\,\mathbb{R}$, is the two-sided version of the cylindrical Wiener process $w^Q$. As known, $\xi(u)$   is distributed like $\nu^\nu$, so that we can give an alternative representation of $S$ as
\begin{equation}  \label{sa167}
S(u)=\mathbb{E}\left(\left[D\mathfrak{g}^{-1}(u)\xi(u)\right]\xi(u)\right).	
\end{equation}
Due to \eqref{sa110} we have
that $\xi$ maps $H^1$ into $L^2(\Omega;H^1)$ and \begin{equation}
\label{sa168}
\mathbb{E}\Vert \xi(u)\Vert_{H^1}^2\leq c\left(1+\Vert u\Vert_{H^1}^2\right).	
\end{equation}
Moreover, we can show  that $\xi:H^2\to L^2(\Omega;H^1)$ is differentiable. In order to prove that, for every $u, h \in\,H^1$ we consider the equation
\[\frac{dx}{dt}(t)=-\mathfrak{g}(u)x(t),\ \ \ \ x(0)=h.\]
Its solution is given by
$x_{h}(u,t):=e^{\,-\mathfrak{g}(u)t}h$, for every $t\geq 0$, and 
\begin{equation}   \label{sa164}
\Vert x_{h}(u,t)\Vert_{H^1}\leq e^{-\gamma_0 t}\Vert h\Vert_{H^1},\ \ \ \ \ \ t\geq 0.	
\end{equation}
For every $h \in\,H^1$ and $t\geq 0$ the mapping $u \in\,H^1\mapsto x_{h}(u,t) \in\,H^1$ is differentiable and if we denote $y_{h,k}(u,t):=D_u x_h(u,t) k$, with $k \in\,H^1$, we have that the function $t\mapsto y_{h,k}(u,t)$ satisfies the equation
\[\frac{dy}{dt}(t)=-\mathfrak{g}(u) y(t)-[D\mathfrak{g}(u)k]\cdot x_h(u,t),\ \ \ \ \ \ y(0)=0.\]
This means that
\[y_{h,k}(u,t)=-\int_0^t e^{\,-\mathfrak{g}(u)(t-s)}[D\mathfrak{g}(u)k] x_h(u,s)\,ds,\]
and, thanks to \eqref{sa164}, 
\begin{align*}
\Vert y_{h,k}(u,t)\Vert_{H^1}&\leq \int_0^t e^{-\gamma_0(t-s)}\Vert x_h(u,s)\Vert_{H^1}\,ds \Vert D\mathfrak{g}(u)\Vert_{\mathcal{L}(H^1;\mathbb{R}^{r\times r})}	\Vert k\Vert_{H^1}\\[10pt]
&\hsl\leq c\int_0^t e^{-\gamma_0 t}\,ds \Vert h\Vert_{H^1}\Vert k\Vert_{H^1}= c\,t e^{-\gamma_0 t}\Vert h\Vert_{H^1}\Vert k\Vert_{H^1}.
\end{align*}
In particular,
\begin{equation}
\label{sa165}
\Vert D_u (e^{\,\mathfrak{g}(u)t} h)\Vert_{H^1}\leq -c\,t\,e^{\gamma_0 t}\Vert h\Vert_{H^1},\ \ \ \ \ t\leq 0.	
\end{equation}

This, together with  Hypothesis \ref{as5} implies that
$\xi:H^2\mapsto \xi(u) \in\,L^2(\Omega;H^1)$ is differentiable and 
\[D \xi(u) h=\int_{-\infty}^0[D_u(e^{\,\mathfrak{g}(u)s}h)]\,\sigma(u)\,dw^Q(s)+\int_{-\infty}^0e^{\,\mathfrak{g}(u)s}[D\sigma(u) h]\,dw^Q(s),\ \ \ \ \ u, h \in\,H^2.\]
Moreover, due to \eqref{gsb43},  \eqref{gsb43-bis} and \eqref{sa165}, we have 
\begin{align}   \begin{split}
\label{sa166}
\mathbb{E}&\Vert D\xi(u) h\Vert_{H^1}^2\leq c\,\int_{-\infty}^0 t^2e^{2\gamma_0 t}\,dt\,\Vert \sigma(u)\Vert^2_{\mathcal{L}_2(H_Q,H^1)}\,\Vert h\Vert_{H^1}^2\\[10pt]
&\hsp \hslp+c\,\int_{-\infty}^0 e^{2\gamma_0 t}\,\,dt\,\Vert D\sigma(u)h\Vert^2_{\mathcal{L}_2(H_Q,H^1)}\leq c\,\left(1+\Vert u\Vert_{H^2}^2\right)\Vert h\Vert_{H^1}^2.
\end{split}   \end{align}
Thus, in view of \eqref{sa167} and Hypothesis \ref{as5}, we have that $S$ is differentiable, with 
\begin{align*}
DS(u)h=&\mathbb{E}\left([D^2 \mathfrak{g}^{-1}(u)(h,\xi(u))]\xi(u)\right)\\[10pt]
&\hslp+\mathbb{E}\left([D\mathfrak{g}^{-1}(u)(D_u\xi(u)h)]\xi(u)\right)+\mathbb{E}\left([D \mathfrak{g}^{-1}(u)\xi(u)]D_u\xi(u)h\right),	
\end{align*}
and
\[\Vert DS(u)h\Vert_{H^1}\leq c\,\mathbb{E}\Vert \xi(u)\Vert_{H^1}^2\,\Vert h\Vert_{H^1}
+c\,\mathbb{E}\Vert D\xi(u)h\Vert_{H^1}\Vert \xi(u)\Vert_{H^1}.\]
Thanks to  \eqref{sa168} and \eqref{sa166}, this implies \eqref{sa162}.
		
	\end{proof}

	\subsection{ Pathwise uniqueness  of equation (\ref{lim-eq})} 
	As a consequence of  \eqref{sa162}, we have that the mapping $S:H^2\to H^1$ is locally-Lipschitz continuous and 
	\begin{equation}
	\label{sa170}
	\Vert S(u_1)-S(u_2)\Vert_{H^1}\leq c\left(1+\Vert u_1\Vert_{H^2}^2+	\Vert u_2\Vert_{H^2}^2\right)\Vert u_1-u_2\Vert_{H^1}.
	\end{equation}
This allows to prove the following uniqueness result for equation \eqref{lim-eq}.
	
	\begin{Lemma}
	\label{teo7.2}
Under Hypotheses \ref{as1} to \ref{as5}, equation \eqref{lim-eq}	admits at most one solution $u$ in the  space $L^2(\Omega;C([0,T];H^1)\cap L^2(0,T;H^2))$.
\end{Lemma}

\begin{proof}
Assume that $u_1, u_2 \in\,	L^2(\Omega;C([0,T];H^1)\cap L^2(0,T;H^2))$ are two solutions of equation \eqref{lim-eq}. If we define $\rho(t):=u_1(t)-u_2(t)$, we have
\begin{align*}
	\begin{split}
	\partial_t \rho(t,x)&=\left(\mathfrak{g}^{-1}(u_1(t))\Delta \rho(t,x)+\left[\mathfrak{g}^{-1}(u_1(t))-\mathfrak{g}^{-1}(u_2(t))\right]\Delta u_2(t,x)\right.\\[10pt]
	&\hsl\left.+[f_\mathfrak{g}(u_1(t))-f_\mathfrak{g}(u_2(t))]+[S(u_1(t))-S(u_2(t))]\right)\,dt+[\sigma_{\mathfrak{g}}(u_1(t))-\sigma_{\mathfrak{g}}(u_2(t))]\partial_tw^Q(t),	
	\end{split}
\end{align*}
where
\[f_{\mathfrak{g}}(u):=\mathfrak{g}^{-1}(u)f(u),\ \ \ \ \ \ \sigma_{\mathfrak{g}}(u):=\mathfrak{g}^{-1}(u)\sigma(u),\ \ \ \ \ u \in\,H^1.\]
It is immediate to check that
\begin{equation}  \label{sa171}\Vert f_{\mathfrak{g}}(u_1)-f_{\mathfrak{g}}(u_2)\Vert_{H^1}\leq c\Vert u_1-u_2\Vert_{H^1}\left(1+\Vert u_2\Vert_{H^1}\right),\end{equation}
and
\begin{equation}  \label{sa172}\Vert \sigma_{\mathfrak{g}}(u_1)-\sigma_{\mathfrak{g}}(u_2)\Vert_{\mathcal{L}_2(H_Q,H^1)}\leq c\Vert u_1-u_2\Vert_{H^1}\left(1+\Vert u_2\Vert_{H^1}\right).\end{equation}
Moreover, we have
\begin{align}\begin{split} \label{xfine230}
&\left\langle\left[\mathfrak{g}^{-1}(u_1(t))-\mathfrak{g}^{-1}(u_2(t))\right]\Delta u_2(t,x),\rho(t)\right\rangle_{H^1}\leq \Vert\mathfrak{g}^{-1}(u_1(t))-\mathfrak{g}^{-1}(u_2(t))\Vert_{\mathbb{R}^{r\times r}}\Vert u_2(t)\Vert_{H^2}\Vert\rho(t)\Vert_{H^2} \\[10pt]
&\hslp\leq c\,\Vert \rho(t)\Vert_{H^1}\Vert u_2(t)\Vert_{H^2}\Vert \rho(t)\Vert_{H^2}\leq \frac{\tilde{\gamma_0}}2\,\Vert \rho(t)\Vert_{H^2}^2+c\,\rho(t)\Vert^2_{H^1}\Vert u_2(t)\Vert^2_{H^2}.
\end{split}\end{align}
Hence, according to \eqref{sa170}, \eqref{sa171},  \eqref{sa172} and \eqref{xfine230}, from  the It\^o formula we obtain
\begin{align*}
	\frac 12d\Vert \rho(y)\Vert_{H^1}^2&\leq -\frac{\tilde{\gamma_0}}{2}\,\Vert \rho(t)\Vert_{H^2}^2\,dt+c\,\Vert \rho(t)\Vert_{H^1}^2\left(\Vert u_1(t)\Vert_{H^2}^2+\Vert u_2(t)\Vert_{H^2}^2+1\right)\,dt\\[10pt]
	&\hsp+\langle [\sigma_{\mathfrak{g}}(u_1(t))-\sigma_{\mathfrak{g}}(u_2(t))]dw^Q(t),\rho(t)\rangle_{H^1}.	\end{align*}
Next, for an arbitrary $\kappa>0$ we 
define
\[\Gamma_\kappa(t):=\exp\left(-\kappa\int_0^t \left(\Vert u_1(s)\Vert_{H^2}^2+\Vert u_2(s)\Vert_{H^2}^2+1\right)\,ds\right),\]
and apply It\^{o}'s formula to 
the process 
$\Gamma_\kappa(t)\Vert \rho(t)\Vert_{H^1}^2$.
We obtain
\begin{align*}
	d\left(\Gamma_\kappa(t)\Vert \rho(t)\Vert_{H^1}^2\right)&\leq \Gamma_\kappa(t)\left(-\frac{\tilde{\gamma_0}}{2}\,\Vert \rho(t)\Vert_{H^2}^2+\bar{c}\,\Vert \rho(t)\Vert_{H^1}^2\left(\Vert u_1(t)\Vert_{H^2}^2+\Vert u_2(t)\Vert_{H^2}^2+1\right)\right)\,dt\\[10pt]
	&\hslp+\Gamma_\kappa(t)\langle [\sigma_{\mathfrak{g}}(u_1(t))-\sigma_{\mathfrak{g}}(u_2(t))]dw^Q(t),\rho(t)\rangle_{H^1} \\[10pt]
	&\hsp-\Gamma_\kappa(t)\kappa \left(\Vert u_1(t)\Vert_{H^2}^2+\Vert u_2(t)\Vert_{H^2}^2+1\right)\Vert \rho(t)\Vert_{H^1}^2\,dt.
\end{align*}
In particular, if we take $\kappa:=\bar{c}$, we get
\[d\left(\Gamma_{\bar{c}}(t)\Vert \rho(t)\Vert_{H^1}^2\right)\leq \Gamma_\kappa(t)\langle [\sigma_{\mathfrak{g}}(u_1(t))-\sigma_{\mathfrak{g}}(u_2(t))]dw^Q(t),\rho(t)\rangle_{H^1}.\]
Therefore, if we integrate in time and take the expectation of both sides, since $\rho(0)=0$ we get
\[\mathbb{E}\,\left(\Gamma_{\bar{c}}(t)\,\Vert \rho(t)\Vert_{H^1}^2\right)\leq 0,\ \ \ \ \ \ t\geq 0.\]
Since $u_1, u_2 \in\,L^2(\Omega;L^2(0,T;H^2))$, we have    $\mathbb{P}(\Gamma_{\bar{c}}(t)>0,\ t\geq 0)=1$, and this implies that $u_1=u_2$.

\end{proof}

\medskip

Now, for every $R\geq 1$ and $u, v \in\,H^1$,  we denote by $y^{u,v}_R$ the solution of problem \eqref{sa40}, where the mappings $\mathfrak{g}$ and $\sigma$ are replaced by $\mathfrak{g}_R$ and $\sigma_R$. Clearly, we have \ 
\begin{equation}
\label{sa41-H1-R}
\sup_{t \geq 0}\mathbb{E}\Vert y_R^{u,v}(t)\Vert_{H^1}^p\leq c_{p}\left(1+e^{-\gamma_0 t}\Vert v\Vert_{H^1}^p+\Vert u\Vert_{H^1}^p\right),\ \ \ \ \ t\geq 0,\ \ \ \ \ R\geq 1.
\end{equation} 
Next, we define
\[P^{R,u}_t\varphi(v)=\mathbb{E}\,\varphi(y_R^{u,v}(t)),\ \ \ \ \ v \in\,H,\ \ \ \ t\geq 0,\]
for every function $\varphi \in\,B_b(H^1)$, and
\[\Lambda_{R,u}:=\int_0^\infty e^{-\mathfrak{g}_R(u)s}[\sigma_R(u)Q][\sigma_R(u)Q]^\star e^{-\mathfrak{g}_R^t(u)s}\,ds.\]
As we have seen for $\Lambda_u$, we have that  $\Lambda_{R, u} \in\,\mathcal{L}^+_1(H)\cap\mathcal{L}^+_1(H^1)$, 
$\nu_R^u:=\mathcal{N}(0,\Lambda_{R,u})$ is the unique invariant measure for the semigroup $P^{R, u}_t$ and for every $p\geq 1$ \begin{equation}
\label{sa110-R}
\int_{H^1}\Vert z\Vert_{H^1}^{2p}\,d\nu_R^u(z)\leq c_p\left(\text{Tr}_{H^1}\Lambda_u\right)^p\leq c_p\,\left(1+\Vert u\Vert_{H^1}^{2p}\right),\end{equation}
for some constant $c$ independent of $R\geq 1$.

Finally, if we define
\begin{equation}
\label{sa160-R}
S_R(u):=\int_{H^1}\left[D\mathfrak{g}_R^{-1}(u)	z\right]z\,d\nu_R^u(z),\ \ \ \ \ u \in\,H^1,
\end{equation}
it is immediate to check that Lemma \ref{lemma7.1} is true also for $S_R$, so that
\begin{equation}
	\label{sa170-R}
	\Vert S_R(u_1)-S_R(u_2)\Vert_{H^1}\leq c\left(1+\Vert u_1\Vert_{H^1}^2+	\Vert u_2\Vert_{H^1}^2\right)\Vert u_1-u_2\Vert_{H^1},\ \ \ \ \ u_1, u_2 \in\,H^1,
	\end{equation}for some constant $c>0$ independent of $R\geq 1$.
For every $R\geq 1$, we introduce the problem
\begin{equation}
\label{lim-eq-R}
\begin{cases}
\ds{\partial_t u^R(t,x)=\mathfrak{g}_R^{-1}(u^R(t))\Delta u^R(t,x)+\mathfrak{g}_R^{-1}(u^R(t))f_R(u^R(t,x))}\\[14pt]
\ds{\hsp\hslp+S_R(u^R(t))+\mathfrak{g}_R^{-1}(u^R(t))\sigma_R(u^R(t))\partial_tw^Q(t),}\\[10pt]
\ds{u^R(0,x)=u_0(x),\ \ x \in\,\mathcal{O},\ \ \ \ \ \ \ \ u^R(t,x)=0,\ \ x \in \partial\mathcal{O}.}	
\end{cases}
	\end{equation}
	Proposition \ref{teo7.2} applies also to problem \eqref{lim-eq-R}, so that for every $R\geq 1$ there exists at most one solution $u^R \in\,L^2(\Omega;C([0,T];H^1)\cap L^2(0,T;H^2))$.

\section{Proof of Theorem \ref{teo3.4} for the truncated equation (\ref{SPDE-R})}

In what follows we will prove the following analog of Theorem \ref{teo3.4} for  the truncated equation \eqref{SPDE-R}.
\begin{Theorem}
\label{teo3.4-R}
Assume Hypotheses \ref{as1} to \ref{as5} and fix  an arbitrary $\varrho<1$ and $\vartheta \in\,[1,2)$. Moreover, assume that $(u_0^\mu,v_0^\mu) $ belong to $\mathcal{H}_3$, for every $\mu \in\,(0,1)$, and satisfy conditions \eqref{finex}, \eqref{xsa11} and \eqref{fx2}. Then, 
 for every   $R\geq 1$, $p<2/(\vartheta-1)$ and $\eta, T>0$
\begin{equation}
\label{sa161-R}
\lim_{\mu\to 0}\,\mathbb{P}\left(\sup_{t \in \,[0,T]}\Vert u^R_\mu(t)-u^R(t)\Vert_{H^\varrho}+\int_0^T \Vert u^R_\mu(t)-u^R(t)\Vert_{H^\vartheta}^p\,dt>\eta\right)=0,	
\end{equation}
where $u^R \in\,L^2(\Omega;C([0,T];H^1)\cap L^2(0,T;H^2))$ is the unique solution of  problem \eqref{lim-eq-R}.
\end{Theorem}

In Proposition \ref{teo-tight} we have seen that  $\left(\mathcal{L}(u^R_{\mu_k})\right)_{k \in\,\mathbb{N}}$  is  tight  in 
$C([0,T];H^\varrho)\cap L^p(0,T;H^\vartheta)$, for every $R\geq 1$.
Hence,
since equation \eqref{lim-eq-R} has pathwise uniqueness, Theorem \ref{teo3.4-R} follows once we show that any weak limit point of $\{u^R_\mu\}_{\mu \in\,(0,1)}$ is a solution equation of \eqref{lim-eq-R}. Notice that throughout this section, we will assume that all Hypotheses, from \ref{as1} to \ref{as5} are satisfied.

\subsection{The construction of the corrector functions}

Due to Hypotheses \ref{as4} and \ref{as5},  the mapping 
\[u \in\,H^1\mapsto y_R^{u,v}(t) \in\,L^2(\Omega;H^1),\]
 is differentiable, for any fixed $v \in\,H^1$ and $t\geq 0$, and if we denote $\eta^{u,v}_{R, k}(t):=D_u\, y_R^{u,v}(t)\cdot k$, we have that the process $\eta^{u,v}_{R, k}$ is the solution of the problem
\[d\eta(t)=-\left(\mathfrak{g}_R(u)\eta(t)+[D\mathfrak{g}_R(u)k]\,y_R^{u,v}(t)\right)\,dt+[D\sigma_R(u)k]\,dw^Q(t),\ \ \ \ \eta(0)=0.\]
In particular, we have 
\[\eta^{u,v}_{R, k}(t)=\int_0^t e^{-\mathfrak{g}_R(u)(t-s)}[D\mathfrak{g}_R(u)k]\,y_R^{u,v}(s)\,ds+\int_0^t e^{-\mathfrak{g}_R(u)(t-s)}[D\sigma_R(u)k]\,dw^Q(s).\]
This implies that
\begin{align*}
	\mathbb{E}\,&\Vert \eta^{u,v}_{R, k}(r)\Vert_{H^1}^2\leq c\left(\sup_{s \in\,[0,t]}\mathbb{E}\Vert y_R^{u,v}(s)\Vert_{H^1}^2\int_0^t e^{-2\gamma_0 s}\,ds+c\int_0^te^{-2\gamma_0 (t-s)}\,ds\,\left(1+\Vert u\Vert_{H^2}^2\right)\right)\Vert k\Vert_{H^1}^2,
\end{align*}
so that, thanks to \eqref{sa41-H1-R}, we obtain that for every $T>0$ fixed
\begin{equation}
\label{sa130}
\sup_{t\geq 0}\,	\mathbb{E}\,\Vert \eta^{u,v}_{R, k}(t)\Vert_{H^1}^2\leq c\left(1+\Vert u\Vert_{H^2}^2+\Vert v\Vert_{H^1}^2\right)\Vert k\Vert_{H^1}^2,
\end{equation}
for some constant $c$ independent of $R\geq 1$.

It is immediate to check that $P^{R, u}_t$ is a Feller contraction semigroup in $C_b(H^1)$. Moreover, it is  weakly continuous in $C_b(H^1)$ (for the definition see \cite[Appendix B]{cerrai94}). In particular, this means that there exists a closed operator $M^{R, u}:D(M^{R, u})\subseteq C_b(H^1)\to C_b(H^1)$ such that for every $\la>0$ and $\varphi \in\,C^1_b(H^1)$
\[(\la-M^{R, u})^{-1}\varphi(v)=\int_0^\infty e^{-\la t} P^{R, u}_t\varphi(v)\,dt,\ \ \ \ \ \ v \in\,H^1.\] Notice that $C^2_b(H^1)\subset D(M^{R, u})$ and $M^{R, u}\varphi=\mathcal{M}^{R, u}\varphi$, for every $\varphi \in\,C^2_b(H^1)$,
where
$\mathcal{M}^{R, u}$ is the Kolmogorov operator associated with equation \eqref{sa40}
\begin{equation} \label{sa180}\mathcal{M}^{R, u}\varphi(v)=\frac 12\text{Tr}_{H^1}\left(D^2\varphi(v)[\sigma_R(u)Q][\sigma_R(u)Q]^\star\right)-\langle \mathfrak{g}_R(u) v,D\varphi(v)\rangle_{H^1},\ \ \ \ \varphi \in\,C^2_b(H^1).\end{equation}

\begin{Lemma} \label{lemma6.1}
	Let $\varphi:H^1\to\mathbb{R}$ be any continuous function such that
	\[\vert\varphi(v_1)-\varphi(v_2)\vert\leq c_\varphi\,\Vert v_1-v_2\Vert_{H^1}\left(1+\Vert v_1\Vert^{\kappa_\varphi}_{H^1}+ \Vert v_2\Vert^{\kappa_\varphi}_{H^1}\right),\ \ \ \ \ v_1, v_2 \in\,H^1,\]
	for some $c_\varphi>0$ and $\kappa_\varphi\geq 0$.
	Then, there exists some other constant $\bar{c}_\varphi>0$ such that for every $u, v \in\,H^1$
	\begin{equation} \label{sa111}
	\left\vert P^{R, u}_t\varphi(v)- \int_{H^1} \varphi(z)\,d\nu_R^u(z)\right\vert\leq \bar{c}_\varphi	\left(1+\Vert v\Vert_{H^1}^{{\kappa_\varphi+1}}+\Vert u\Vert_{H^1}^{{\kappa_\varphi+1}}\right)\, e^{-\gamma_0 t}.
	\end{equation}
\end{Lemma}
\begin{proof}
Thanks to \eqref{sa41-H1-R}, for every $v_1, v_2 \in\,H^1$, we have
\begin{align*}
\vert P_t^{R, u}\varphi(v_1)-&P_t^{R, u}\varphi(v_2)\vert\leq c_\varphi\,\mathbb{E}\,\left(\Vert y_R^{u,v_1}(t)	-y_R^{u,v_2}(t)\Vert_{H^1}\left(1+\Vert y_R^{u,v_1}(t)\Vert_{H^1}^{\kappa_\varphi}+\Vert y_R^{u,v_2}(t)\Vert_{H^1}^{\kappa_\varphi}\right)\right)\\[10pt]
&\leq c_\varphi\,\Vert e^{-\mathfrak{g}(u)t}(v_1-v_2)\Vert_{H^1}\mathbb{E}\left(1+\Vert y_R^{u,v_1}(t)\Vert_{H^1}^{\kappa_\varphi}+\Vert y_R^{u,v_2}(t)\Vert_{H^1}^{\kappa_\varphi}\right)\\[10pt]
&\leq c_\varphi\,e^{-\gamma_0 t}\Vert v_1-v_2\Vert_{H^1}c_{{\kappa_\varphi}}\left(1+\Vert u\Vert_{H^1}^{\kappa_\varphi}+\Vert v_1\Vert_{H^1}^{\kappa_\varphi}+\Vert v_2\Vert_{H^1}^{\kappa_\varphi}\right),
\end{align*}	
last inequality following from \eqref{sa41-H1}. Due to the invariance of $\nu^u_R$, 
this implies that 
\begin{align*}
	&\left\vert P^{R, u}_t\varphi(v)- \int_{H^1} \varphi(z)\,d\nu_R^u(z)\right\vert\leq
	 \int_{H^1} \left\vert P^{R, u}_t\varphi(v)- P^{R, u}_t\varphi(z)\right\vert d\nu_R^u(z)\\[10pt]
	 &\hslp\leq  c_\varphi\,c_{{\kappa_\varphi}}\,e^{-\gamma_0 t}\int_{H^1}\Vert v-z\Vert_{H^1}\left(1+\Vert u\Vert_{H^1}^{\kappa_\varphi}+\Vert v\Vert_{H^1}^{\kappa_\varphi}+\Vert z\Vert_{H^1}^{\kappa_\varphi}\right)\,d\nu_R^u(z)\\[10pt]
	 &\hsp\leq  c_\varphi\,c_{{\kappa_\varphi}}\,e^{-\gamma_0 t}\left(1+\Vert u\Vert_{H^1}^{\kappa_\varphi+1}+\Vert v\Vert_{H^1}^{\kappa_\varphi+1}+\int_{H^1}\Vert z\Vert_{H^1}^{\kappa_\varphi+1}\,d\nu_R^u(z)\right),
\end{align*}
and thanks to \eqref{sa110-R} we obtain \eqref{sa111}.
\end{proof}

In what follows, we fix $h \in\,H$ and we introduce the function
\[\varphi^R_1(u,v):=\langle \mathfrak{g}_R^{-1}(u)v,h\rangle_{H},\ \ \ \ \ u \in\,H^1,\ \ v \in\,H.\]
It is immediate to check that that $\varphi^R_1(u,\cdot) \in\,C^2_b(H)$ and 
\begin{equation}
\label{sa101}
\mathcal{M}^{R, u} \varphi^R_1(u,v)=-\langle v,h\rangle_H.	
\end{equation}
Thanks to Hypothesis \ref{as4}, for every fixed $v \in\,H$  we have that $\varphi^R_1(\cdot,v) \in\,C^1_b(H^1)$ and 
\begin{equation}
\label{sa102}
\langle D_u\varphi^R_1(u,v),k\rangle_{H^1}=\langle [D\mathfrak{g}_R^{-1}(u)k]v,h\rangle_H=-\langle \mathfrak{g}_R^{-1}(u)[D\mathfrak{g}_R(u)k]\mathfrak{g}_R^{-1}(u)v,h\rangle_H,\ \ \ \ k \in\,H^1.	
\end{equation}

Now,  if we introduce the mapping
\[\psi_R(u,v):=\langle D_u\varphi^R_1(u,v),v\rangle_{H^1}=\langle[D\mathfrak{g}_R^{-1}(u)v]v,h\rangle_H,\ \ \ \ \ u, v \in\,H^1,\]
we have
\[\int_{H^1}\psi_R(u,z)\,d\nu_R^u(z)=\langle S_R(u),h\rangle_H.\]
\begin{Lemma}
For every $u \in\,H^1$, the function $\psi_R(u,\cdot):H^1\to \mathbb{R}$ has quadratic growth and is locally Lipschitz-continuous, uniformly with respect to $u \in\,H^1$. Namely
\begin{equation}
\label{sa112}
\sup_{u \in\,H^1}\vert \psi_R(u,v)\vert\leq c\,\Vert h\Vert_{H}\,\Vert v\Vert_{H^1}\Vert v\Vert_H,	
\end{equation}
and
\begin{align}\begin{split}
\label{sa113}
\sup_{u \in\,H^1}\vert \psi_R(u,v_1)-\psi_R(u,v_2)\vert\leq	 c\,\Vert v_1-v_2\Vert_{H^1}\left(\Vert v_1\Vert_{H^1}+\Vert v_2\Vert_{H^1}\right)\Vert h\Vert_H.
\end{split} \end{align}
Moreover, $\psi_R(u,\cdot)$ is twice continuously differentiable in $H^1$, with
\begin{equation}\label{sa120}
	\sup_{u \in\,H^1}\Vert D_v\psi_R(u,v)\Vert_{H^1}\leq c\,\Vert v\Vert_{H^1}\Vert h\Vert_H,
	\end{equation}
	and
\begin{equation}  \label{sa121}\sup_{u, v \in\,H^1}\Vert D_v^2\psi_R(u,v)\Vert_{\mathcal{L}(H^1)}\leq c\,\Vert h\Vert_{H}.	 
	 \end{equation}
	 Finally, for every $v \in\,H^1$ the mapping $\psi_R(\cdot,v):H^1\to\mathbb{R}$ is differentiable and 
	 \begin{equation}
	 \label{sa133}
	\sup_{u \in\,H^1} \Vert D_u \psi_R(u,v)\Vert_{H^1}\leq c\,\Vert v\Vert_{H^1}\Vert v\Vert_{H}\,\Vert h\Vert_H.	
	 \end{equation}

\end{Lemma}

\begin{proof}
Bound \eqref{sa112} is obvious. Moreover, \eqref{sa113} follows immediately, as soon as we see that 	
\[\vert \psi_R(u,v_1)-\psi_R(u,v_2)\vert\leq	c\left(\Vert v_1-v_2\Vert_{H^1}\Vert v_1\Vert_H+\Vert v_1-v_2\Vert_{H}\Vert v_2\Vert_{H^1}\right)\Vert h\Vert_H.\]
As for the differentiability of $\psi^u$, we have
\[\langle D_v\psi_R(u,v),k\rangle_{H^1}=-\langle \mathfrak{g}_R^{-1}(u)[D\mathfrak{g}_R(u)k]\mathfrak{g}_R^{-1}(u)v+\mathfrak{g}_R^{-1}(u)[D\mathfrak{g}_R(u)v]\mathfrak{g}_R^{-1}(u)k,h\rangle_H,\]
and this immediately implies \eqref{sa120}. Moreover, if we differentiate once more, we get
\[\langle D_v^2\psi_R(u,v)k_1,k_2\rangle_{H^1}=-\langle \mathfrak{g}_R^{-1}(u)[D\mathfrak{g}_R(u)k_1]\mathfrak{g}_R^{-1}(u)k_2+\mathfrak{g}_R^{-1}(u)[D\mathfrak{g}_R(u)k_2]\mathfrak{g}_R^{-1}(u)k_1,h\rangle_H,\]
and \eqref{sa121} follows.

Finally, since we are assuming that $\mathfrak{g}_R$ is twice differentiable and $\mathfrak{g}_R^{-1}$ is bounded, we have that $\psi_R(\cdot,v)$ is differentiable and for every $k \in\,H^1$ 
\begin{align*}
\langle D_u\psi_R(u,v),&k\rangle_{H^1}=\langle \mathfrak{g}_R^{-1}(u)\left[D\mathfrak{g}_R(u)k\right]\mathfrak{g}_R^{-1}(u)	\left[D\mathfrak{g}_R(u)v\right]\mathfrak{g}_R^{-1}(u)v,h\rangle_{H}\\[10pt]
&\hsl\hsl-\langle \mathfrak{g}_R^{-1}(u)\left[D^2\mathfrak{g}_R(u)(v,k)\right]\mathfrak{g}_R^{-1}(u)	v,h\rangle_{H}+\langle \mathfrak{g}_R^{-1}(u)\left[D\mathfrak{g}_R(u)v\right]\mathfrak{g}_R^{-1}(u)	\left[D\mathfrak{g}_R(u)k\right]\mathfrak{g}_R^{-1}(u)v,h\rangle_{H}.
\end{align*}
In particular, \eqref{sa133} follows.

\end{proof}

Next, for every $u, v \in\,H^1$ and $\mu \in\,(0,1)$, we define
\[\varphi_{2}^{R, \mu}(u,v):=\int_0^\infty e^{-\lambda(\mu) t}\left(P^{R, u}_t\psi_R(u,\cdot)(v)-\langle S_R(u),h\rangle_H\right)\,dt,\]
where $\la:[0,1)\to \mathbb{R}$ is some continuous increasing function such that $\lambda(0)=0$, to be determined.
As a consequence of Lemma \ref{lemma6.1} and  \eqref{sa113}, we have
\begin{align}\begin{split}
\label{sa114}
\left\vert P_t^{R, u}\psi_R(u,\cdot)(v)-\langle S_R(u),h\rangle_H\right|&=\left\vert P_t^{R, u}\psi_R(u,\cdot)(v)-\int_{H^1}\psi_R(u,z)\,d\nu_R^u(z)\right|\\[10pt]
&\hs\leq 	c	\left(1+\Vert v\Vert_{H^1}^{2}+\Vert u\Vert_{H^1}^{2}\right)\Vert h\Vert_H\, e^{-\gamma_0 t},
\end{split} \end{align}
so that  the function $\varphi_{2}^{R, \mu}:H^1\times H^1\to\mathbb{R}$ is well defined and 
\begin{equation}
	\label{sa140}
	\sup_{\mu \in\,(0,1)}\vert\varphi_{2}^{R, \mu}(u,v)\vert\leq c \left(1+\Vert v\Vert_{H^1}^{2}+\Vert u\Vert_{H^1}^{2}\right)\Vert h\Vert_H.
\end{equation}
\begin{Lemma}\label{lemma7.5}
For every fixed $\mu \in\,(0,1)$ and $u \in\,H^1$, the function $\varphi_{2}^{R, \mu}(u,\cdot):H^1\to\mathbb{R}$ is twice continuously differentiable and	
\begin{equation}
	\label{sa122}
	\mathcal{M}^{R, u} \varphi_{2}^{R, \mu}(u,v)=\lambda(\mu)\,\varphi_{2}^{R, \mu}(u,v)-\left(\psi_R(u,v)-\langle S_R(u),h\rangle_H\right),\ \ \ \ u, v \in\,H^1.
\end{equation}
Moreover, for every $\mu \in\,(0,1)$ and $v \in\,H^1$, the function $\varphi_{2}^{R, \mu}(\cdot,v):H^2\to\mathbb{R}$ is continuously differentiable, with
\begin{equation}   \label{sa180}\vert D_u \varphi_{2}^{R, \mu}(u,v)\,k\vert \leq \frac{c}{\lambda(\mu)}\left(1+\Vert u\Vert_{H^2}^2+\Vert v\Vert_{H^1}^2\right)\Vert k\Vert_{H^1}\Vert h\Vert_H,\ \ \ \ \ k \in\,H^2.\end{equation}
\end{Lemma}
\begin{proof}
For every $u \in\,H^1$ and $t\geq 0$ the mapping
\[v \in\,H^1\mapsto y_R^{u,v}(t) \in\,L^2(\Omega;H^1),\] is twice differentiable with
\[D_v\, y_R^{u,v}(t)k=e^{-\mathfrak{g}(u) t}k,\ \ \ \ \ D^2_v\, y_R^{u,v}(t)(k_1,k_2)=0.\]
Hence, since the function $\psi_R(u,\cdot)$ is twice continuously differentiable, we have that $P_t^{R, u}\psi_R(u,\cdot)$ is twice continuously differentiable in $H^1$, with
\[\langle D_v P_t^{R, u}\psi_R(u,\cdot)(v),k\rangle_{H^1}=\mathbb{E}\langle D_v\psi_R(u,y_R^{u,v}(t)),e^{-\mathfrak{g}(u) t}k\rangle_{H^1}\]
and
\[\langle D_v^2 P_t^{R, u}\psi_R(u,\cdot)(v)k_1,k_2\rangle_{H^1}=\mathbb{E}\langle D_v^2\psi_R(u,y_R^{u,v}(t))\,e^{-\mathfrak{g}(u) t}k_1,e^{-\mathfrak{g}(u) t}k_2\rangle_{H^1}.\]
These two equalities, together with \eqref{sa41-H1-R}, \eqref{sa120} and \eqref{sa121}, imply that 
\[\Vert D_v P_t^{R, u}\psi_R(u,\cdot)(v)\Vert_{H^1}\leq c\,\Vert h\Vert_H\,\mathbb{E}\Vert y_R^{u,v}(t)\Vert_{H^1}e^{-\gamma_0 t}\leq c\,\Vert h\Vert_H\,\left(1+\Vert v\Vert_{H^1}+\Vert u\Vert_{H^1}\right)e^{-\gamma_0 t},\]
and
\[\Vert D_v^2 P_t^{R, u}\psi_R(u,\cdot)(v)\Vert_{\mathcal{L}(H^1)}\leq c\,\Vert h\Vert_H\,e^{-2\gamma_0 t}.\]
In particular, the function $\varphi_{2}^{R, \mu}(u,\cdot):H^1\to\mathbb{R}$ is twice continuously differentiable, with
\begin{equation}  \label{sa191}\Vert D_v\varphi_{2}^{R, \mu}(u,v)\Vert_{H^1}\leq c\,\left(1+\Vert v\Vert_{H^1}+\Vert u\Vert_{H^1}\right)\Vert h\Vert_H,\end{equation}
and
\[\Vert D^2_v\varphi_{2}^{R, \mu}(u,v)\Vert_{\mathcal{L}(H^1)}\leq c\,\Vert h\Vert_H.\]
Finally, since we have
\[M^{R, u} \varphi_{2}^{R, \mu}(u,v)=\lambda(\mu)\varphi_{2}^{R, \mu}(u,v)-\left(\psi_R(u,v)-\langle S_R(u),h\rangle_H\right),\ \ \ \ u, v \in\,H^1,\]
\eqref{sa122} follows once we notice that, as $\varphi_{2}^{R, \mu}(u,\cdot)$ is twice continuously differentiable, it holds 
\[\mathcal{M}^{R, u}\varphi_{2}^{R, \mu}(u,v)=M^{R, u}\varphi_{2}^{R, \mu}(u,v).\]

Now,  we have already seen that the mapping $S_R:H^2\to H^1$ is differentiable
and 
\begin{equation}
\label{xsa52}
\Vert DS_R(u)h\Vert_{H^1}\leq c\left(1+\Vert u\Vert_{H^2}\right)\,\Vert k\Vert_{H^1},\ \ \ \ u, k \in\,H^2.
\end{equation}
Hence, in order to study the differentiability of $\varphi_{2}^{R, \mu}$ with respect to $u \in\,H^1$, we need to study the differentiability of the mapping
\[u \in\,H^1\mapsto \Psi^R_{v,t}(u):=P_t^{R, u}\psi_R(u,\cdot)(v)=\mathbb{E}\,\psi_R(u,y_R^{u,v}(t)) \in\,\mathbb{R},\]
for every fixed $t\geq 0$ and $v \in\,H^1$.

We have
\[\langle D_u\Psi^R_{v,t}(u),k\rangle_{H^1}=\mathbb{E}\langle D_u\psi_R(u,y_R^{u,v}(t)),k\rangle_{H^1}+\mathbb{E}\langle D_v\psi_R(u,y_R^{u,v}(t)),\eta^{u,v}_{R, k}(t)\rangle_{H^1}.\]
Thus, thanks to \eqref{sa120} and \eqref{sa133}, we obtain
\[\vert \langle D_u\Psi^R_{v,t}(u),k\rangle\vert_{H^1}\leq c\,\left(\mathbb{E}\Vert y_R^{u,v}(t)\Vert_{H^1}^2\Vert k\Vert_{H^1}+\mathbb{E}\Vert y_R^{u,v}(t)\Vert_{H^1}\Vert \eta^{u,v}_{R, k}(t)\Vert_{H^1}\right)\Vert h\Vert_H,\]
and, as a consequence of \eqref{sa41-H1-R} and \eqref{sa130}, we conclude that 
\[
\Vert D_u\Psi^R_{v,t}(u)\Vert_{H^1}\leq c\left(1+\Vert u\Vert_{H^2}^2+\Vert v\Vert_{H^1}^2\right)\Vert h\Vert_H.	
\]
This, together with \eqref{sa162} and \eqref{xsa52} , implies that $\varphi_{2}^{R, \mu}(\cdot,v):H^2\to \mathbb{R}$ is differentiable, for every $v \in\,H^1$, and
\[
\vert D_u\varphi_{2}^{R, \mu}(u,v)\,k\vert\leq c\int_0^\infty e^{-\lambda(\mu)t}\,dt\,\left(1+\Vert u\Vert_{H^2}^2+\Vert v\Vert_{H^1}^2\right)\,\Vert k\Vert_{H^1}\Vert h\Vert_H,
\]
so that \eqref{sa180} follows.
\end{proof}

\subsection{The identification of the limit}

We fix $h \in\,H$, and for every $R\geq 1$ and  $\mu \in\,(0,1)$  we define
\[\varphi_{R, \mu}(u,v)=\langle u,h\rangle_H+\sqrt{\mu}\,\varphi^R_1(u,v)+\mu\,\varphi_{2}^{R, \mu}(u,v),\ \ \ \ \ (u,v) \in\,H^1.\]
We have seen that $\varphi^R_1 \in\,C^2_b(H^1\times H^1)$. Hence, due to Lemma \ref{lemma7.5}, we argue that the function $\varphi_{R, \mu}$ belongs to  $C^2_b(H^1\times H^1)$. In particular, we can apply It\^o's formula to $\varphi_{R, \mu}$ and $(u^R_\mu,\sqrt{\mu} \partial_t \,u^R_\mu)$, and we have
\begin{align}\begin{split} \label{sa183}
	&d \varphi_{R, \mu}(u^R_\mu(t),\sqrt{\mu} \partial_t \,u^R_\mu(t))\\[10pt]
	&\hsllp=\mathcal{K}^{R, \mu}\varphi_{R, \mu}(u^R_\mu(t),\sqrt{\mu} \partial_t \,u^R_\mu(t))\,dt+\frac 1{\sqrt{\mu}}\langle \sigma_R(u^R_\mu(t))\partial_tw^Q(t),D_v \varphi_{R, \mu}(u^R_\mu(t),\sqrt{\mu} \partial_t \,u^R_\mu(t))\rangle_{H^1},\end{split}
\end{align}
where $\mathcal{K}^{R, \mu}$ is the  
Kolmogorov operator defined by
\begin{align*}
\mathcal{K}^{R, \mu} \varphi(u,v)&=\frac 1{\sqrt{\mu}}\left(\langle D_u\varphi(u,v),v\rangle_{H^1}+\langle D_v\varphi(u,v),\Delta u+f_R(u)\rangle_{H^1}	\right)\\[10pt]
&\hsllp-\frac 1\mu\langle D_v\varphi(u,v),\mathfrak{g}_R(u)v\rangle_{H^1}+\frac 1{2\mu}\text{Tr}\left(D^2_v\varphi(u,v)[\sigma_R(u)Q][\sigma_R(u)Q]^\star\right).
\end{align*}
In particular, we have
\begin{align*}
\mathcal{K}^{R, \mu}& \varphi_{R, \mu}(u,v)=\frac 1{\sqrt{\mu}}\left(\mathcal{M}^{R, u}\varphi^R_1(u,v)+\langle h,v\rangle_{H^1}\right)+\mathcal{M}^{R, u}\varphi_{2}^{R, \mu}(u,v)+\langle D_u\varphi^R_1(u,v),v\rangle_{H^1}	\\[10pt]
&\hsl+\langle D_v\varphi^R_1(u,v),\Delta u+f_R(u)\rangle_{H^1}+\sqrt{\mu}\left(\langle D_v\varphi_{2}^{R, \mu}(u,v),\Delta u+f_R(u)\rangle_{H^1}+\langle D_u\varphi_{2}^{R, \mu}(u,v),v\rangle_{H^1}\right),
\end{align*}
where $\mathcal{M}^{R, u}$ is the differential operator defined in \eqref{sa101}.

Thanks to \eqref{sa101} we have
\begin{equation}
\label{sa181}
\mathcal{M}^{R,u}\varphi^R_1(u,v)+\langle D_u\varphi(u),v\rangle_{H^1}=0.	
\end{equation}
Moreover, recalling that $\langle D_u\varphi^R_1(u,v),v\rangle_{H^1}=\psi_R(u,v)$, thanks to \eqref{sa122} we have
\begin{equation}
\label{sa182}
\mathcal{M}^{R,u}\varphi_{2}^{R, \mu}(u,v)+\langle D_u\varphi^R_1(u,v),v\rangle_{H^1}=\lambda(\mu)\varphi_{2}^{R, \mu}(u,v)+\langle S_R(u),h\rangle_H.	
\end{equation}
Therefore, since 
\[\langle D_v\varphi^R_1(u,v),k\rangle_{H^1}=\langle \mathfrak{g}_R^{-1}(u)k,h\rangle_H,\]
if we plug \eqref{sa181} and \eqref{sa182} into \eqref{sa183} and integrate both side with respect to time, 
 we obtain
\begin{align}   \begin{split}   \label{sa193}
\langle u^R_\mu(t)&,h\rangle_H=\langle u^\mu_0,h\rangle_H+\int_0^t\langle \mathfrak{g}_R^{-1}(u^R_\mu(s))\Delta u^R_\mu(s)+f_{R, \mathfrak{g}}(u^R_\mu(s))+S_R(u^R_\mu(s)),h\rangle_{H}\,ds \\[10pt]
	&\hsp+\int_0^t\langle \sigma_{R, \mathfrak{g}}(u^R_\mu(s))\partial_tw^Q(s),h\rangle_{H}+\mathfrak{R}_{R,\mu}(t),\end{split}
	\end{align}
where
\begin{align*}
	\mathfrak{R}_{R, \mu}&(t)=\sqrt{\mu}\,\langle \varphi^R_1(u^\mu_0,v^\mu_0)-\varphi^R_1(u^R_\mu(t),\sqrt{\mu}\,\partial_t u^R_\mu(t)),h\rangle_H\\[10pt]
	&\hslp +\mu\langle \varphi_{2}^{R, \mu}(u^\mu_0,v^\mu_0)-\varphi_{2}^{R, \mu}(u^R_\mu(t),\sqrt{\mu}\,\partial_t u^R_\mu(t)),h\rangle_H\\[10pt]
	&\hsl+\int_0^t\left(\lambda(\mu)\varphi_{2}^{R, \mu}(u^R_\mu(s),\sqrt{\mu} \partial_t \,u^R_\mu(s))+\mu\langle D_u\varphi_{2}^{R, \mu}(u^R_\mu(s),\sqrt{\mu}\, \partial_t u^R_\mu(s)),\partial_t u^R_\mu(s)\rangle_{H^1}\right)\,dt\\[10pt]
	&\hslp+\sqrt{\mu}\int_0^t\langle D_v\varphi_{2}^{R, \mu}(u^R_\mu(s),\sqrt{\mu}\, \partial_t u^R_\mu(s)),\Delta u^R_\mu(t)+f_R(u^R_\mu(s))\rangle_{H^1}\,ds\\[10pt]
	&\hsp+\sqrt{\mu}\int_0^t\langle D_v\varphi_{2}^{R, \mu}(u^R_\mu(s),\sqrt{\mu}\, \partial_t u^R_\mu(s)),\sigma_R(u^R_\mu(s))\partial_t w^Q(s)\rangle_H.
\end{align*}

\begin{Lemma}
Assume that
\begin{equation}  \label{sa185}\lim_{\mu\to 0}\, \lambda(\mu)=0,\ \ \ \ \ \ \ \ \ \ \ \lim_{\mu\to 0}\,\frac{\mu^{\frac 12-\delta}}{\lambda(\mu)}=0,\end{equation}
where $\delta \in\,(0,1/2)$ is the constant introduced in \eqref{xsa11}.
Then, under Hypotheses \ref{as1} to \ref{as5}, for every $T>0$ and $(u^\mu_0,v^\mu_0) \in\,\mathcal{H}_2$ satisfying \eqref{finex} and \eqref{xsa11}, we have
\begin{equation} \label{sa190}
\lim_{\mu\to 0}\mathbb{E}\sup_{t \in\,[0,T]}\vert \mathfrak{R}_{R, \mu}(t)\vert =0.	
\end{equation}
	\end{Lemma}
\begin{proof}
Since
\[\vert \varphi^R_1(u,v)|\leq c\,\Vert v\Vert_H\Vert h\Vert_H\]
due to \eqref{finex} we have
\begin{align*}
\sqrt{\mu}\vert &\langle \varphi^R_1(u^\mu_0,v^\mu_0)-\varphi^R_1(u^R_\mu(t),\sqrt{\mu}\,\partial_t u^R_\mu(t)),h\rangle_H\vert\leq c\,\sqrt{\mu}\,\Vert h\Vert_H+c\,\mu\,\Vert \partial_t u^R_\mu(t)\Vert_H\,\Vert h\Vert_H,
	\end{align*}
	so that, from \eqref{sg32}, we obtain
	\begin{equation}
	\label{sa198}
	\lim_{\mu\to 0}\,\sqrt{\mu}\,\mathbb{E}\sup_{t \in\,[0,T]}	\vert \langle \varphi^R_1(u^\mu_0,v^\mu_0)-\varphi^R_1(u^R_\mu(t),\sqrt{\mu}\,\partial_t u^R_\mu(t)),h\rangle_H\vert=0.
	\end{equation}

Next, according to \eqref{sa140}, 
\[\vert \varphi_{2}^{R, \mu}(u,v)\vert\leq c\left(1+\Vert u\Vert_{H^1}^2+\Vert v\Vert_{H^1}^2\right)\Vert h\Vert_H,\]	
so that, thanks to \eqref{xsa11}, \eqref{sg31-bis} and \eqref{sg32}, we get
\begin{equation}
\label{sa199}
\lim_{\mu\to 0}\mu\,\mathbb{E}\sup_{t \in\,[0,T]}	\vert \langle \varphi_{2}^{R, \mu}(u^\mu_0,v^\mu_0)-\varphi_{2}^{R, \mu}(u^R_\mu(t),\sqrt{\mu}\,\partial_t u^R_\mu(t)),h\rangle_H\vert=0,	\end{equation}
and 
thanks to \eqref{sg31-tris},  \eqref{sa29} and \eqref{sa185} we get
\begin{equation}
\label{sa186}
\lim_{\mu\to 0}\int_0^T\lambda(\mu)\,\mathbb{E}\,\vert \varphi_{2}^{R, \mu}(u^R_\mu(t),\sqrt{\mu}\, \partial_t u^R_\mu(t))\vert	\,dt=0.
\end{equation}
Moreover, according to \eqref{sa180}, we have
\begin{align*}
 \mu\vert\langle D_u\varphi_{2}^{R, \mu}&(u^R_\mu(t),\sqrt{\mu}\, \partial_t u^R_\mu(t)),\partial_t u^R_\mu(t)\rangle_{H^1}\vert\\[10pt]
&\leq \frac{c\,\sqrt{\mu}}{\lambda(\mu)}\left(1+\Vert u^R_\mu(t)\Vert_{H^2}^2+\mu\Vert \partial_t u^R_\mu(t)\Vert_{H^1}^2\right)\sqrt{\mu}\Vert \partial_t u^R_\mu(t)\Vert_{H^1}\Vert h\Vert_H\\[10pt]
&\hsl\leq \frac{c\,\sqrt{\mu}}{\lambda(\mu)}\left(1+\Vert u^R_\mu(t)\Vert_{H^2}^2+\mu \Vert u^R_\mu(t)\Vert_{H^2}^2\Vert \partial_t u^R_\mu(t)\Vert_{H^1}^2+\mu^2 \Vert \partial_t u^R_\mu(t)\Vert_{H^1}^4\right)\Vert h\Vert_H.	
\end{align*}
In view of \eqref{sg31-bis}, \eqref{xfine141-final} and \eqref{sa185}, this implies
\begin{equation}
\label{sa187}
\lim_{\mu\to 0}\,\mu\int_0^T	\mathbb{E}\,\vert\langle D_u\varphi_{2}^{R, \mu}(u^R_\mu(t),\sqrt{\mu} \partial_t \,u^R_\mu(t)),\partial_t u^R_\mu(t)\rangle_{H^1}\vert\,dt=0.
\end{equation}
Finally, due to \eqref{sa191}, we have
\begin{align*}
	\vert\langle D_v\varphi_{2}&^{R, \mu}(u^R_\mu(t),\sqrt{\mu} \partial_t \,u^R_\mu(t)),\Delta u^R_\mu(t)+f(u^R_\mu(t))\rangle_{H^1}\vert\\[10pt]
	&\hslp\leq c\,\left(1+\Vert u^R_\mu(t)\Vert_{H^1}+\sqrt{\mu}\,\Vert \partial_t u^R_\mu(t)\Vert_{H^1}\right)\left(\Vert u^R_\mu(t)\Vert_{H^3}+1\right)\Vert h\Vert_H,
\end{align*}
and hence, thanks to \eqref{sg31-tris}, \eqref{sa29} and \eqref{xsa17}, we have
\begin{align*}
\int_0^T	&\mathbb{E}\,\vert\langle D_v\varphi_{2}^{R, \mu}(u^R_\mu(t),\sqrt{\mu}\, \partial_t u^R_\mu(t)),\Delta u^R_\mu(t)+f(u^R_\mu(t))\rangle_{H^1}\vert\,dt\leq c_T\,\left(\frac{\rho_{T,R}(\mu)}{\mu^\delta}\right)^{1/2}.\end{align*}
According to \eqref{xsa19} and \eqref{sa185}, this implies 
\begin{equation}
\label{sa188}
\lim_{\mu\to 0}\int_0^T	\sqrt{\mu}\,\mathbb{E}\,\vert\langle D_v\varphi_{2}^{R, \mu}(u^R_\mu(t),\sqrt{\mu} \,\partial_t u^R_\mu(t)),\Delta u^R_\mu(t)+f_R(u^R_\mu(t))\rangle_{H^1}\vert\,dt=0.
\end{equation}

Therefore, combining together \eqref{sa198}, \eqref{sa199}, \eqref{sa186}, \eqref{sa187} and \eqref{sa188}, we obtain \eqref{sa190}.

\end{proof}

\subsection{Conclusion}

All  subsequences, filtered probability spaces and Wiener processes that we are going to introduce in what follows depend on $R\geq 1$. However, since we keep $R$  fixed in this subsection, for the sake of simplicity of notations we do not  emphasize their dependence on $R$.

In what follows, for every 
\[\varrho<1,\ \ \ \ 1\leq \vartheta<2,\ \ \ \ p<\frac{2}{\vartheta-1},\]
we define
\[\mathcal{X}_{\vartheta, \varrho}^p:=C([0,T];H^\varrho)\cap L^p(0,T;H^\vartheta),\]
and
\[\mathcal{K}_{\vartheta,\varrho}^p:=\left[\mathcal{X}_{\vartheta, \varrho}^p\times L^\infty(0,T;H^\varrho)\right]^2\times C([0,T];U),\]
where $U$ is a Hilbert space containing the reproducing kernel $H_Q$ with Hilbert-Schmidt embedding.

In Proposition \ref{teo-tight} we have seen that the family $\{\mathscr{L}(u^R_\mu)\}_{\mu \in\,(0,1)}$ is tight in $\mathcal{X}_{\vartheta, \varrho}^p$. Due to \eqref{sg32}, this implies that the family
	$\{\mathscr{L}(u^R_\mu, \mu\,\partial_t u^R_\mu) \}_{\mu \in\,(0,1)}$ is tight in  $\mathcal{X}_{\vartheta, \varrho}^p\times L^\infty(0,T;H^\varrho)$.  In particular, thanks to the Skorokhod theorem, for any two sequences $\{\mu^1_k\}_{k \in\,\mathbb{N}}$ and  $\{\mu^2_k\}_{k \in\,\mathbb{N}}$, both converging to zero, there exist two subsequences, still denoted by $\{\mu^1_k\}_{k \in\,\mathbb{N}}$ and  $\{\mu^2_k\}_{k \in\,\mathbb{N}}$, a sequence of random variables
	\[\mathscr{Y}_k:=((\varrho^1_k,\vartheta^1_k),(\varrho^2_k,\vartheta^2_k),\hat{w}_k^Q),\ \ \ \ k \in\,\mathbb{N},\]
	in $\mathcal{K}_{\vartheta,\varrho}^p$ and a random variable $\mathscr{Y}=(\varrho^1,\varrho^2,\hat{w}^Q)$ in $\mathcal{X}_{\vartheta,\varrho}^p\times \mathcal{X}_{\vartheta,\varrho}^p\times C([0,T];U)$, all defined on some probability space $(\hat{\Omega}, \hat{\mathscr{F}}, \hat{\mathbb{P}})$, such that
	\begin{equation}
	\label{sz200}
	\mathscr{L}(\mathscr{Y}_k)=\mathscr{L}((u^R_{\mu^1_{k}}, \sqrt{\mu^1_{k}} \,\partial_t u^R_{\mu^1_k}), (u^R_{\mu^2_{k}}, \sqrt{\mu^2_{k}} \,\partial_t u^R_{\mu^1_k}),\,w^Q),\;\;\; k \in\,\mathbb{N},	
	\end{equation}
and, for $i=1, 2$,
\begin{equation}
\label{sz201-bis}
\lim_{k\to\infty} \Vert \varrho^i_{k}-\varrho^i\Vert_{\mathcal{X}_{\vartheta, \varrho}^p}+	\sqrt{\mu_k^i}\,\Vert \vartheta^i_{k}\Vert_{L^\infty(0,T;H)}+\Vert \hat{w}_k^Q-\hat{w}^Q\Vert_{C([0,T];U)}=0,\;\;\;\hat{\mathbb{P}}-\text{a.s.}	\end{equation}
It is important to stress that, due to  \eqref{sg31-bis}, the sequences $\{\varrho^i_k\}_{k \in\,\mathbb{N}}$ are bounded in the space 
$L^2(\hat{\Omega},C([0,T];H^1)\cap L^2(0,T;H^2))$, for $i=1, 2$, and 
\[\varrho^i \in\,L^2(\hat{\Omega},C([0,T];H^1)\cap L^2(0,T;H^2)),\ \ \ \ \ \ i=1, 2.\]

Next, a filtration $(\hat{\mathscr{F}}_t)_{t\geq 0}$ is introduced in $(\hat{\Omega},\hat{\mathscr{F}},\hat{\mathbb{P}})$, by taking the augmentation of the canonical filtration of the process $(\rho^1, \rho^2,\hat{w}^Q)$, generated by its restrictions  to every interval $[0,t]$. Due to this construction, $\hat{w}^Q$ is a $(\hat{\mathscr{F}}_t)_{t\geq 0}$ Wiener process with covariance $Q^*Q $ (for a proof see \cite[Lemma 4.8]{DHV2016}).

	Now, if we show that $\varrho^1=\varrho^2$, we have that $u^R_\mu$ converges in probability to some 
	\[u^R
\in\,L^p(0,T;H^{\vartheta})\cap C([0,T];H^1)\cap L^2(0,T;H^2),\ \ \ \ \hat{\mathbb{P}}-\text{a.s.}\] Actually, as observed by Gy\"ongy and Krylov in
\cite{gk}, if $E$ is any Polish space equipped with the Borel
$\sigma$-algebra, a sequence $(\xi_n)_{n \in\,\mathbb{N}}$ of $E$-valued random
variables converges in probability if and only if for every pair
of subsequences $(\xi_m)_{m \in\,\mathbb{N}}$ and $(\xi_l)_{l \in\,\mathbb{N}}$ there exists an
$E^2$-valued subsequence $\eta_k:=(\xi_{m(k)},\xi_{l(k)})$
converging weakly to a random variable $\eta$ supported on the
diagonal $\{(h,k) \in\,E^2\ :\ h=k\}$.

In order to show that $\varrho^1=\varrho^2$, we prove that they are both a solution of equation \eqref{lim-eq-R}, which, as proven in Theorem \ref{teo7.2}, has pathwise uniqueness in $L^2(\Omega;C([0,T];H^1)\cap L^2(0,T;H^2))$. Due to \eqref{sz200}, we have that
both $(\varrho^1_k,\vartheta_k^1)$ and $(\varrho^2_k,\vartheta_k^2)$ satisfy equation \eqref{SPDE-R}, with $w^Q$ replaced by $\hat{w}_k^Q$. Therefore, we have
\begin{align}   \begin{split}   \label{sa193}
\langle \varrho^i_k(t)&,h\rangle_H=\langle u^{\mu^i_k}_0,h\rangle_H+\int_0^t\langle \mathfrak{g}_R^{-1}(\varrho^i_k(s))\Delta \varrho^i_k(s)+f_{R, \mathfrak{g}}(\varrho^i_k(s))+S_R(\varrho^i_k(s)),h\rangle_{H}\,ds \\[10pt]
	&\hsp+\int_0^t\langle \sigma_{R, \mathfrak{g}}(\varrho^i_k(s))\,\partial_t \hat{w}_k^Q(s),h\rangle_{H}+\mathfrak{R}^i_{R,k}(t),\end{split}
	\end{align}
where
\begin{align*}
	\mathfrak{R}^i_{R, k}(t)=&\sqrt{\mu^i_k}\,\langle \varphi^R_1(\zeta^i_k)-\varphi^R_1(\zeta^i_k(t)),h\rangle_H+\mu^i_k\,\langle \varphi_{2}^{R, \mu}(\zeta^i_k)-\varphi_{2}^{R, \mu}(\zeta^i_k(t)),h\rangle_H\\[10pt]
	&\hsllp+\int_0^t\left(\lambda(\mu^i_k)\,\varphi_{2}^{R, \mu}(\zeta^i_k(s))+\sqrt{\mu^i_k}\,\langle D_u\varphi_{2}^{R, \mu}(\zeta^i_k(s)),\vartheta^i_k(s)\rangle_{H^1}\right)\,dt\\[10pt]
	&\hslp+\sqrt{\mu^i_k}\int_0^t\langle D_v\varphi_{2}^{R, \mu}(\zeta^i_k(s)),\Delta \varrho^i_k(t)+f_R(\varrho^i_k(s))\rangle_{H^1}\,ds\\[10pt]
	&\hsp+\sqrt{\mu^i_k}\int_0^t\langle D_v\varphi_{2}^{R, \mu}(\zeta^i_k(s)),\sigma_R(\varrho^i_k(s))\,\partial_t \hat{w}_k^Q(s)\rangle_H,
\end{align*}
and
\[\zeta^i_k:= (u^{\mu^i_k}_0,v^{\mu_k^i}_0),\ \ \ \ \ \zeta^i_k(t):=(\varrho^i_k(t),\vartheta^i_k(t)).\]
According to \eqref{sz200} and \eqref{sa190}, we have
\begin{equation}  \label{xsa23-bis}\lim_{k\to\infty}\hat{\mathbb{E}}\sup_{t\in\,[0,T]}\,\vert \mathfrak{R}^i_{R, k}(t)\vert=0,\ \ \ \ \ i=1, 2.\end{equation}

If we assume that $h \in\,H^1$, we have
\[\int_0^t\langle \mathfrak{g}_R^{-1}(\varrho^i_k(s))\Delta \varrho^i_k(s),h\rangle_H\,ds=-\int_0^t\langle  \varrho^i_k(s),\,[\mathfrak{g}_R^{-1}(\varrho^i_k(s))]^th\rangle_{H^1}\,ds.\]
Thus, thanks to \eqref{sz201-bis}, we have
\begin{equation}
	\label{xsa20}
	\lim_{k\to\infty}\int_0^t\langle \mathfrak{g}_R^{-1}(\varrho^i_k(s))\Delta \varrho^i_k(s),h\rangle_H\,ds=\int_0^t\langle \mathfrak{g}_R^{-1}(\varrho^i(s))\Delta \varrho^i(s),h\rangle_H\,ds,\ \ \ \ \hat{\mathbb{P}}-\text{a.s.}\end{equation}
Moreover, due to \eqref{sa171}, we have that $f_{R, 
\mathfrak{g}}$ is locally Lipschitz continuous from $H^1$ into $H$, with  \[\int_0^t\left\vert \langle f_{R,\mathfrak{g}}(\varrho^i_k(s))-f_{R, \mathfrak{g}}(\varrho^i(s)),h\rangle_H\right\vert\,ds\leq c\,\Vert \varrho^i_k-\varrho^i\Vert_{L^2(0,T;H^1)}\left(1+\Vert \varrho^i\Vert_{L^2(0,T;H^1)}\right)\,\Vert h\Vert_H,\]
and \eqref{sz201-bis} implies
\begin{equation}
\label{xsa21}
\lim_{k\to\infty}	\int_0^t \langle f_{R, \mathfrak{g}}(\varrho^i_k(s)),h\rangle_H\,ds=\int_0^t \langle f_{R, \mathfrak{g}}(\varrho^i(s)),h\rangle_H\,ds,\ \ \ \ \hat{\mathbb{P}}-\text{a.s.}\end{equation}
Finally, due to \eqref{sa170-R}, we have
\begin{align*}
	\int_0^t&\left\vert \langle S_R(\varrho^i_k(s))-S_R(\varrho^i(s)),h\rangle_H\right\vert\,ds\\[10pt]
&\hslp	\leq c\,\Vert \varrho^i_k-\varrho^i\Vert_{L^2(0,T;H^1)}\left(1+\Vert \varrho^i_k\Vert_{L^4(0,T;H^1)}^2+\Vert \varrho^i\Vert_{L^4(0,T;H^1)}^2\right)\,\Vert h\Vert_H,
\end{align*}
and due to \eqref{sz201-bis} we conclude
that
\begin{equation}
\label{xsa22}	
\lim_{k\to\infty}	\int_0^t \langle S_R(\varrho^i_k(s)),h\rangle_H\,ds=\int_0^t \langle S_R(\varrho^i(s)),h\rangle_H\,ds,\ \ \ \ \hat{\mathbb{P}}-\text{a.s.}
\end{equation}

Now, for $i=1, 2$ and $t \in\,[0,T]$ we define
\[M_R^i(t):=\langle\varrho^i(t),h\rangle_H-\langle u_0,h\rangle_H-\int_0^t\langle \mathfrak{g}_R^{-1}(\varrho^i(s))\Delta \varrho^i(s)+f_{R, \mathfrak{g}}(\varrho^i(s))+S_R(\varrho^i(s)),h\rangle_{H}\,ds. \]
By proceeding as in the proof of \cite[Lemma 4.9]{DHV2016}, in view of \eqref{xsa23-bis}, \eqref{xsa20}, \eqref{xsa21} and \eqref{xsa22}, thanks to \eqref{fx2} we have for every $t \in\,[0,T]$
\[\left\langle M_R^i-\int_0^\cdot \sigma_{R, \mathfrak{g}}(\varrho^i(s))\,\partial_t \hat{w}^Q(s)\right\rangle_t=0,\]
where $\langle\cdot\rangle_t$ is the quadratic variation process. This implies that both $\varrho^1$ and $\varrho^2$ satisfy equation \eqref{lim-eq-R} and, as we have explained above, this allows to conclude the proof of Theorem \ref{teo3.4-R}.

\section{Proof of Theorem \ref{teo3.4}}

As a consequence of Theorem \ref{teo3.4-R} and Proposition \ref{teo7.2},  for every $R\geq 1$ and $u_0 \in\,H^1$ there exists a unique solution $u^R \,\in\,L^2(\Omega;C([0,T];H^1)\cap L^2(0,T;H^2))$ for equation \eqref{lim-eq-R}.
We start by showing that this allows to prove that equation \eqref{lim-eq} is well-posed in $C([0,T];H^1)\cap L^2(0,T;H^2)$, as well. 

\begin{Proposition}
Under Hypotheses \ref{as1} to \ref{as5}, for every $u_0 \in\,H^1$ equation \eqref{lim-eq} admits a unique solution $u \in\,L^2(\Omega;C([0,T];H^1)\cap L^2(0,T;H^2))$.	
\end{Proposition}

\begin{proof} If $u^R$ is the unique solution of equation \eqref{lim-eq-R} with initial condition $u_0$, we have
\begin{align*}
\frac 12 &d\Vert u^R(t)\Vert^2_{H^1}=\langle u^R(t),\mathfrak{g}_R^{-1}(u^R(t))\Delta u^R(t)+f_{R, \mathfrak{g}}	(u^R(t))+S_R(u^R(t))\rangle_{H^1}\,dt\\[10pt]
&+\frac 12\,\Vert \sigma_{R,\mathfrak{g}}(u^R(t))\Vert^2_{\mathcal{L}_2(H_Q,H^1)}\,dt+\langle u^R(t),\sigma_{R,\mathfrak{g}}(u^R(t))dw^Q(t)\rangle_{H^1}.
\end{align*}
Due to \eqref{xsa51} we have
\[\vert\langle u,S_R(u)\rangle_{H^1}\vert\leq \Vert u\Vert_{H^2}\Vert S_R(u)\Vert_H\leq c\,\Vert  u\Vert_{H^2}\left(1+ \Vert u\Vert_{H^1}\right),\ \ \ \ u \in\,H^2.\]
Hence, thanks to \eqref{xfine160} we get
\begin{align*}
\frac 12 &d\Vert u^R(t)\Vert^2_{H^1}\leq -\tilde{\gamma_0}\,\Vert u^R(t)\Vert_{H^2}^2\,dt+c\,\left(1+\Vert u^R(t)\Vert_{H^1}^2\right)\,dt+\\[10pt]
&\hslp+\Vert  u^R(t)\Vert_{H^2}\left(1+ \Vert u^R(t)\Vert_{H^1}\right)\,dt+\langle u^R(t),\sigma_{R,\mathfrak{g}}(u^R(t))dw^Q(t)\rangle_{H^1}\\[10pt]
&\leq -\frac{\tilde{\gamma_0}}2\,\Vert u^R(t)\Vert_{H^2}^2\,dt+c\,\left(1+\Vert u^R(t)\Vert_{H^1}^2\right)\,dt+\langle u^R(t),\sigma_{R,\mathfrak{g}}(u^R(t))dw^Q(t)\rangle_{H^1}.
\end{align*}
After we first integrate both sides with respect to $t$, then take the supremum and finally take the expectation, we get
\begin{align*}
	\mathbb{E}\sup_{s \in\,[0,t]}\Vert u^R(s)\Vert_{H^1}^2+\tilde{\gamma_0}\,\int_0^t \mathbb{E}\Vert u^R(s)\Vert_{H^2}^2\,ds\leq \Vert u_0\Vert_{H^1}^2+c_T+c_T\int_0^t \mathbb{E}\Vert u^R(s)\Vert_{H^1}^2\,ds,
\end{align*}
and Gronwall's lemma gives
\begin{equation}
\label{xfine170}
\sup_{R\geq 1}\left(	\mathbb{E}\sup_{s \in\,[0,T]}\Vert u^R(s)\Vert_{H^1}^2+\int_0^T \mathbb{E}\Vert u^R(s)\Vert_{H^2}^2\,ds\right)\leq c_T.
\end{equation}

Now, we introduce the stopping time
\[\tau^R:=\inf \left\{ t \in\,[0,T]\,:\,\Vert u^R(t)\Vert_{H^{\bar{r}}}\geq R\right\},\]
with the  convention that $\inf \emptyset =+\infty$. Since 
the family of stopping times $\{\tau^R\}_{R\geq 1}$ is non-decreasing,  we can define
\[\tau:=\lim_{R\to\infty}\tau^R.\]
According to \eqref{xfine170} we have 
\begin{equation}
\label{xsa23}
\mathbb{P}\left(\tau=+\infty \right)=1.	
\end{equation}
Hence, if we define $\Omega^\prime:=\{\tau=+\infty\}$, we have that for every $t \in\,[0,T]$ and $\omega \in\,\Omega^\prime$ there exists $R\geq 1$ such that $t\leq \tau^R(\omega)$ and we define 
\[u(t)(\omega)=u^R(t)(\omega).\]
Due to \eqref{xfine125}, by a uniqueness argument we have that
\[R_1\leq R_2\Longrightarrow u^{R_2}(t)=u^{R_1}(t),\ \ \ \ t\leq \tau^{R_1}.\]
This implies that the definition of $u$ is a good definition and $u$ is a solution of equation \eqref{lim-eq}. Moreover, in view of \eqref{xfine170}, we have that  $u$ belongs to $L^2(\Omega;C([0,T];H^1)\cap L^2(0,T;H^2))$ and, in particular,  due to Proposition \ref{teo7.2} is the unique solution of equation \eqref{lim-eq}.
	
\end{proof}

Now, we are finally ready to conclude the proof of Theorem \ref{teo3.4}. In the same spirit of what we have done above for $u^R$, for every  $\mu \in\,(0,1)$ we define
\[\tau^{R}_{\mu}:=\inf \left\{ t \in\,[0,T]\,:\,\Vert u^R_\mu(t)\Vert_{H^{\bar{r}}}\geq R\right\},\]
again with the  convention that $\inf \emptyset =+\infty$. 
For every $\eta>0$ we have
\begin{align*}
\mathbb{P}&\left(\Vert u_\mu-u\Vert_{\mathcal{X}^p_{\vartheta, \varrho}}	
>\eta\right)\leq \mathbb{P}\left(\Vert u_\mu-u^R_\mu\Vert_{\mathcal{X}^p_{\vartheta, \varrho}}>\eta/3\,;\,\tau^R_\mu<\infty\right)\\[10pt]
&\hslp	+  \mathbb{P}\left(\Vert u^R_\mu-u^R\Vert_{\mathcal{X}^p_{\vartheta, \varrho}}>\eta/3\right)+ \mathbb{P}\left(\Vert u^R-u\Vert_{\mathcal{X}^p_{\vartheta, \varrho}}>\eta/3\,;\,\tau^R<\infty\right)\\[10pt]
&\hsp\leq \mathbb{P}(\tau^R_\mu<\infty)+\mathbb{P}(\tau^R<\infty)+\mathbb{P}\left(\Vert u^R_\mu-u^R\Vert_{\mathcal{X}^p_{\vartheta, \varrho}}>\eta/3\right).
\end{align*}
Therefore, since
\begin{align*}
\mathbb{P}&\left(\tau^R_{\mu}<\infty\right)\leq \mathbb{P}\left(\sup_{t \in\,[0,T]}\,\Vert u^R_\mu(t)-u^R(t)\Vert_{H^{\bar{r}}}\geq R/2\right)+\mathbb{P}\left(\sup_{t \in\,[0,T]}\,\Vert u^R(t)\Vert_{H^{\bar{r}}}\geq R/2\right)\\[10pt]
&\hslp\hsp\leq \mathbb{P}\left(\sup_{t \in\,[0,T]}\,\Vert u^R_\mu(t)-u^R(t)\Vert_{H^{\bar{r}}}\geq R/2\right)+\mathbb{P}\left (\tau^{R/2}<\infty\right),
\end{align*} 
we have
\begin{align*}
\mathbb{P}&\left(\Vert u_\mu-u\Vert_{\mathcal{X}^p_{\vartheta, \varrho}}	>\eta\right)\\[10pt]
&\hsllp\leq 2\,\mathbb{P}(\tau^{R/2}<\infty)+\mathbb{P}\left(\Vert u^R_\mu-u^R\Vert_{\mathcal{X}^p_{\vartheta, \bar{r}}}\geq R/2\right)+	\mathbb{P}\left(\Vert u^R_\mu-u^R\Vert_{\mathcal{X}^p_{\vartheta, \varrho}}	
>\eta/3\right).
\end{align*}
Thanks to \eqref{xsa23} for every $\e>0$ we can fix $R_\e\geq 1$ such that
\[\mathbb{P}(\tau^{R_\e/2}<\infty)<\frac \epsilon 2,\]
and, as a consequence of Theorem \ref{teo3.4-R}, we get
\begin{align*}
\liminf_{\mu\to 0}&\,	\mathbb{P}\left(\Vert u_\mu-u\Vert_{\mathcal{X}^p_{\vartheta, \varrho}}	>\eta\right)\\[10pt]
&\leq \epsilon+\lim_{\mu\to 0}\mathbb{P}\left(\Vert u^{R_\epsilon}_\mu-u^{R_\epsilon}\Vert_{\mathcal{X}^p_{\vartheta, \bar{r}}}\geq R_\epsilon/2\right)+	\lim_{\mu\to 0}\mathbb{P}\left(\Vert u^{R_\epsilon}_\mu-u^{R_\epsilon}\Vert_{\mathcal{X}^p_{\vartheta, \varrho}}	
>\eta/3\right)=\epsilon.
\end{align*}
Due to the arbitrariness of $\epsilon>0$, we conclude that \eqref{sa161} holds and Theorem \ref{teo3.4} follows.

\end{document}